\documentclass{amsart}
\usepackage{amsfonts,amscd,amsthm,amsgen,amsmath,amssymb}
\usepackage[all]{xy}
\usepackage{epsfig,color}
\usepackage[vcentermath]{youngtab}
\newtheorem{theorem}{Theorem}[section]
\newtheorem{lemma}[theorem]{Lemma}
\newtheorem{corollary}[theorem]{Corollary}
\newtheorem{proposition}[theorem]{Proposition}

\theoremstyle{definition}

\newtheorem{example}[theorem]{Example}

\theoremstyle{remark}
\newtheorem{remark}[theorem]{Remark}

\numberwithin{equation}{section}

\newcommand{\largewedge}{\mbox{\Large $\wedge$}}

\begin{document}

\setlength\parskip{0.5em plus 0.1em minus 0.2em}

\title[The mod 2 Seiberg--Witten invariants of spin structures]{The mod 2 Seiberg--Witten invariants of spin structures and spin families}
\author{David Baraglia}

\address{School of Mathematical Sciences, Adelaide University, Adelaide SA 5005, Australia}

\email{david.baraglia@adelaide.edu.au}


\date{\today}

\begin{abstract}
We completely determine the mod $2$ Seiberg--Witten invariants for any spin structure on any closed, oriented, smooth $4$-manifold $X$. Our computation confirms the validity of the simple type conjecture mod $2$ for spin structures. Our proof also works for families of spin $4$-manifolds and thus computes the mod $2$ Seiberg--Witten invariants for spin families.

The proof of our main result uses $Pin(2)$-symmetry to define an enhancement of the mod $2$ Seiberg--Witten invariants. We prove a connected sum formula for the enhanced invariant using localisation in equivariant cohomology. Unlike the usual Seiberg--Witten invariant, the enhanced invariant does not vanish on taking connected sums and by exploiting this property, we are able to compute the enhanced invariant.

\end{abstract}

\maketitle


\section{Introduction}

The Seiberg--Witten invariant is an invariant of smooth $4$-manifolds which has proven to be very effective in distinguishing smooth structures on homeomorphic $4$-manifolds. In contrast it has been observed that in a number of cases, the mod $2$ Seiberg--Witten invariant for spin-structures is a topological invariant \cite{ms,rs,bau,li} and a similar ``rigidity" phenomenon occurs for the families Seiberg--Witten invariants \cite{kkn}. As we will show, when $b_+(X) \ge 3$ this rigidity will imply that the mod $2$ Seiberg--Witten invariants of spin structures can not be used to distinguish smooth structures. When $b_+(X) = 1$ or $2$, we will show that the mod $2$ Seiberg--Witten invariant is ``essentially'' a topological invariant in a sense that will be clarified below. Nevertheless, there are many good reasons for wanting to compute these invariants. Their vanishing can be used to obstruct the existence of symplectic structures. Their non-vanishing can be used to obstruct the existence of positive scalar curvature metrics and gives a lower bound for the genus of embedded surfaces through the adjunction inequality. Furthermore, using various operations such as blowup, rational blowdown \cite{fs1} and Fintushel--Stern knot surgery \cite{fs2}, we can also calculate the mod $2$ Seiberg--Witten invariants for various spin$^c$-structures that are not spin.

In this paper we completely determine the mod $2$ Seiberg--Witten invariants for any spin structure and also the mod $2$ families Seiberg--Witten invariants for spin families (under some mild assumptions on the family). Some consequences of these results and some related results are examined.

\subsection{Main results}

Let $X$ be a closed, oriented, smooth $4$-manifold with $b_+(X)>1$ and $\mathfrak{s}$ a spin$^c$-structure. The mod 2 Seiberg--Witten invariant is a homomorphism
\[
SW_{X , \mathfrak{s}} : \mathbb{A}(X) \to \mathbb{Z}_2
\]
where $\mathbb{A}(X) = \largewedge^* H^1(X;\mathbb{Z})^*  \otimes_{\mathbb{Z}} \mathbb{Z}[x]$ is the tensor product of the exterior algebra on $H^1(X ; \mathbb{Z})^* = Hom( H^1(X ; \mathbb{Z}) , \mathbb{Z})$ with the polynomial ring $\mathbb{Z}[x] \cong H^*_{S^1}(pt ; \mathbb{Z})$ on a single generator $x$ (see eg. \cite{os}). The Seiberg--Witten invariant is also defined when $b_+(X)=1$ but depends in addition on the choice of a chamber. For our purposes it is convenient to recast this in the following form. Let $Pic^{\mathfrak{s}}(X)$ denote the space of gauge equivalence classes of spin$^c$-connections with curvature equal to a fixed $2$-form representing $-2\pi i c(\mathfrak{s})$. This is a torsor over $Pic(X)$, the group of flat unitary line bundles on $X$ and hence there are canonical isomorphisms
\[
H^*( Pic^{\mathfrak{s}}(X) ; \mathbb{Z}_2 ) \cong H^*( Pic(X) ; \mathbb{Z}_2 ) \cong {\largewedge}^* H^1(X ; \mathbb{Z})^* \otimes_{\mathbb{Z}} \mathbb{Z}_2.
\]
So $SW_{X,\mathfrak{s}}$ can be viewed as a map $H^*( Pic^{\mathfrak{s}}(X) ; \mathbb{Z}_2) \otimes_{\mathbb{Z}_2} H^*_{S^1}(pt ; \mathbb{Z}_2) \to \mathbb{Z}_2$, or as a map $H^*_{S^1}(pt ; \mathbb{Z}_2) \to H^*( Pic^\mathfrak{s}(X) ; \mathbb{Z}_2)^*$. By Poinca\'re duality, $H^*( Pic^\mathfrak{s}(X)  ; \mathbb{Z}_2)^* \cong H^*( Pic^\mathfrak{s}(X) ; \mathbb{Z}_2)$, so $SW_{X,\mathfrak{s}}$ is equivalent to a map
\[
SW_{X,\mathfrak{s}} : H^*_{S^1}(pt ; \mathbb{Z}_2) \to H^*( Pic^{\mathfrak{s}}(X) ; \mathbb{Z}_2).
\]
This map has degree $-d(X,\mathfrak{s})$, where we set
\[
d(X,\mathfrak{s}) =  \frac{ c(\mathfrak{s})^2 - \sigma(X)  }{4} - b_+(X) - 1.
\]
This is the form of the Seiberg--Witten invariants that most naturally emerges from the Bauer--Furuta invariant and it is the form that we will use throughout this paper. Since $H^*_{S^1}(pt ; \mathbb{Z}_2) \cong \mathbb{Z}_2[x]$, $SW_{X,\mathfrak{s}}$ is determined by the classes
\[
SW_{X,\mathfrak{s}}( x^m ) \in H^{2m - d(X,\mathfrak{s})}( Pic^{\mathfrak{s}}(X)  ; \mathbb{Z}_2)
\]
where $m \ge 0$.

Let $D_\mathfrak{s} \to Pic^{\mathfrak{s}}(X)$ denote the families index of the family of spin$^c$ Dirac operators parametrised by $Pic^{\mathfrak{s}}(X)$. Let $s_j(D_\mathfrak{s}) \in H^{2j}( Pic^{\mathfrak{s}}(X) ; \mathbb{Z})$ denote the $j$-th Segre class of $D_\mathfrak{s}$ (recall that the total Segre class $s(V) = 1 + s_1(V) + s_2(V) + \cdots $ of a virtual bundle is defined by $c(V)s(V) = 1$, where $c(V) = 1 + c_1(V) + c_2(V) + \cdots $ is the total Chern class). From the families index theorem (see Section \ref{sec:swspin} for details), it follows that $s_j(D_\mathfrak{s}) = 0$ for odd $j$ and
\[
s_{2j}(D_\mathfrak{s}) = \frac{1}{j!} s_2(D_\mathfrak{s}),
\]
where
\begin{equation}\label{equ:seg2}
s_2(D_\mathfrak{s}) = \sum_{i_1 < i_2 < i_3 < i_4} \langle y_{i_1} y_{i_2} y_{i_3} y_{i_4} , [X] \rangle x_{i_1} x_{i_2} x_{i_3} x_{i_4}.
\end{equation}
Here $y_1, \dots y_{b_1(X)}$ is a basis for $H^1(X ; \mathbb{Z})$ and $x_1 , \dots , x_{b_1(X)}$ is a corresponding dual basis for $H^1( Pic^\mathfrak{s}(X) ; \mathbb{Z}) \cong H^1(X ; \mathbb{Z})^*$.

To describe the Seiberg--Witten invariant in the cases $b_+(X) = 1$ or $2$, we need to introduce one more concept, namely functions 
\[
q_{\mathfrak{s}}^k : H^1(X ; \mathbb{Z})^k \to \mathbb{Z}_2
\]
for $k=2,3$ defined as follows. Let $a_1, \dots , a_k \in H^1(X ; \mathbb{Z})$. Each class $a_i$ is equivalent to specifying a homotopy class of map $X \to S^1$, hence we get a map $f : X \to T^k$, where $T^k$ is a $k$-dimensional torus. The level set $C = f^{-1}(t)$ of a regular value of $f$ is a normally framed submanifold of $X$ of dimension $4-k$. The framing allows us to define the restriction $\mathfrak{s}|_C$ of $\mathfrak{s}$ to $C$ and we define $q^k_{\mathfrak{s}}(a_1 , \dots , a_k)$ to the be mod $2$ index of $\mathfrak{s}|_{C}$. Cobordism invariance of the index implies that $q^k_{\mathfrak{s}}$ is a well-defined function (see Section \ref{sec:b+12} for further details).

Our first main result is a complete formula for $SW_{X,\mathfrak{s}}$ for spin structures. Note that in the case $b_+(X) = 1$, the mod $2$ Seiberg--Witten invariant of a spin structure does not depend on the choice of chamber (Lemma \ref{lem:wcf1}).

\begin{theorem}\label{thm:thm1}
Let $X$ be a compact, oriented, smooth $4$-manifold with $b_+(X) > 0$ and let $\mathfrak{s}$ be a spin-structure on $X$. Then $SW_{X,\mathfrak{s}}(x^m) = 0$ for all $m > 0$. Furthermore, we have:
\begin{itemize}
\item[(1)]{if $b_+(X) = 1$, then for any $a,b \in H^1(X ; \mathbb{Z}) \cong H_1(Pic^{\mathfrak{s}}(X) ; \mathbb{Z})$,
\[
\langle SW_{X,\mathfrak{s}}(1) , a \smallsmile b \rangle = q^2_{\mathfrak{s}}( a , b) \; ({\rm mod} \; 2),
\]
}
\item[(2)]{if $b_+(X) = 2$, then for any $a,b,c \in H^1(X ; \mathbb{Z}) \cong H_1(Pic^{\mathfrak{s}}(X) ; \mathbb{Z})$,
\[
\langle SW_{X,\mathfrak{s}}(1) , a \smallsmile b \smallsmile c \rangle = q^3_{\mathfrak{s}}( a , b , c) \; ({\rm mod} \; 2),
\]
}
\item[(3)]{if $b_+(X) = 3$, then 
\[
SW_{X,\mathfrak{s}}( 1 ) = s_{2+\sigma(X)/8}(D_\mathfrak{s}) \; ({\rm mod} \; 2)
\]
where we set $s_l(D_\mathfrak{s}) = 0$ if $l < 0$,}
\item[(4)]{if $b_+(X) > 3$, then $SW_{X,\mathfrak{s}}(1) = 0$.}
\end{itemize}

\end{theorem}

We note here that the Seiberg--Witten invariant $SW_{X , \mathfrak{s}}$ of a spin structure $\mathfrak{s}$ depends only on the underlying spin$^c$-structure associated to $\mathfrak{s}$. We will say that two spin structures are {\em equivalent} if they have the same underlying spin$^c$-structure. 

Recall that if $X$ is a compact, oriented, smooth spin $4$-manifold with $b_+(X) = 3$, then $\sigma(X) = 0$ or $-16$ \cite{fur}. When $\sigma(X) = -16$, Theorem \ref{thm:thm1} gives $SW_{X , \mathfrak{s}}(1) = 1 \; ({\rm mod} \; 2)$. The $b_1(X) = 0$ case of this result was proven by Morgan and Szab\'o \cite{ms}. When $\sigma(X) = 0$, Theorem \ref{thm:thm1} gives $SW_{X , \mathfrak{s}}(1) = s_2(D_\mathfrak{s}) \; ({\rm mod} \; 2)$. The $b_1(X) = 4$ case of this result was proven by Ruberman and Strle \cite{rs}. Other special cases of Theorem \ref{thm:thm1} were proven by Bauer \cite{bau} and Li \cite{li}. The cases $b_+(X) = 1,2$ appear to be new.

\begin{example}\label{ex:kod0}
We give examples illustrating cases (1)-(3) of Theorem \ref{thm:thm1}. Interestingly these examples are all compact complex surfaces of Kodaira dimension $0$.

1. Let $X = K3$ be a $K3$-surface, then $b_+(X) = 3$, $\sigma(X) = -16$ and $b_1(X) = 0$. Hence by case (3) of Theorem \ref{thm:thm1}, $SW_{K3 , \mathfrak{s}}(1) = 1 \; ({\rm mod} \; 2)$ for the unique spin structure $\mathfrak{s}$.

2. Let $X = T^4$ be a $4$-torus, then $b_+(X) = 3$, $\sigma(X) = 0$ and $b_1(X) = 4$. Hence by case (3) of Theorem \ref{thm:thm1}, $SW_{T^4,\mathfrak{s}}(1) = s_2(D_\mathfrak{s}) \; ({\rm mod} \; 2)$ for the unique equivalence class of spin structure $\mathfrak{s}$. From Equation \ref{equ:seg2}, it is easily seen that $s_2(D_\mathfrak{s})$ is the non-trivial element of $H^4( Pic^{\mathfrak{s}}(X) ; \mathbb{Z}_2) \cong \mathbb{Z}_2$.
 
3. Let $X$ be a primary Kodaira surface with $H_1(X ; \mathbb{Z})$ torsion-free. Then $X$ is diffeomorphic to $S^1 \times Y$, where $Y \to C$ is the total space of a degree $1$ circle bundle over an elliptic curve $C$. We have $b_+(X) = 2$, $b_1(X) = 3$ and $X$ has a unique equivalence class $\mathfrak{s}$ of spin structure. Since $C$ is an elliptic curve, $TC$ has a translation invariant trivialisation. Then using a connection to split the tangent bundle of $Y$ into horizontal and vertical components, we obtain a trivialisation of $TY$ and hence a spin structure $\mathfrak{s}_Y$. Taking the product of $\mathfrak{s}_Y$ with any spin structure on $S^1$ defines a spin structure $\mathfrak{s}$ on $X$. Let $a,b,c$ denote generators of $H^1(X ; \mathbb{Z}) \cong \mathbb{Z}^3$. Then $a \smallsmile b \smallsmile c$ is Poincar\'e dual to a fibre $F$ of $Y \to C$. By construction, the normal framing on $F$ is translation invariant under the circle action on $F$ and it follows that $\mathfrak{s}|_F$ is the unique translation invariant spin structure on the circle, which has index $1$. Then by case (2) of Theorem \ref{thm:thm1}, $SW_{X,\mathfrak{s}}(1)$ is the non-trivial element of $H^3( Pic^{\mathfrak{s}}(X) ; \mathbb{Z}_2) \cong \mathbb{Z}_2$.

4. Let $X$ be a hyperelliptic surface constructed as follows. Let $\Lambda$ be the lattice in $\mathbb{C}^2$ with basis $e_1 = (1,0)$, $e_2 = (i,0)$, $e_3 = (0,1)$, $e_4 = (0,\omega)$, where $\omega = e^{\pi i/3}$. Let $T$ be the $4$-torus $T = \mathbb{C}/\Lambda$. Then $\mathbb{Z}_6 = \langle g \rangle$ acts freely on $T$ by $g(z_1, z_2) = (z_1 + 1/6 , \omega z_2)$ and we set $X = T/\mathbb{Z}_6$. Then $b_+(X) = 1$, $b_1(X) = 2$ and $X$ has a unique equivalence class $\mathfrak{s}$ of spin structure. Let $C$ be the elliptic curve given by the quotient of $\mathbb{C}$ by the lattice spanned by $1/6$ and $i$ and let $\pi : X \to C$ be the map induced by the projection $(z_1,z_2) \mapsto z_1$. This gives $X$ the structure of an elliptic fibre bundle over $C$. Let $a,b$ denote generators of $H^1(X ; \mathbb{Z}) \cong \mathbb{Z}^2$. Then $a \smallsmile b$ is Poincar\'e dual to a fibre $F$ of $\pi : X \to C$. Using a similar argument to the previous case, we find that $\mathfrak{s}|_F$ is the unique translation invariant spin structure on the $2$-torus, which has index $1$. Then by case (1) of Theorem \ref{thm:thm1}, $SW_{X,\mathfrak{s}}(1)$ is the non-trivial element of $H^2( Pic^{\mathfrak{s}}(X) ; \mathbb{Z}_2) \cong \mathbb{Z}_2$.

\end{example}

Theorem \ref{thm:thm1} shows that when $b_+(X) \ge 3$, $SW_{X,\mathfrak{s}}(x^m)$ can be expressed in terms of the Segre class $s_{2}(D_{\mathfrak{s}})$, which in light of Equation (\ref{equ:seg2}) is a topological invariant. So we obtain:

\begin{corollary}
For $b_+(X) \ge 3$, the mod $2$ Seiberg--Witten invariant of a spin structure $\mathfrak{s}$ on a compact, oriented, smooth $4$-manifold $X$ depends only on the underlying topology of $X$.
\end{corollary}

Strictly speaking, to say that $SW_{X,\mathfrak{s}}$ is a topological invariant requires a notion of spin structure on topological $4$-manifolds. For details on how this is done, see for example \cite[\textsection 4.2]{bk}.

When $b_+(X) = 1$ or $2$, it is not immediately clear from Theorem \ref{thm:thm1} whether $SW_{X, \mathfrak{s}}$ depends on the smooth structure of $X$. However we can show that $SW_{X , \mathfrak{s}}$ is a topological invariant up to certain identifications, as described in the theorem below:

\begin{theorem}\label{thm:almosttop}
Let $X_1, X_2$ be two compact, oriented, smooth spin $4$-manifolds which are homeomorphic to each other and with $b_+(X_1) = b_+(X_2) = 1$ or $2$. Let $S_{X_1}, S_{X_2}$ denote the set of spin structures on $X_1$ and $X_2$. Then there exists a bijection $\varphi : S_{X_1} \to S_{X_2}$ and an isomorphism $\psi : H^1(X_1 ; \mathbb{Z}) \to H^1(X_2 ; \mathbb{Z})$ inducing an isomorphism $\psi : H^*(Pic^{\mathfrak{s}}(X_1) ) \to H^*( Pic^{\varphi({\mathfrak{s})}}(X_2))$ for any $\mathfrak{s} \in S_{X_1}$, such that
\[
\psi( SW_{X_1 , \mathfrak{s}}) = SW_{X_2 , \varphi(\mathfrak{s})}
\]
for all $\mathfrak{s} \in S_{X_1}$.
\end{theorem}

We do not know whether the maps $\varphi, \psi$ are induced by a homeomorphism $X_1 \to X_2$, however they are induced by a diffeomorphism $X_1 \# k(S^2 \times S^2) \to X_2 \# k(S^2 \times S^2)$ for some $k \ge 0$ (see Section \ref{sec:b+12}).

For $4$-manifolds with $b_1(X) = 0$, there is only one possible value of $m$ for which $SW_{X,\mathfrak{s}}(x^m)$ can be non-zero, which is $m = d(X,\mathfrak{s})/2$, provided $d(X,\mathfrak{s})$ is even and non-negative. When $b_1(X) = 0$, $d(X,\mathfrak{s})$ is the dimension of the moduli space of solutions to the Seiberg--Witten equations, so $m$ is half the dimension. In all cases where the Seiberg--Witten invariants have been computed and $b_+(X)>1$, we have $SW_{X,\mathfrak{s}}( x^m ) = 0$ unless $m=0$, that is, unless the dimension of the moduli space is zero. This has become known as the {\em simple type conjecture}. One can also formulate a simple type conjecture for $4$-manifolds with $b_1(X) > 0$. In this case one usually considers only the ``pure" Seiberg--Witten invariant
\[
SW(X , \mathfrak{s}) = \int_{Pic^{\mathfrak{s}}(X)} SW_{X,\mathfrak{s}}( x^m ) \in \mathbb{Z},
\]
where $2m = d(X,\mathfrak{s}) + b_1(X)$ is the dimension of the moduli space. The simple type conjecture then states that $SW(X , \mathfrak{s}) = 0$ unless $m=0$. From Theorem \ref{thm:thm1}, we immediately see that the mod $2$ simple type conjecture is true for spin structures:

\begin{corollary}\label{cor:simplet}
Let $X$ be a compact, oriented, smooth $4$-manifold with $b_+(X) > 0$ and let $\mathfrak{s}$ be a spin structure on $X$. Then $SW_{X,\mathfrak{s}}(x^m) = 0 \; ({\rm mod} \; 2)$, unless $m = 0$. In particular the mod $2$ simple type conjecture holds for spin structures.
\end{corollary}

This corollary was previously known in some special cases. When $b_1(X) = 0$ and $b_+(X) = 3 \; ({\rm mod} \; 4)$ the result follows from \cite[Corollary 3.6, Theorem 3.7]{bf}. The paper \cite{kny} also establishes the mod $2$ simple type conjecture (for all spin$^c$-structures) under more general assumptions, although it does not cover all cases handled by Corollary \ref{cor:simplet}.

It is interesting to consider $SW_{X , \mathfrak{s}}(1)$ as a function of the spin structure $\mathfrak{s}$. Recall that the set of spin structures is a torsor for $H^1(X ; \mathbb{Z}_2)$. If $\mathfrak{s}$ is a spin structure and $A \in H^1(X ; \mathbb{Z}_2)$, then we denote the torsor structure by $(A , \mathfrak{s}) \mapsto A \otimes \mathfrak{s}$.

The following theorem shows that $SW_{X,\mathfrak{s}}(1)$ is quadratic, linear or constant as a function of $\mathfrak{s}$ according to whether $b_+(X)$ is $1$, $2$ or $3$.

\begin{theorem}\label{thm:thm2}
Let $X$ be a compact, oriented, smooth spin $4$-manifold with $b_+(X) > 0$. 
\begin{itemize}
\item[(1)]{If $b_+(X) = 1$, then for any spin structure $\mathfrak{s}$, $A,B \in H^1(X ; \mathbb{Z}_2)$ and $a,b \in H^1(X ; \mathbb{Z})$, we have:
\begin{align*}
& \langle SW_{X , A \otimes B \otimes \mathfrak{s}}(1) , a \smallsmile b \rangle + \langle SW_{X , A \otimes \mathfrak{s}}(1) , a \smallsmile b \rangle + \langle SW_{X , B \otimes \mathfrak{s}}(1) , a \smallsmile b \rangle + \langle SW_{X , \mathfrak{s}}(1) , a \smallsmile b \rangle \\
& \quad \quad \quad  = \langle [X] , A \smallsmile B \smallsmile a \smallsmile b \rangle.
\end{align*}
}
\item[(2)]{If $b_+(X) = 2$, then for any spin structure $\mathfrak{s}$, $A \in H^1(X ; \mathbb{Z}_2)$ and $a,b,c \in H^1(X ; \mathbb{Z})$, we have:
\[
\langle SW_{X , A \otimes \mathfrak{s}}(1) , a \smallsmile b \smallsmile c \rangle + \langle SW_{X , \mathfrak{s}}(1) , a \smallsmile b \smallsmile c \rangle  = \langle [X] , A \smallsmile a \smallsmile b \smallsmile c \rangle.
\]
}
\item[(3)]{If $b_+(X) = 3$, then $SW_{X,\mathfrak{s}}(1)$ taken mod $2$ does not depend on the choice of spin structure.}
\end{itemize}
\end{theorem}

From Theorems \ref{thm:thm1} and \ref{thm:thm2} we obtain a non-vanishing theorem concerning Seiberg--Witten invariants on spin $4$-manifolds:

\begin{corollary}\label{cor:nonzero}
Let $X$ be a compact, oriented, smooth spin $4$-manifold.
\begin{itemize}
\item[(1)]{If $b_+(X) = 1$ and if there exists classes $a,b \in H^1(X ; \mathbb{Z})$ and classes $A,B \in H^1(X ; \mathbb{Z}_2)$ such that $A \smallsmile B \smallsmile a \smallsmile b \neq 0$, then there is a spin structure on $X$ with non-zero mod $2$ Seiberg--Witten invariant.}
\item[(2)]{If $b_+(X) = 2$ and if there exists classes $a,b,c \in H^1(X ; \mathbb{Z})$ and a class $A \in H^1(X ; \mathbb{Z}_2)$ such that $A \smallsmile a \smallsmile b \smallsmile c \neq 0$, then there is a spin structure on $X$ with non-zero mod $2$ Seiberg--Witten invariant.}
\item[(3)]{If $b_+(X) = 3$, $\sigma(X)=0$ and if there exists classes $a,b,c,d \in H^1(X ; \mathbb{Z})$ such that $a \smallsmile b \smallsmile c \smallsmile d \neq 0 \; ({\rm mod} \; 2)$, then there is a spin structure on $X$ with non-zero mod $2$ Seiberg--Witten invariant.}
\end{itemize}

\end{corollary}

One application of non-vanishing Seiberg--Witten invariants is the adjunction inequality:

\begin{theorem}
Let $X$ be a compact, oriented, smooth, spin $4$-manifold and suppose one of the conditions (1)-(3) of Corollary \ref{cor:nonzero} is satisfied. Let $\Sigma \subset X$ be a smooth, compact, oriented surface embedded in $X$ and representing a non-torsion class $a \in H^2(X ; \mathbb{Z})$. Then the genus $g$ of $\Sigma$ satisfies
\[
2g-2  \ge | a^2 |.
\]
\end{theorem}
\begin{proof}
First of all note that if $X$ satisfies one of the conditions of Corollary \ref{cor:nonzero}, then $\sigma(X) = 0$ and hence $X$ with the opposite orientation also satisfies one of the conditions of Corollary \ref{cor:nonzero}. We choose the orientation on $X$ for which the self-intersection of $\Sigma$ is non-negative (and thus equals $|a^2|$). By Corollary \ref{cor:nonzero}, the Seiberg--Witten invariant of some spin structure on $X$ is non-zero. The adjunction inequality (eg. \cite[Theorem 11]{law}) then gives $2g-2 \ge |a^2|$.
\end{proof}

Our techniques can also be used to compute the mod $2$ families Seiberg--Witten invariants for spin families. Let $B_0$ be a compact, smooth manifold and $\pi : E \to B_0$ a smooth family of $4$-manifolds parametrised by $B_0$. This means that $E$ is a fibre bundle with fibres given by a compact, oriented smooth $4$-manifold $X$, with transition functions given by orientation preserving diffeomorphisms of $X$. Suppose that $\mathfrak{s}_E$ is a spin$^c$-structure on the vertical tangent bundle $T(E/B_0) = Ker( (\pi)_* : TE \to TB_0 )$. One can consider the Seiberg--Witten equations of the family $E$ with respect to the spin$^c$-structure \cite{bk1,bk,liliu,liu,na1,rub1,rub3}. Let $B = Pic^{\mathfrak{s}_E}(E/B_0)$ denote the space of gauge equivalence classes of spin$^c$-connections on the fibres of $E$. This is a torus bundle over $B_0$ whose fibre over $b \in B_0$ is $Pic^{\mathfrak{s}_E|_{X_b}}(X_b)$, where $X_b = \pi_E^{-1}(b)$ is the fibre of $E$ over $b$. If $b_1(X)>0$, then for technical reasons we also need to assume there exists a section $s : B_0 \to E$ \cite[Example 2.4]{bk}. A chamber for the family $E \to B_0$ is defined to be a homotopy class of non-vanishing section $\phi : B \to H^+$, where $H^+ \to B$ is the pullback to $B$ of the fibre bundle on $B_0$ whose fibre over $b \in B_0$ is $H^+(X_b)$, the space of harmonic self-dual $2$-forms (with respect to some given family of metrics on the fibres of $E$). Then we obtain the mod $2$ families Seiberg--Witten invariant which is a homomorphism
\[
SW^\phi_{E,\mathfrak{s}_E} : H^*_{S^1}(pt ; \mathbb{Z}_2) \to H^*(B ; \mathbb{Z}_2)
\]
of degree $-d(X,\mathfrak{s})$, where $X$ is a fibre of $E$ and $\mathfrak{s}$ is the restriction of $\mathfrak{s}_E$ to $X$.

The following result completely determines the mod $2$ families Seiberg--Witten invariants of spin families for even powers of $x$, under some mild assumptions on the family. The general result, Theorem \ref{thm:spinfam2}, is slightly complicated to state so here we give the result only for $b_1(X)=0$.

\begin{theorem}\label{thm:thm3}
Let $E \to B_0$ be a spin family and suppose that $b_1(X)=0$. Then for any chamber $\phi$ we have
\[
SW^\phi_{E,\mathfrak{s}_E}( x^{2m} ) = w_{b_+(X)-3}(H^+(X)) s_{2(m+1+\sigma(X)/16)}(D_{\mathfrak{s}_E}).
\]
\end{theorem}

In particular, Theorem \ref{thm:thm3} implies that the mod $2$ invariants $SW^\phi_{E,\mathfrak{s}_E}(x^{2m})$ depend only on $b_+(X), \sigma(X)$ and the $K$-theory classes $[H^+] \in KO^0(B)$, $[D_{\mathfrak{s}_E}] \in K^0(B)$. This recovers and generalises the rigidity theorem of Kato--Konno--Nakamura \cite{kkn}, which deals with the case $b_1(X) = 0$, $m=0$, $b_+(X) \ge dim(B)+2$.

In the course of proving Theorem \ref{thm:thm3}, we also obtain some constraints that spin families must satisfy. We state the result here only for the case $b_1(X)=0$.

\begin{theorem}
Let $E \to B_0$ be a spin family and suppose that $b_1(X) = 0$. Then
\[
w_l(H^+(X))s_{2(j+1+\sigma(X)/16)}(D) = 0.
\]
for all $j \ge 0$ and $b_+(X) - 2 \le l \le b_+(X)$.
\end{theorem}

\subsection{Outline of the proof of the main results}

We give an outline of the main steps in our computation of the Seiberg--Witten invariants:
\begin{itemize}
\item[(1)]{For spin structures, the Seiberg--Witten equations posess an additional symmetry $j$ known as {\em charge conjugation}. Since $j^2 = -1$, no irreducible solution to the Seiberg--Witten equations can be fixed by $j$. If we could choose a $j$-invariant perturbation for which the Seiberg-Witten moduli space is smooth, then $j$ could be used to pair off solutions, giving a mod $2$ vanishing result.}
\item[(2)]{Unfortunately, such a simple approach does not work as there are no non-zero $j$-invariant perturbations. One of the key ideas in this paper is to instead consider a $j$-invariant {\em family} of perturbations. Such families exist, in fact we can take the parameter space of the family to be the unit sphere $S(H^+(X))$ in $H^+(X)$ with $j$ acting as the antipodal map on $S(H^+(X))$.}
\item[(3)]{In light of point (2), it is possible to keep track of charge conjugation symmetry, but the price to pay is that we must now consider the Seiberg--Witten equations for a family. More precisely, we construct an enhancement of the usual Seiberg--Witten invariant  $SW_{X,\mathfrak{s}}$ which takes the form of a map
\[
SW_{X,\mathfrak{s}}^{Pin(2)} : H^*_{Pin(2)}(pt ; \mathbb{Z}_2) \to H_{\mathbb{Z}_2}^{*-d(X,\mathfrak{s})}( Pic^{\mathfrak{s}}(X) \times S(H^+(X)) ; \mathbb{Z}_2).
\]
Moreover $SW_{X,\mathfrak{s}}^\phi$ can be recovered from $SW_{X,\mathfrak{s}}^{Pin(2)}$ in that we have a commutative square of the form
\[
\xymatrix{
H^*_{Pin(2)}(pt ; \mathbb{Z}_2) \ar[rr]^-{SW^{Pin(2)}_{X,\mathfrak{s}}} \ar[d] & & H^{*-d(X,\mathfrak{s})}_{\mathbb{Z}_2}( Pic^{\mathfrak{s}}(X) \times S(H^+(X)) ; \mathbb{Z}_2) \ar[d]^-{\phi^*} \\
H^*_{S^1}(pt ; \mathbb{Z}_2) \ar[rr]^-{SW^\phi_{X,\mathfrak{s}}} & & H^{*-d(X,\mathfrak{s})}( Pic^{\mathfrak{s}}(X) ; \mathbb{Z}_2) \\
}
\]
}
\item[(4)]{Now remains the task of computing $SW_{X,\mathfrak{s}}^{Pin(2)}$. At first this might seem no easier than computing $SW_{X,\mathfrak{s}}$. However, it turns out that $SW_{X,\mathfrak{s}}^{Pin(2)}$, unlike the ordinary Seiberg--Witten invariants, behaves well under taking connected sums. Exploiting this property of $SW_{X,\mathfrak{s}}^{Pin(2)}$ allows us to compute it, and in turn to compute $SW_{X,\mathfrak{s}}$.}
\item[(5)]{We need to prove a connected sum formula for $SW_{X,\mathfrak{s}}^{Pin(2)}$. This is difficult if one chooses to work directly with moduli spaces, so instead we work throughout this paper with the Bauer--Furuta cohomotopy refinement of the Seiberg--Witten invariants. For spin structures, the Bauer--Furuta invariant has $Pin(2)$-symmetry and our enhanced Seiberg--Witten invariant $SW_{X,\mathfrak{s}}^{Pin(2)}$ can be recovered from the $Pin(2)$ Bauer--Furuta invariant.}
\item[(6)]{The Bauer--Furuta invariant of a connected sum $X \# Y$ is the smash product of the Bauer--Furuta invariants for $X$ and $Y$. Since the Bauer--Furuta invariants of $X$ and $Y$ are both equivariant, their smash product has $S^1 \times S^1$-symmetry, or $Pin(2) \times Pin(2)$ in the spin case. The usual $S^1$ or $Pin(2)$ symmetry group is obtained by restricting to the diagonal subgroup. However, it is beneficial to retain the larger symmetry group. Another key idea of this paper to use localisation in equivariant cohomology with respect to the additional circle group of symmetry. This leads to a product formula for $SW_{X\#Y , \mathfrak{s}_X \# \mathfrak{s}_Y}^{Pin(2)}$ which ultimately allows us to compute $SW_{X,\mathfrak{s}}^{Pin(2)}$ and $SW_{X,\mathfrak{s}}$.}
\item[(7)]{When $b_+(X) \ge 3$ our formula shows that $SW_{X , \mathfrak{s}}$ can be expressed in terms of Segre classes of the families index $D_{\mathfrak{s}}$ and these are easily computed. When $b_+(X) = 1$ or $2$, $SW_{X , \mathfrak{s}}$ is given by a more subtle invariant of $D_{\mathfrak{s}}$ that depends on the class of $D_{\mathfrak{s}}$ in the Quaternionic $K$-theory group $KH( Pic^{\mathfrak{s}}(X) )$. Further work is then required to identify this class in terms of the topology of $X$.}
\end{itemize}

\subsection{Structure of the paper}

In Section \ref{sec:swi} we recall the Bauer--Furuta invariant and how the Seiberg--Witten invariants of a $4$-manifold or a family of $4$-manifolds can be recovered from the Bauer--Furuta invariant. This leads us to consider more generally Seiberg--Witten type invariants associated to any $S^1$-equivariant cohomotopy class. In Section \ref{sec:pin2} we consider the Seiberg--Witten invariants in the case of spin structures. In this case there is an additional symmetry, charge conjugation, which leads us to consider $Pin(2)$-equivariant cohomotopy classes. We will see that in the $Pin(2)$-equivariant case, the mod $2$ Seiberg--Witten invariants admit an enhancement that we call the $Pin(2)$-equivariant Seiberg--Witten invariants. In Section \ref{sec:prod}, we consider the Seiberg--Witten invariants or their $Pin(2)$-equivariant enhancement for the smash product of two cohomotopy classes. Such a cohomotopy class has an additional circle of symmetry and by applying localisation in equivariant cohomology with respect to this extra symmetry, we arrive at a product formula for the Seiberg--Witten invariants or their $Pin(2)$-equivariant enhancement of a smash product. In Section \ref{sec:swspin} we apply the product formula to arrive at a formula for the $Pin(2)$-equivariant Seiberg--Witten invariants of spin $4$-manifolds. From this we obtain the mod $2$ Seiberg--Witten invariants of any spin structure. The cases $b_+(X) = 1,2$ require further analysis and this is carried out in Section \ref{sec:b+12}. Finally, in Section \ref{sec:swfam} we use the same approach to compute the Seiberg--Witten invariants for spin families.

\subsection{Acknowledgements}

We thank Hokuto Konno for comments on a draft of this paper.

\section{Monopole maps and Seiberg--Witten invariants}\label{sec:swi}

In this section we recall how the Seiberg--Witten invariants of a $4$-manifold can be recovered from the Bauer--Furuta cohomotopy refinement. We will use a more general framework that is suitable for contructing the Seiberg--Witten invariants for families of $4$-manifolds. 

We will be concerned with maps of sphere bundles over a base space $B$ which is assumed to be a compact manifold. If $V$ is a complex vector bundle over $B$ and $U$ a real vector bundle, we let $S^{V,U}$ denote the fibrewise compactification of $V \oplus U$, or equivalently the unit sphere bundle of $V \oplus U \oplus \mathbb{R}$. We let $S^1$ act on $V$ by scalar multiplication and trivially on $U$. This determines an $S^1$-action on $S^{V,U}$. We let $s_{V,U} : B \to S^{V,U}$ denote the section at infinity. Consider an $S^1$-equivariant map of sphere bundles 
\[
f : S^{V,U} \to S^{V',U'},
\]
where $V,V'$ are complex vector bundles on $B$ and $U,U'$ are real vector bundles. Assume that $f$ sends $s_{V,U}$ to $s_{V',U'}$ and that the restriction of $f$ to $S^{U}$ is homotopic to the map $S^U \to S^{U'}$ induced by an inclusion $U \to U'$ of vector bundles. We will refer to such a map of sphere bundles as a {\em monopole map}.

Suppose that $X$ is a compact, oriented, smooth $4$-manifold and that $\mathfrak{s}$ is a spin$^c$-structure on $X$. By taking a finite dimensional approximation of the Seiberg--Witten equations as in \cite{bf}, we obtain a monopole map $f : S^{V,U} \to S^{V' , U'}$ over $B = Pic^{\mathfrak{s}}(X)$, the space of gauge equivalence classes of spin$^c$-connections with curvature equal to a fixed $2$-form representing $-2\pi i c(\mathfrak{s})$. This is a torsor over $Pic(X)$, the group of flat unitary line bundles on $X$ and hence is a torus of dimension $b_1(X)$. The Bauer--Furuta invariant of $(X,\mathfrak{s})$ is the equivariant stable cohomotopy class of $f$.

We introduce some notation associated to a monopole map $f : S^{V,U} \to S^{V',U'}$. Let $a,a'$ denote the complex ranks of $V,V'$ and let $b,b'$ denote the real ranks of $U,U'$. Further, set $d = a-a'$ and $b_+ = b'-b$. Let $D$ denote the virtual vector bundle $D = V - V'$ and let $H^+ = U' - U$. Since $f|_U$ is homotopy equivalent to an inclusion, we have an isomorphism $H^+ \cong U'/U$, in particular $H^+$ is a genuine vector bundle. In the case that $f$ is the Bauer--Furuta monopole map for a $4$-manifold $X$ and spin$^c$-structure $\mathfrak{s}$, $D \to Pic^{\mathfrak{s}}$ is the families index of the family of Dirac operators parametrised by $Pic^{\mathfrak{s}}(X)$ and $H^+$ is the trivial bundle with fibre $H^+(X)$, the space of harmonic self-dual $2$-forms on $X$ for some Riemannian metric.

In \cite{bk}, we constructed cohomological invariants associated to a monopole map $f : S^{V,U} \to S^{V',U'}$, which in the case of the Bauer--Furuta monople map for $(X,\mathfrak{s})$ recovers the Seiberg--Witten invariant. More generally, for a family of $4$-manifolds this procedure recovers the families Seiberg--Witten invariants. We recall the construction of these invariants. Since our interest is in the mod $2$ Seiberg--Witten invariants, we will work throughout with $\mathbb{Z}_2$-coefficients. This has the benefit that we do not have to consider orientations.

Since the restriction of $f$ to $S^U$ is homotopic to an inclusion, we can identify $U$ with a subbundle of $U'$ and we can further assume that $f|_{S^U}$ is given by the inclusion map. The cohomological invariants of $f$ depend on a choice of a {\em chamber} for $f$, which by definition is a homotopy class of section $\phi : B \to U' \setminus U$. Equivalently as chamber can be regarded as a homotopy class of non-vanishing section of $H^+ = U'/U$. Such a chamber determines a lift $\tau^\phi_{V',U'} \in H^{2a'+b'}_{S^1}( S^{V',U'} , S^{U} )$ of the Thom class $\tau_{V',U'} \in H_{S^1}^{2a'+b'}( S^{V',U'} , s_{V',U'})$ as follows. Let $N_\phi$ denote an $S^1$-invariant tubular neighbourhood of $\phi(B)$ in $S^{V',U'}$. The Thom class $\tau_{N_\phi}$ of $N_\phi \to \phi(B)$ is valued in $H^{2a'+b'}_{S^1}( N_\phi , N_\phi \setminus \phi(B))$ which by excision is isomorphic to $H^{2a'+b'}_{S^1}( S^{V',U'} , S^{V',U'} \setminus \phi(B))$. Then since $\phi(B)$ is disjoint from $S^{U}$ we can map $\tau_{N_\phi}$ to a class in $H^{2a'+b'}_{S^1}( S^{V',U'} , S^U)$, which is $\tau^\phi_{V',U'}$. 

Pulling back the lifted Thom class by $f$ gives $f^*(\tau^\phi_{V',U'}) \in H^{2a'+b'}_{S^1}(S^{V,U} , S^U)$. Let $N^U$ denote an $S^1$-invariant tubular neighbourhood of $S^U$ in $S^{V,U}$ and let $\widetilde{Y}^{V,U} = S^{V,U} \setminus N^U$ be the complement (alternatively one can construct $\widetilde{Y}^{V,U}$ as the real blowup of $S^{V,U}$ along $S^U$). Then $\widetilde{Y}^{V,U}$ is a manifold with boundary. Furthermore, the $S^1$-action is free and we set $Y^{V,U} = \widetilde{Y}^{V,U}/S^1$. By excision and homotopy invariance we have isomorphisms
\[
H^*_{S^1}( S^{V,U} , S^U) \cong H^*_{S^1}( S^{V,U} , N^U) \cong H^*_{S^1}( \widetilde{Y}^{V,U} , \partial \widetilde{Y}^{V,U}) \cong H^*( Y^{V,U} , \partial Y^{V,U}).
\]
Hence we can regard $f^*( \tau^\phi_{V,U})$ as an element of $H^{2a'+b'}( Y^{V,U} , \partial Y^{V,U} )$.

Let $\pi_{V,U} : Y^{V,U} \to B$ be the projection map. We have a pushforward map
\[
(\pi_{V,U})_* : H^*( Y^{V,U} , \partial Y^{V,U} ) \to H^{*-(2a+b-1)}(B),
\]
which is obtained from the corresponding pushforward map in homology using Poincar\'e--Lefschetz duality. Now we define the {\em Seiberg--Witten invariant of $f$ with respect to the chamber $\phi$} to be the homomorphism
\[
SW^\phi_f : H^*_{S^1}(pt) \to H^{*-2d+b_++1}(B)
\]
given by
\[
SW^\phi_f(\theta) = (\pi_{V,U})_*(  \theta  f^*(\tau^\phi_{V',U'})).
\]

Sometimes it is useful to consider a slightly more general invariant $SW^\phi_f : H^*_{S^1}( \widetilde{Y}^{V,U} ) \to H^{*-2d+b_++1}(B)$ by allowing $\theta$ to be an element of $H^*_{S^1}( \widetilde{Y}^{V,U})$. However, this is not a stable invariant of $f$ since the space $\widetilde{Y}^{V,U}$ depends on $V$ and $U$. As a special case, we can take any element in $H^*_{S^1}(B)$ and pull it back to $H^*_{S^1}(\widetilde{Y}^{V,U})$ and in this case we do get a stable invariant.

Recall that $H^*_{S^1}(pt) \cong \mathbb{Z}_2[x]$, where $deg(x) = 2$. Therefore $SW^\phi_f$ is completely determined by the collection of cohomology classes $SW^\phi_f( x^m) \in H^{2m-2d+b_+ + 1}(B)$, where $m \ge 0$.

In the case that $f$ is the monopole map associated to a $4$-manifold $X$ with spin$^c$-structure $\mathfrak{s}$, we have $d = (c(\mathfrak{s})^2 - \sigma(X))/8$, $b_+ = b_+(X)$ and $B = Pic^{\mathfrak{s}}(X)$. Since $Pic^{\mathfrak{s}}(X)$ is a torsor for $T_X = H^1(X ; \mathbb{R})/H^1(X ; \mathbb{Z})$, we have a canonical isomorphism $H^*(Pic^{\mathfrak{s}}(X)) \cong H^*(T_X)$. So the Seiberg--Witten invariant takes the form $SW^\phi_{X,\mathfrak{s}} : H^*_{S^1}(pt) \to H^{*-2d+b_++1}(T_X)$ and is equivalent to the collection of cohomology classes $SW^\phi_f(x^m) \in H^{2m-2d+b_+ + 1}(T_X)$.

\section{Spin structures and $Pin(2)$-symmetry}\label{sec:pin2}

For a spin-structure $\mathfrak{s}$, the Seiberg--Witten equations have an additional symmetry known as charge conjugation, which we denote by $j$. The corresponding monopole map is $Pin(2)$-equivariant, where $Pin(2) = S^1 \cup j S^1$ with relations $j e^{i\theta} = e^{-i\theta}j$, $j^2= -1$. This motivates us to consider $Pin(2)$-equivariant monopole maps more generally.

Let $B$ be a compact manifold. Assume that $B$ is equipped with an involution $\iota : B \to B$ and let $Pin(2)$ act on $B$, where $S^1 \subset Pin(2)$ acts trivially and $j$ acts as $\iota$. Let $E \to B$ be a complex vector bundle on $B$. Suppose that $E$ is equipped with an antilinear endomorphism $J : E \to E$ covering $\iota$ and satisfying $J^2 = -1$. Then we make $E$ into a $Pin(2)$-equivariant vector bundle over $B$ by letting $S^1 \subset Pin(2)$ act by scalar multiplication and $j$ act by $J$. Let $F \to B$ be a real vector bundle and suppose $F$ is equipped with an involutive endomorphism $J : F \to F$ covering $\iota$. Then we make $F$ into a $Pin(2)$-equivariant vector bundle over $B$ by letting $S^1 \subset Pin(2)$ act trivially and let $j$ act by $J$.

Consider now a $Pin(2)$-equivariant map $f : S^{V,U} \to S^{V',U'}$, where $V,V'$ are complex vector bundles equipped with anti-linear endomorphisms covering $\iota$ and squaring to $-1$ and $U,U'$ are real vector bundles equipped with involutive endomorphisms covering $\iota$. Assume that $f$ sends $s_{V,U}$ to $s_{V',U'}$ and that $f|_{S^U}$ is $Pin(2)$-homotopic to the map $S^U \to S^{U'}$ induced by an inclusion of vector bundles $U \to U'$. We will refer to such a map as a {\em $Pin(2)$-equivariant monopole map}.

In the case that $f$ is the monopole map for a $4$-manifold $X$ and spin-structure $\mathfrak{s}$, recall that $B = Pic^{\mathfrak{s}}(X)$. A spin connection defines an origin in $Pic^{\mathfrak{s}}(X)$ and hence gives an identification $Pic^{\mathfrak{s}}(X) \cong H^1(X ; \mathbb{R})/H^1(X ; \mathbb{Z})$. The involution $\iota : B \to B$ acts on $H^1(X ; \mathbb{R})/H^1(X ; \mathbb{Z})$ as $-1$ (i.e. the inverse map of the group structure). Furthermore, $U,U'$ are trivial vector bundles over $B$ and $j$ is the map $-\iota^*$, that is, a combination of pullback by $\iota$ and multiplication by $-1$.

The construction of cohomological invariants given in Section \ref{sec:swi} can be repeated for $Pin(2)$-monopole maps but now keeping track of the additional symmetry. A chamber in the $Pin(2)$ sense is a $Pin(2)$-equivariant homotopy class $\phi : B \to U' \setminus U$. This determines a lifted Thom class $\tau^\phi_{V',U'} \in H_{Pin(2)}^{2a'+b'}( S^{V',U'} , S^{U})$ which pulls back to $f^*(\tau^\phi_{V',U'}) \in H_{Pin(2)}^{2a'+b'}( S^{V,U} , S^U)$. As before we have isomorphisms
\[
H^*_{Pin(2)}(S^{V,U} , S^U) \cong H^*_{Pin(2)}( \widetilde{Y}^{V,U} , \partial \widetilde{Y}^{V,U}) \cong H^*_{\mathbb{Z}_2}( Y^{V,U} , \partial Y^{V,U}).
\]
Now the projection map $\pi_{V,U} : Y^{V,U} \to B$ is $\mathbb{Z}_2$-equivariant and hence it determines a push-forward map $(\pi_{V,U})_* : H^*_{\mathbb{Z}_2}( Y^{V,U} , \partial Y^{V,U} ) \to H^{*-(2a+b-1)}_{\mathbb{Z}_2}(B)$ in equivariant cohomology. This requires some justification since Poincar\'e--Lefschetz duality does not hold in equivariant cohomology. Consider more generally a compact Lie group $G$ and a $G$-equivariant fibre bundle $\pi : E \to B$, where the fibre $F$ is a compact $n$-manifold with boundary. The transition functions for the fibre bundle are homeomorphisms so they must send the boundary of $F$ to itself. Hence we also have a fibre bundle $\partial E \to B$ whose fibres are the boundaries of the fibres of $E$. Replacing $E$ and $B$ by their Borel models $E_G = E \times_G EG$, $B_G = B \times_G BG$, we have a fibre bundle $\pi : E_G \to B_G$, with fibre $F$. Let $\partial E_G$ denote the Borel model for $\partial E$ and let $E'_G$ be the space obtained by collapsing the boundary of each fibre to a disctint point. This is a fibre bundle over $B_G$ with fibre $F/\partial F$. Consider the Leray--Serre spectral sequence $E_r^{p,q}$ for $E'_G \to B_G$. Since $H^k( F/\partial F) = 0$ for $k > n$, we have $E_r^{p,q} = 0$ for $q > n$ and hence there is a map 
\[
H^m( E'_G ) \to E_2^{m-n , n} = H^{m-n}( B_G ; H^n( F , \partial F) ).
\]
Here $H^n(F , \partial F)$ is to be understood as a local system on $B$. However, since we are working with $\mathbb{Z}_2$-coefficients and $F$ is a compact $n$-manifold, we have $H^n(F , \partial F) \cong \mathbb{Z}_2$, the trivial local system with coefficient group $\mathbb{Z}_2$. So we have a well defined map $H^m( E'_G) \to H^{m-n}(B_G)$. From the definition of $E'_G$, there is a quotient map $E'_G \to E_G/\partial E_G$ and hence a pullback map $H^m( E_G , \partial E_G) \to H^m(E'_G)$. Composing, we get a map
\[
\pi_* : H^m( E_G , \partial E_G) \to H^m(E'_G) \to H^{m-n}( B_G)
\]
or equivalently, a map in equivariant cohomology
\[
\pi_* : H^m_G( E , \partial E) \to H^m_G(E') \to H^{m-n}_G(B).
\]

Thus, to any $Pin(2)$-monopole map $f$ and chamber $\phi$, we may define the {$Pin(2)$-equivariant Seiberg--Witten invariant of $f$ with respect to $\phi$} to be the map
\[
SW^\phi_{Pin(2) , f} : H^*_{Pin(2)}(pt) \to H^{*-2d+b_++1}_{\mathbb{Z}_2}(B)
\]
given by
\[
SW^\phi_{Pin(2), f}(\theta) = (\pi_{V,U})_*( \theta  f^*(\tau^\phi_{V',U'})).
\]

Forgetting the additional symmetry recovers the usual Seiberg--Witten invariant in the sense that we have a commutative diagram
\begin{equation}\label{equ:cd1}
\xymatrix{
H^*_{Pin(2)}(pt) \ar[rr]^-{SW^\phi_{Pin(2),f}} \ar[d] & & H^{*-2d+b_++1}_{\mathbb{Z}_2}(B) \ar[d] \\
H^*_{S^1}(pt) \ar[rr]^-{SW^\phi_f} & & H^{*-2d+b_++1}(B) \\
}
\end{equation}

However there may be some loss of information in passing to the $Pin(2)$-equivariant Seiberg--Witten invariant as the map $H^*_{Pin(2)}(pt) \to H^*_{S^1}(pt)$ is not surjective. In fact, we have $H^*_{Pin(2)}(pt) \cong \mathbb{Z}_2[u,q]/(u^3)$ where $deg(u)=1$, $deg(q) = 4$ (eg. \cite[\textsection 5]{bar}) and the map $H^*_{Pin(2)}(pt) \to H^*_{S^1}(pt)$ sends $u$ to zero and $q$ to $x^2$.

The map $SW_{Pin(2),f}^{\phi}$ is a morphism of $H^*_{\mathbb{Z}_2}(pt)$-modules. Recall that $H^*_{\mathbb{Z}_2}(pt) \cong \mathbb{Z}_2[u]$, where $deg(u) = 1$. So we have, $SW_{Pin(2),f}^{\phi}(u\theta) = uSW_{Pin(2),f}^{\phi}(\theta)$.

As in the $S^1$-equivariant case, it is sometmes convenient to regard the domain of $SW^\phi_{Pin(2),f}$ to be either $H^*_{Pin(2)}( \widetilde{Y}^{V,U})$ or $H^*_{Pin(2)}(B)$ 

In the case that $f$ is the Seiberg--Witten monopole map for a $4$-manifold $X$ with spin-structure $\mathfrak{s}$, we run into an immediate problem. There are no $Pin(2)$-equivariant chambers because the action of $j$ on $Pic^{\mathfrak{s}}(X)$ always has fixed points, while $j$ acts on $H^+(X)$ as $-1$. So it would appear that we can not take advantage of the $Pin(2)$-symmetry in this situation. Fortunately there is a simple way to circumvent this difficulty, which we previously made use of in \cite[\textsection 9.4]{bk1}. The idea is to replace $Pic^{\mathfrak{s}}(X)$ with $B_{X,\mathfrak{s}} = Pic^{\mathfrak{s}}(X) \times S(H^+(X))$, where $S(H^+(X))$ is the unit sphere in $H^+(X)$. We define $\iota : B_{X,\mathfrak{s}} \to B_{X,\mathfrak{s}}$ to be the product of $-1$ on $Pic^{\mathfrak{s}}(X)$ with the antipodal map on $S(H^+(X))$. Then we simply pullback the monopole map from $Pic^{\mathfrak{s}}(X)$ to $B_{X,\mathfrak{s}}$. Then we have a tautological chamber $\phi^{taut} : B_{X,\mathfrak{s}} \to H^+(X) \setminus \{0\}$ which given by the projection $B_{X,\mathfrak{s}} \to S(H^+(X))$, followed by the inclusion $S(H^+(X)) \to H^+(X) \setminus \{0\}$. Hence to any $4$-manifold $X$ with $b_+(X)>0$ and with spin structure $\mathfrak{s}$, we obtain a $Pin(2)$-equivariant Seiberg--Witten invariant
\[
SW^{\phi^{taut}}_{Pin(2),f} : H^*_{Pin(2)}(pt) \to H^{*-2d+b_++1}_{\mathbb{Z}_2}( B_{X,\mathfrak{s}}).
\]
To simplify notation we write $SW^{Pin(2)}_{X,\mathfrak{s}}$ for $SW^{\phi^{taut}}_{Pin(2),f}$.

Recall that $H^*_{\mathbb{Z}_2}(pt) \cong \mathbb{Z}_2[u]$, where $deg(u) = 1$.

\begin{proposition}\label{prop:z2cohom}
We have a canonical isomorphism of $\mathbb{Z}_2[u]$-algebras:
\[
H^*_{\mathbb{Z}_2}(B_{X,\mathfrak{s}}) \cong H^*_{\mathbb{Z}_2}( Pic^{\mathfrak{s}}(X))/( u^{b_+(X)}).
\]
Furthermore, we have a (non-canonical) isomorphism of $\mathbb{Z}_2[u]$-modules
\[
H^*_{\mathbb{Z}_2}( Pic^{\mathfrak{s}}(X) ) \cong H^*( Pic^{\mathfrak{s}}(X) )[u].
\]
Hence, we also have an isomorphism of $\mathbb{Z}_2[u]$-modules
\[
H^*_{\mathbb{Z}_2}(B_{X,\mathfrak{s}}) \cong H^*(Pic^{\mathfrak{s}}(X))[u]/(u^{b_+(X)}).
\]
\end{proposition}
\begin{proof}
We first prove the isomorphism $H^*_{\mathbb{Z}_2}( Pic^{\mathfrak{s}}(X) ) \cong H^*( Pic^{\mathfrak{s}}(X) )[u]$ of $\mathbb{Z}_2[u]$-modules. Consider the Borel spectral sequence for $H^*_{\mathbb{Z}_2}( Pic^{\mathfrak{s}}(X)$ which has $E_2^{p,q} = H^p( B\mathbb{Z}_2 ; H^q(Pic^{\mathfrak{s}}(X)) )$. Since $-1$ acts trivially on the mod $2$ cohomology of $Pic^{\mathfrak{s}}(X)$, this is isomorphic to $H^p(B\mathbb{Z}_2 ) \otimes_{\mathbb{Z}_2} H^q( Pic^{\mathfrak{s}}(X) )$. Let $o \in Pic^{\mathfrak{s}}(X)$ denote the distinguished point of $Pic^{\mathfrak{s}}(X)$ determined by the spin-connection. Then $o$ is fixed by $-1$ and hence there can be no differentials into $E_r^{p,0}$ for any $p$ or $r$. This implies that the forgetful map $H^1_{\mathbb{Z}_2}( Pic^{\mathfrak{s}}(X) ) \to H^1( Pic^{\mathfrak{s}}(X)$ is surjective. But $Pic^{\mathfrak{s}}(X)$ is a torus so $H^*( Pic^{\mathfrak{s}}(X))$ is generated by $H^1( Pic^{\mathfrak{s}}(X))$, hence the forgetful map $H^k_{\mathbb{Z}_2}( Pic^{\mathfrak{s}}(X) ) \to H^k( Pic^{\mathfrak{s}}(X))$ is surjective for all $k$. This implies that there are no differentials in the spectral sequence and hence $H^*_{\mathbb{Z}_2}( Pic^{\mathfrak{s}}(X) ) \cong H^*( Pic^{\mathfrak{s}}(X) )[u]$ as $\mathbb{Z}_2[u]$-modules. 

Set $n = b_+(X)-1$. The projection $B_{X , \mathfrak{s}} \to Pic^{\mathfrak{s}}(X)$ induces a canonical homomorphism $H^*_{\mathbb{Z}_2}( Pic^{\mathfrak{s}}(X) ) \to H^*_{\mathbb{Z}_2}(B_{X,\mathfrak{s}})$. Furthermore, since the map $B_{X,\mathfrak{s}} \to \{pt \}$ factors as $B_{X,\mathfrak{s}} \to S(H^+(X)) \to pt$, it follows that $u^{b_+(X)} = 0$ in $H^*_{\mathbb{Z}_2}(B_{X,\mathfrak{s}})$, because $u^{b_+(X)} = u^{n+1} = 0$ in $H^*_{\mathbb{Z}_2}( S(H^+(X)) ) \cong H^*( \mathbb{RP}^n ) \cong \mathbb{Z}_2[u]/(u^{n+1})$. So we have a canonical $\mathbb{Z}_2[u]$-algebra homomorphism $\phi : H^*_{\mathbb{Z}_2}( Pic^{\mathfrak{s}}(X) )/(u^{n+1}) \to H^*_{\mathbb{Z}_2}( B_{X,\mathfrak{s}})$. It remains to show that this map is an isomorphism.

Since the antipodal map acts freely on $S(H^+(X)) \cong S^n$, it follows that $\iota$ acts freely and that the quotient $B_{X,\mathfrak{s}}/\langle \iota \rangle$ has the structure of a fibre bundle over $\mathbb{RP}^n$ with fibre $Pic^{\mathfrak{s}}(X)$. Furthermore, the point $o \in Pic^{\mathfrak{s}}(X)$ is fixed by $-1$ and hence defines a section $\mathbb{RP}^n \to B_{X,\mathfrak{s}}/\langle \iota \rangle$. Let $E_2^{p,q}$ denote the Leray--Serre spectral sequence for $p : B_{X,\mathfrak{s}}/\langle \iota \rangle \to \mathbb{RP}^n$. The existence of a section implies that $p^* : H^*(\mathbb{RP}^n) \to H^*( B_{X,\mathfrak{s}}/\langle \iota \rangle )$ is injective. Hence there are no differentials into $E_r^{p,0}$ for any $p$ or $r$. This implies that the pullback map $H^1( B_{X,\mathfrak{s}}/\langle \iota \rangle ) \to H^1( Pic^{\mathfrak{s}}(X))$ is surjective. But $Pic^{\mathfrak{s}}(X)$ is generated by $H^1( Pic^{\mathfrak{s}}(X))$, hence the pullback map $H^k( B_{X,\mathfrak{s}}/\langle \iota \rangle ) \to H^k( Pic^{\mathfrak{s}}(X))$ is surjective for all $k$. The Leray--Hirsch theorem then implies that (as $\mathbb{Z}_2[u]$-modules) we have 
\[
H^*_{\mathbb{Z}_2}( B_{X , \mathfrak{s}}) \cong H^*( Pic^{\mathfrak{s}}(X) )[u]/(u^{n+1}) \cong H^*_{\mathbb{Z}_2}(Pic^{\mathfrak{s}}(X))/(u^{b_+(X)}).
\]
\end{proof}

\begin{remark}
The $\mathbb{Z}_2[u]$-module isomorphism $H^*_{\mathbb{Z}_2}( B_{X,\mathfrak{s}}) \cong H^*( Pic^{\mathfrak{s}}(X) )[u]/(u^{b_+(X)})$ of Proposition \ref{prop:z2cohom} is determined once we choose a splitting $H^*( Pic^{\mathfrak{s}}(X) ) \to H^*_{\mathbb{Z}_2}( Pic^{\mathfrak{s}}(X))$ of the forgetful map $H^*_{\mathbb{Z}_2}( Pic^{\mathfrak{s}}(X) ) \to H^*( Pic^{\mathfrak{s}}(X))$. Under any such choice of isomorphism, the forgetful map $H^*_{\mathbb{Z}_2}( Pic^{\mathfrak{s}}(X) ) \to H^*( Pic^{\mathfrak{s}}(X))$ corresponds to the obvious map $H^*(Pic^{\mathfrak{s}}(X) )[u] \to H^*(Pic^{\mathfrak{s}}(X))$ sending $u$ to zero.
\end{remark}

\begin{remark}\label{rem:u0}
Since the forgetful map $H^*_{\mathbb{Z}_2}( Pic^{\mathfrak{s}}(X) ) \to H^*( Pic^{\mathfrak{s}}(X))$ sends $u$ to zero, it factors through a map
\[
H^*_{\mathbb{Z}_2}( Pic^{\mathfrak{s}}(X) )/(u^{b_+(X)}) \to H^*(Pic^{\mathfrak{s}}(X) ).
\]
Combined with the canonical isomorphism $H^*_{\mathbb{Z}_2}( B_{X,\mathfrak{s}} ) \cong H^*_{\mathbb{Z}_2}( Pic^{\mathfrak{s}}(X) )/(u^{b_+(X)})$ of Proposition \ref{prop:z2cohom}, this defines a canonical map
\[
H^*_{\mathbb{Z}_2}( B_{X,\mathfrak{s}} ) \to H^*(Pic^{\mathfrak{s}}(X))
\]
which we will denote by $\alpha \mapsto \alpha|_{u=0}$. Under a choice of isomorphism $H^*_{\mathbb{Z}_2}(B_{X,\mathfrak{s}}) \cong H^*( Pic^{\mathfrak{s}}(X) )[u]/(u^{b_+(X)})$, this map is just the map which sends $u$ to zero.
\end{remark}

Let us write $SW_{X,\mathfrak{s}}^{\mathbb{Z},\phi}$ to distinguish the integral Seiberg--Witten invariant from the mod $2$ Seiberg--Witten invariant $SW^\phi_{X,\mathfrak{s}}$.

\begin{lemma}\label{lem:wcf1}
If $b_+(X) = 1$ and $\mathfrak{s}$ is a spin-structure, then $SW_{X,\mathfrak{s}}^{\mathbb{Z},\phi}$ does not depend on the chamber $\phi$.
\end{lemma}
\begin{proof}
Since $X$ is spin and $b_+(X)=1$, we must have $b_-(X)=1$ (by Donaldson's Theorem B in the simply connected case \cite{don}, or the 10/8 inequality more generally \cite{fur}). The wall-crossing formula (eg. \cite{bk}) implies that $SW_{X,\mathfrak{s}}^{\mathbb{Z},\phi}(x^m) - SW_{X,\mathfrak{s}}^{\mathbb{Z},-\phi}(x^m) = \pm s_{m+1}(D)$, where $s_j(D)$ denotes the $j$-th Segre class of the index bundle $D \to Pic^{\mathfrak{s}}(X)$. Since $b_+(X)=1$ and $c(\mathfrak{s}) = 0$, the calculation in \cite[\textsection 5.3]{bk} shows that $s_j(D) = 0$ for all $j > 0$, hence the result follows.
\end{proof}

By Lemma \ref{lem:wcf1}, for a spin structure, $SW_{X,\mathfrak{s}}^{\mathbb{Z},\phi}$ and $SW_{X,\mathfrak{s}}^{\phi}$ do not depend on $\phi$ even when $b_+(X)=1$ and so we will denote these invariants as $SW_{X,\mathfrak{s}}^{\mathbb{Z}}$ and $SW_{X,\mathfrak{s}}$.

\begin{lemma}\label{lem:u=0}
Let $X$ be a compact, oriented, smooth $4$-manifold with $b_+(X) > 0$ and $\mathfrak{s}$ a spin-structure. If $m$ is odd then $SW^{\mathbb{Z}}_{X,\mathfrak{s}}(x^m) = 0$. If $m$ is even then $SW_{X,\mathfrak{s}}(x^m) = SW^{Pin(2)}_{X,\mathfrak{s}}(q^{m/2})|_{u=0}$, where the map $\alpha \mapsto \alpha|_{u=0}$ is defined as in Remark \ref{rem:u0}.
\end{lemma}
\begin{proof}
Let $\iota : Pic^{\mathfrak{s}}(X) \to Pic^{\mathfrak{s}}(X)$ denote the inversion map. The charge conjugation symmetry of the Seiberg--Witten equations implies that $SW_{X,\mathfrak{s}}^{\mathbb{Z}}( x^m) = (-1)^{\sigma} \iota^* SW_{X,\mathfrak{s}}^{\mathbb{Z}}(x^m)$, where $\sigma = m+d+b_+(X)+1$. However $SW_{X,\mathfrak{s}}^{\mathbb{Z}}(x^m)$ has degree $2m-2d+b_+(X)+1$ so $\iota^*$ acts as $(-1)^{2m-2d+b_+(X)+1}$. So the formula simplifies to  $SW_{X,\mathfrak{s}}^{\mathbb{Z}}( x^m) = (-1)^{m+d} SW_{X,\mathfrak{s}}^{\mathbb{Z}}(x^m)$. Furthermore, $d = -\sigma(X)/8$ is even as $\sigma(X)$ is a multiple of $16$. So $SW_{X,\mathfrak{s}}^{\mathbb{Z}}( x^m) = (-1)^{m} SW_{X,\mathfrak{s}}^{\mathbb{Z}}(x^m)$ and hence $SW_{X,\mathfrak{s}}^{\mathbb{Z}}(x^m) = 0$ if $m$ is odd, as $H^*( Pic^{\mathfrak{s}}(X) ; \mathbb{Z})$ has no torsion.

Now suppose $m$ is even. Let $f : S^{V,U} \to S^{V',U'}$ be the Seiberg--Witten monopole map over $Pic^{\mathfrak{s}}(X)$. Pull this back to $B_{X,\mathfrak{s}} = Pic^{\mathfrak{s}}(X) \times S(H^+(X))$. Adapting the commutative diagram (\ref{equ:cd1}) to this setting, we have a commutative diagram
\[
\xymatrix{
H^*_{Pin(2)}( pt ) \ar[rr]^-{SW^{Pin(2)}_{X,\mathfrak{s}}} \ar[d] & & H^{*-2d+b_++1}_{\mathbb{Z}_2}(B_{X,\mathfrak{s}}) \ar[d] \ar[r]^-{\cong} & H^{*-2d+b_++1}( Pic^{\mathfrak{s}}(X))[u]/(u^{b_+(X)}) \ar[dl]^-{ \; |_{u=0}}\\
H^*_{S^1}( pt ) \ar[rr]^-{SW_{X,\mathfrak{s}}} & & H^{*-2d+b_++1}( Pic^{\mathfrak{s}}(X)) & \\
}
\]
Then since $q^{m/2} \in H^*_{Pin(2)}(pt)$ gets sent to $x^m \in H^*_{S^1}( pt )$, commutativity of the diagram gives $SW_{X,\mathfrak{s}}(x^m) = SW^{Pin(2)}_{X,\mathfrak{s}}(q^{m/2})|_{u=0}$.
\end{proof}

By this lemma, the task of computing the mod $2$ Seiberg--Witten invariants for spin structures is reduced to calculating $SW^{Pin(2)}_{X,\mathfrak{s}}|_{u=0}$. In fact, we will compute the whole invariant $SW^{Pin(2)}_{X,\mathfrak{s}}$, not just its evalution at $u=0$. However, before carrying this out we already obtain a strong vanishing theorem which implies that $SW_{X,\mathfrak{s}}$ for a spin structure is usually zero mod $2$.

\begin{theorem}\label{thm:vanb+}
Let $X$ be a compact, oriented, smooth $4$-manifold and $\mathfrak{s}$ a spin-structure. If $b_+(X) > 3$, then $SW_{X,\mathfrak{s}}(\theta) = 0$ for all $\theta \in H^*_{S^1}(pt)$.
\end{theorem}
\begin{proof}
Assume that $b_+(X)>3$. By Lemma \ref{lem:u=0}, it suffices to show that \linebreak $SW^{Pin(2)}_{X,\mathfrak{s}}(q^m)|_{u=0} = 0$ for all $m \ge 0$. Recall that $SW^{Pin(2)}_{X , \mathfrak{s}}$ is a map 
\[
SW^{Pin(2)}_{X,\mathfrak{s}} : H^*_{Pin(2)}(pt) \to H^{*-2d+b_+(X)+1}_{\mathbb{Z}_2}( B_{X,\mathfrak{s}}) \cong H^{*-2d+b_+(X)+1}( Pic^{\mathfrak{s}})[u]/(u^{b_+(X)}).
\]
Recall also that $u^3 = 0$ in $H^*_{Pin(2)}(pt)$. Hence 
\[
u^3 SW^{Pin(2)}_{X,\mathfrak{s}}(q^m) = SW^{Pin(2)}_{X,\mathfrak{s}}(u^3 \theta ) = 0.
\]
This means that $SW^{Pin(2)}_{X , \mathfrak{s}}(q^m)$ is divisible by $u^{b_+(X)-3}$ and hence $SW^{Pin(2)}_{X , \mathfrak{s}}(q^m)|_{u=0} = 0$.
\end{proof}

Theorem \ref{thm:vanb+} is a generalisation of the main result of \cite{bau} (see also \cite{li}), which corresponds to the case that $-\sigma(X)/4-b_+(X)-1+b_1(X) = 0$, or equivalently the case that the moduli space of the Seiberg--Witten equations is zero-dimensional.

\section{A product formula for Seiberg--Witten invariants}\label{sec:prod}

Suppose that we have two $S^1$-equivariant monopole maps
\[
f_i : S^{V_i,U_i} \to S^{V'_i,U'_i}, \quad i=1,2
\]
over a common base $B$. Let $f = f_1 \wedge_B f_2 : S^{V,U} \to S^{V',U'}$ be the fibrewise smash product of $f_1$ and $f_2$, where $V = V_1 \oplus V_2$ etc. It is $S^1$-equivariant where $S^1$ acts on both factors. Our goal in this Section is to compute $SW_f$ in terms of $SW_{f_1}$ and $SW_{f_2}$.

Let $\phi : B \to H^+ \setminus \{0\}$ be a chamber for $f$. Since $H^+ = H^+_1 \oplus H^+_2$ we can write $\phi = (\phi_1 , \phi_2)$, where $\phi_1$ and $\phi_2$ do not simultaneously vanish. Perturbing $\phi$ slightly, we may assume that $\phi_1$ and $\phi_2$ meet the zero sections of $H^+_1, H^+_2$ transversally. Let $Z_1,Z_2 \subset B$ be the zero loci of $\phi_1, \phi_2$. So $Z_1, Z_2$ are disjoint and $Z_i$ is Poincar\'e dual to the Euler class $e(H^+_i) \in H^{b^+_i}(B)$, where $b^+_i$ denotes the rank of $H^+_i$.

The key observation is that the map $f$ is $S^1 \times S^1$-equivariant, where the $i$-th copy of $S^1$ acts as scalar multiplication on $V_i$ and $V'_i$. Keeping track of this extra symmetry, we will be able to compute $SW_f^\phi$ using localisation in equivariant cohomology. Carrying out all constructions with respect to this larger group, we get a lifted Thom class
\[
\tau^\phi_{V,U} \in H^*_{S^1 \times S^1}( S^{V',U'} , S^{U} ),
\]
which pulls back to
\[
f^*(\tau^\phi_{V,U}) \in H^*_{S^1 \times S^1}( \widetilde{Y}^{V,U} , \partial \widetilde{Y}^{V,U}) \cong H^*_{S^1}( Y^{V,U} , \partial^{V,U})
\]
where on the right, we have identified the quotient of $S^1 \times S^1$ by the diagonal subgroup $\Delta S^1$ with $S^1$ via
\[
(S^1 \times S^1)/\Delta S^1 \cong S^1, \quad (a,b) \mapsto ab^{-1}.
\]
The quotient map $S^1 \times S^1 \to S^1$ is split surjective with splitting map $S^1 \to S^1 \times S^1$ given by $a \mapsto (a,1)$. Hence we can identify the quotient group with the subgroup $S^1 \times \{1\}$.

The projection map $\pi_{V,U} : Y^{V,U} \to B$ is $S^1$-equivariant, where $S^1$ acts trivially on $B$, hence, as explained in Section \ref{sec:pin2}, defines a push-forward map
\[
(\pi_{V,U})_* : H^*_{S^1}( Y^{V,U} , \partial Y^{V,U}) \to H_{S^1}^{*-(2a+b-1)}(B).
\]
It follows that we can define an enhancement of $SW_f^\phi$ valued in $S^1$-equivariant cohomology:
\begin{align*}
SW_{S^1 \times S^1,f}^{\phi} : H^*_{S^1 \times S^1}(pt) &\to H_{S^1}^{*-2d+b_++1}(B), \\
\quad SW_{S^1\times S^1,f}^\phi(\theta) &= (\pi_{V,U})_*( \theta f^*(\tau^\phi_{V,U})).
\end{align*}
Furthermore, the map is compatible with $SW_f^\phi$ in the sense that we have a commutative diagram
\[
\xymatrix{
H_{S^1 \times S^1}^*(pt) \ar[rr]^-{SW_{S^1 \times S^1, f}^\phi} \ar[d] & & H_{S^1}^{*-2d+b_++1}(B) \ar[d] \\
H_{S^1}^*(pt) \ar[rr]^-{SW_f^\phi} & & H^{*-2d+b_++1}(B)
}
\]
where the vertical maps are the forgetful maps in equivariant cohomology obtained by restricting to the subgroups $\Delta S^1 \subset S^1 \times S^1$ and $\{1 \} \subset S^1$. Moreover, since the map $H_{S^1 \times S^1}^*(pt) \to H_{S^1}^*(pt)$ is surjective, we see that $SW_{S^1 \times S^1,f}^\phi$ completely determines $SW_f^\phi$.

Let us establish notation for various subgroups of $S^1 \times S^1$. Write $S^1_i$ for the subgroup given by the $i$-th copy of $S^1$ and $\Delta S^1$ for the diagonal copy of $S^1$. If we write $S^1$ without any further decoration, it will be understood as the quotient group $(S^1 \times S^1)/\Delta S^1$. We have $H^*_{\Delta S^1}(pt) \cong \mathbb{Z}_2[x]$ and $H^*_{S^1 \times S^1}(pt) \cong \mathbb{Z}_2[x_1,x_2]$, where $x_i$ corresponds to the $i$-th copy of $S^1$. More precisely, $x_i$ is the pullback of the generator of $H^2_{S^1_i}(pt)$. The restriction map $H^*_{S^1 \times S^1}(pt) \to H^*_{\Delta S^1}$ is the map $\mathbb{Z}_2[x_1,x_2] \to \mathbb{Z}_2[x]$ which sends $x_1$ and $x_2$ to $x$. When thinking of $S^1$ as the quotient $(S^1 \times S^1)/\Delta S^1$, we write $H^*_{S^1}(pt) = \mathbb{Z}_2[y]$. Since the quotient map $S^1 \times S^1 \to S^1$ is given by $(a,b) \to ab^{-1}$, it follows that the pullback of $y$ equals $x_1 - x_2$.

Let $(Y^{V,U})^{S^1}$ denote the fixed point set of the $S^1$-action on $Y^{V,U}$ and $\iota : (Y^{V,U})^{S^1} \to Y^{V,U}$ the inclusion. Then $(Y^{V,U})^{S^1}$ is a manifold with boundary and the boundary of $(V^{V,U})^{S^1}$ is the fixed point set of the $S^1$-action on $\partial Y^{V,U}$. It is easily seen that $(Y^{V,U})^{S^1} = F_1 \cup F_2$, where $F_1 = Y^{V_2 , U_1 \oplus U_2}$ and $F_2 = Y^{V_1 , U_1 \oplus U_2}$. Note that $F_1$ and $F_2$ are disjoint. Let $\widetilde{F}_i$ denote the preimage of $F_i$ in $\widetilde{Y}^{V,U}$. Then $\widetilde{F}_i$ is the fixed point set of $S^1_i$ acting on $\widetilde{Y}^{V,U}$. The normal bundle of $\widetilde{F}_i$ in $\widetilde{Y}^{V,U}$ is the pullback of $V_i$ to $\widetilde{F}_i$. Turning this around, the normal bundle of $F_i$ is obtained by taking the normal bundle of $\widetilde{F}_i$ and quotienting by the action of $\Delta S^1$. Since $\Delta S^1$ acts on $V_i$ with weight $+1$, we see that the normal bundle of $F_i$ is $N_i = V_i \otimes L$, where $L \to Y^{V,U}$ is the line bundle associated to the circle bundle $\widetilde{Y}^{V,U} \to Y^{V,U}$. Let $c = c_{1,S^1 \times S^1}(L)$ denote the $S^1 \times S^1$-equivariant Chern class of $L$. The image of $c$ in $\Delta S^1$-equivariant cohomology is $x$. If we restrict $L$ to $F_1$, then $S^1_1$ acts trivially, hence $c|_{F_1} = x_2$. Similarly, $c|_{F_2} = x_1$.

The localisation theorem \cite[III (3.8)]{die} says that the pullback 
\[
\iota^* : H^*_{S^1}( Y^{V,U} , \partial Y^{V,U}) \to H^*_{S^1}( (Y^{V,U})^{S^1} , \partial (Y^{V,U})^{S^1})
\]
is an isomorphism after localising with respect to $y$. The hypotheses of the localisation theorem hold because $S^1$ acts freely on $Y^{V,U} \setminus (Y^{V,U})^{S^1}$. Similarly, the pushforward map $\iota_* : H^*_{S^1}( (Y^{V,U})^{S^1} , \partial (Y^{V,U})^{S^1} ) \to H^{*+n}_{S^1}( Y^{V,U} , \partial Y^{V,U})$ is an isomorphism after localising with respect to $y$ and $n = dim(Y^{V,U}) - dim((Y^{V,U})^{S^1}$. In more detail, $\iota_*$ is given as follows. Let $N$ denote the normal bundle of $(Y^{V,U})^{S^1}$ in $Y^{V,U}$. Let $DN$ be the closed unit disc bundle of $N$ and we identify $DN$ with a closed tubular neighbourhood of $(Y^{V,U})^{S^1}$ in $Y^{V,U}$. Let $\tau_N \in H^n_{S^1}( DN , SN )$ be the equivariant Thom class of $N$. Let $\alpha \in H^*_{S^1}( (Y^{V,U})^{S^1} , \partial (Y^{V,U})^{S^1})$. Let $\pi_N : DN \to (Y^{V,U})^{S^1}$ be the projection map. Then $\pi^*(\alpha) \in H^*_{S^1}( DN , DN|_{\partial (Y^{V,U})^{S^1}})$ and hence $\tau \pi^*(\alpha) \in H^{*+n}_{S^1}( DN , SN \cup DN|_{\partial (Y^{V,U})^{S^1})})$. By excision, $\pi^*(\alpha)\tau$ can be identified with a class in $H^{*+n}_{S^1}( Y^{V,U} , Y^{V,U} \setminus Z)$, where $Z$ is the open unit disc bundle of $N$ restricted to $(Y^{V,U})^{S^1}$ minus a tubular neighbourhood of $\partial (Y^{V,U})^{S^1}$ in $(Y^{V,U})^{S^1}$. Since $\partial Y^{V,U} \subset Y^{V,U} \setminus Z$, we can then map this to a class in $H^{*+n}_{S^1}( Y^{V,U} , \partial Y^{V,U})$, which is $\iota_* \alpha$. This description of $\iota_*(\alpha)$ makes it clear that $\iota^*( \iota_* \alpha ) = e_{S^1}(N) \alpha$.

At this point we should remark that since $F_1$ and $F_2$ will typically have different dimensions, the pushforward does not respect degrees, only the degree mod $2$. In any case, since the map is an isomorphism in the localised rings, there exists a class $\mu \in y^{-1}H^*_{S^1}( (Y^{V,U})^{S^1} , \partial (Y^{V,U})^{S^1})$ of mixed degree such that $\iota_*(\mu) = 1$. Since $\iota^* \iota_*(\mu) = e_{S^1}(N)\mu$, where $N$ denotes the normal bundle of $(Y^{V,U})^{S^1}$ in $Y^{V,U}$ and $e_{S^1}(N)$ is the $S^1$-equivariant Euler class, we must have $\mu = e_{S^1}(N)^{-1}$. We will make this more precise below. 

Let $N_i$ denote $N|_{F_i}$. We have already shown that $N_i = V_i \otimes L$. Identify the quotient group $(S^1 \times S^1)/\Delta S^1$ with $S^1_1$. This acts with weight $+1$ on $V_1$. It acts with weight $-1$ on $V_2$ because $(a,1) \sim (1,a^{-1})$ modulo $\Delta S^1$. Then it follows that
\begin{align*}
e_{S^1}(N_1) &= y^{a_1} + y^{a_1-1}c_1(V_1 \otimes L) + \cdots, \\
e_{S^1}(N_2) &= (-y)^{a_2} + (-y)^{a_2-1}c_1(V_2 \otimes L) + \cdots.
\end{align*}

These formulas can be inverted, giving
\begin{align*}
e_{S^1}(N_1)^{-1} &= y^{-a_1} + y^{-a_1-1}s_1(V_1 \otimes L) + \cdots, \\
e_{S^1}(N_2)^{-1} &= (-y)^{-a_2} + (-y)^{-a_2-1}s_1(V_2 \otimes L) + \cdots
\end{align*}
where $s_j(V_1 \otimes L), s_j(V_2 \otimes L)$ are the Segre classes of $V_1 \otimes L,V_2 \otimes L$. Since $F_i$ is finite dimensional, these are zero for all large enough $j$ and hence the above expressions for $e_{S^1}(N_i)^{-1}$ have only finitely many terms.

For a complex vector bundle $E$ of rank $r$, we have.
\[
c_j( E \otimes L) = \sum_{l = 0}^{j} c_l(E) c_1(L)^{j-l} \binom{r-l}{j-l}, \quad s_j( E \otimes L) = \sum_{l = 0}^{j} s_l(E) c_1(L)^{j-l} \binom{-r-l}{j-l}
\]
In fact, the same expressions hold even when $E$ is a virtual vector bundle. Applying these expressions to $D_1, D_2$, we find
\begin{align*}
e_{S^1}(N_1)^{-1} &= \sum_{j \ge 0} y^{-a_1-j} \sum_{l=0}^j x_1^{j-l} s_l(D_1) \binom{ -d_1 - l}{j-l} \in H^*(F_1)[y,y^{-1}], \\
e_{S^1}(N_2)^{-1} &= \sum_{j \ge 0} (-y)^{-a_1-j} \sum_{l=0}^j x_1^{j-l} s_l(D_1) \binom{ -d_2 - l}{j-l} \in H^*(F_2)[y,y^{-1}].
\end{align*}

Let $\pi_1 : F_1 \to B$, $\pi_2 : F_2 \to B$ be the projections to $B$. The localisation theorem then gives
\begin{equation}\label{equ:loc1}
SW_{S^1 \times S^1,f}^{\phi}(\theta) = (\pi_1)_*( e_{S^1}(N_1)^{-1} \theta f^*(\tau^\phi_{V,U})) + (\pi_2)_*( e_{S^1}(N_2)^{-1} \theta f^*( \tau^\phi_{V,U} )).
\end{equation}
The follows by essentially the same argument used in \cite[\textsection 6.1]{bar2}. Let $j_i : S^{U_i} \to S^{V'_i , U'_i}$ denote the restriction of $f_i$ to $S^{U_i}$. By the assumption that $f_1,f_2$ are monopole maps, we can assume that $j_1,j_2$ are given by inclusions $U_i \to U'_i$. Now to compute $(\pi_1)_*( e_{S^1}(N_1)^{-1} \theta f^*(\tau^\phi_{V,U}))$, note that we are restricting to $S^0 \subseteq S^{V_1}$, so $f$ can be replaced by $j_1 \wedge f_2$.

Let $\iota_1 : S^0 \to S^{V_1 , H^+_1}$ be the inclusion map. Then $j_1 \wedge f_2$ is a suspension of $\iota_1 \wedge f_2 : S^{V_2 , U_2} \to S^{V'_1 , H^+_1} \wedge S^{V'_2 , U'_2}$, so they have the same Seiberg--Witten invariants (as shown in \cite[Proposition 3.8]{bk}). It remains to compute the Seiberg--Witten invariants of $\iota_1 \wedge f_2$ (and similarly $f_1 \wedge \iota_2$).

Recall the chamber $\phi = (\phi_1 , \phi_2)$ and recall that $Z_1$ is the zero locus of $\phi_1$. Recall also that $\phi_2$ is non-vanishing on $Z_1$. After a small perturbation, we may assume that $\phi_2|_{Z_1} : Z_1 \to S^{V'_2,U'_2}$ is transverse to $f_2|_{Z_1} : S^{V_2,U_2} \to S^{V'_2 , U'_2}$ (the fact that this can be done for $S^1$-equivariant monopole maps is explained in \cite[Pages 522-523]{bk}. The same argument also works in the $Pin(2)$-equivariant case, because the stabiliser of any point in $(f_2|_{Z_1})^{-1}( \phi_2|_{Z_1})$ is trivial). Let $\widetilde{\mathcal{M}}_2 \to Z_1$ denote the pre-image $(f_2|_{Z_1})^{-1}( \phi_2|_{Z_1})$ and $\mathcal{M}_2 = \widetilde{\mathcal{M}}_2/S^1$ the quotient. Then $(\iota_1 \wedge f_2)^{-1}(\phi(B)) = \widetilde{\mathcal{M}}_2$. This is a smooth manifold, however it is not cut out transversally. The technique of obstruction bundles (eg. \cite[Section 3]{fm}) can be used to overcome this difficulty. The obstruction bundle on $\widetilde{\mathcal{M}}_2$ is $V'_1$. This descends to the bundle $V'_1 \otimes L$ on $\mathcal{M}_2$. Hence the Seiberg--Witten invariants of $f_2|_{Z_1}$ and $\iota_1 \wedge f_2$ are related by:
\[
SW^\phi_{\iota_1 \wedge f_2}( \theta ) = (j_1)_* SW^{\phi_2}_{f_2|_{Z_1}}( e(V'_1 \otimes L) \theta )
\]
where $j_1 : Z_1 \to B$ is the inclusion map. To apply this to the localisation formula, we need the $S^1 \times S^1$-equivariant extension of this formula,
\[
SW^\phi_{S^1 \times S^1 , \iota_1 \wedge f_2}( \theta ) = (j_1)_* SW^{\phi_2}_{S^1 \times S^1 , f_2|_{Z_1}}( e_{S^1}(V'_1 \otimes L) \theta ).
\]
Some care is required in interpreting the right hand side of this equation. First of all, we can identify $S^1 \times S^1$ with the product $S^1_1 \times \Delta S^1$ via the isomorphism $(a,b) \mapsto (ab^{-1} , b)$. Next, the argument $e_{S^1}(V'_1 \otimes L) \theta$ should be thought of as an element of
\[
y^{-1} H^*_{S^1 \times S^1}( \widetilde{Y}^{V_2,U_2}|_{Z_1}) \cong y^{-1} H^*_{S^1_1 \times \Delta S^1}(\widetilde{Y}^{V_2,U_2}|_{Z_1}) \cong \mathbb{Z}_2[y,y^{-1}] \otimes_{\mathbb{Z}_2} H^*_{\Delta S^1}(\widetilde{Y}^{V_2,U_2}|_{Z_1}).
\]
Let $\psi_2 : y^{-1} H^*_{S^1 \times S^1}(\widetilde{Y}^{V_2,U_2}|_{Z_1}) \cong \mathbb{Z}_2[y,y^{-1}] \otimes_{\mathbb{Z}_2} H^*_{\Delta S^1}(\widetilde{Y}^{V_2,U_2}|_{Z_1})$ denote this isomorphism. Note in particular that $\psi_2(x_1) = y+x$, $\psi_2(x_2) = x$, where $x$ is the generator of $H^2_{\Delta S^1}(pt)$ pulled back to $H^2_{\Delta S^1}(\widetilde{Y}^{V_2,U_2}|_{Z_1})$. From this, it follows that we have a commutative diagram
\[
\xymatrix{
y^{-1} H^*_{S^1 \times S^1}(\widetilde{Y}^{V_2,U_2}|_{Z_1}) \ar[d] \ar[rr]^-{SW^{\phi_2}_{S^1 \times S^1 , f_2|_{Z_1}}} \ar@/_1.8pc/ @<-10ex>[dd]_-{\psi_2} & & y^{-1}H^*_{S^1}(Z_1) \ar[dd]^-{\cong} \\
y^{-1} H^*_{S^1_1 \times \Delta S^1}(\widetilde{Y}^{V_2,U_2}|_{Z_1}) \ar[d] & & \\
\mathbb{Z}_2[y,y^{-1}] \otimes_{\mathbb{Z}_2} H^*_{\Delta S^1}(\widetilde{Y}^{V_2,U_2}|_{Z_1}) \ar[rr]^-{id \otimes SW^\phi_{f_2|_{Z_1}}} & & \mathbb{Z}_2[y,y^{-1}] \otimes_{\mathbb{Z}_2} H^*(Z_1)
}
\]

One similarly defines $\psi_1 : y^{-1} H^*_{S^1 \times S^1}( \widetilde{Y}^{V_1,U_1}|_{Z_2}) \cong \mathbb{Z}_2[y,y^{-1}] \otimes_{\mathbb{Z}_2} H^*_{\Delta S^1}(\widetilde{Y}^{V_1,U_1}|_{Z_2})$ where $\psi_1(x_1) = x$, $\psi(x_2) = -y+x$. Exchanging the roles of $f_1$ and $f_2$ gives a similar formula relating the Seiberg--Witten invariants of $f_1|_{Z_2}$ and $f_1 \wedge \iota_2$. Substituting into (\ref{equ:loc1}) and noting that $e_{S^1}(V_i \otimes L)^{-1} e_{S^1}(V'_i \otimes L) = e_{S^1}(D_i \otimes L)^{-1}$ gives:
\begin{theorem}\label{thm:product1}
For all $\theta \in H^*_{S^1 \times S^1}(B)$, we have
\begin{align*}
SW^{\phi}_{S^1_1 \times S^1_2 , f_1 \wedge f_2}(\theta) &= (j_1)_*\left( id \otimes SW^{\phi_2}_{f_2|_{Z_1}}\right)( \psi_2(e_{S^1}(D_1 \otimes L)^{-1} \theta) ) \\
& \quad \quad + (j_2)_* \left(id \otimes SW^{\phi_1}_{f_1|_{Z_2}}\right)( \psi_1( e_{S^1}(D_2 \otimes L)^{-1} \theta) ).
\end{align*}
\end{theorem}
Some explanation of how to use the formula is required. Here 
\begin{align*}
e_{S^1}(D_1 \otimes L)^{-1} &= y^{-d_1} + y^{-d_1-1}s_1(D_1 \otimes L) + \cdots, \\
e_{S^1}(D_2 \otimes L)^{-1} &= (-y)^{-d_2} + (-y)^{-d_2-1}s_1(D_2 \otimes L) + \cdots,
\end{align*}
where $s_j( D_i \otimes L)$ is the $j$-th Segre class of $D_i \otimes L$ and $L|_{\mathcal{M}_i}$ is the line bundle corresponding to $\widetilde{\mathcal{M}}_i \to \mathcal{M}_i$. Furthermore, since $\psi_i( c_1(L)|_{F_i} ) = \psi_i(x_i) = x$, the Segre classes can be expanded as
\[
\psi_i ( s_j( D_i \otimes L) ) = \sum_{l = 0}^{j} s_l(D_i) x^{j-l} \binom{-d_i-l}{j-l}.
\]

We now consider adapting Theorem \ref{thm:product1} to the case of a smash product of $Pin(2)$-equivariant monopole maps $f_1, f_2$ over a common base $B$. We assume that the two involutions on $B$ corresponding to $f_1$ to $f_2$ commute. In this case the smash product $f = f_1 \wedge_B f_2$ has $Pin(2) \times Pin(2)$-symmetry. Since we are ultimately interested in the diagonal copy of $Pin(2)$, but want to retain the extra circle symmetry we consider the index $2$ subgroup $G \subset Pin(2) \times Pin(2)$ generated by $S^1 \times S^1$ and $(j,j)$. The diagonal circle $\Delta S^1 \subset S^1 \times S^1$ is a normal subgroup of $G$ and $G/\Delta S^1 \cong O(2)$. Carrying out the construction of the Seiberg--Witten invariant of $f$, but with respect to the larger group $G$ gives a map
\[
SW_{G,f}^{\phi} : H^*_{G}(pt) \to H_{O(2)}^{*-2d+b_++1}(B)
\]
compatible with the $Pin(2)$-equivariant Seiberg--Witten invariant in the sense that we have a commutative square
\[
\xymatrix{
H_{G}^*(pt) \ar[rr]^-{SW_{G, f}^\phi} \ar[d] & & H_{O(2)}^{*-2d+b_++1}(B) \ar[d] \\
H_{Pin(2)}^*(pt) \ar[rr]^-{SW_{Pin(2),f}^\phi} & & H_{\mathbb{Z}_2}^{*-2d+b_++1}(B)
}
\]

Let $p_i : G \to Pin(2)$ be the inclusion $G \to Pin(2) \times Pin(2)$, followed by projection to the $i$-th factor and let $p : G \to O(2)$ be the quotient map $G \to G/\Delta S^1 \cong O(2)$. For $i=1,2$, set $q_i = p_i^*(q) \in H^4_{G}(pt)$. Recall that $H^*_{O(2)}(pt) \cong \mathbb{Z}_2[u,y]$ where $deg(u) = 1$, $deg(y) = 2$. Abusing notation, we also write $y \in H^2_{G}(pt)$ for the class $p^*(y)$.

\begin{proposition}\label{prop:q12}
We have $H^*_G(pt) \cong \mathbb{Z}_2[u,y,q_1]/(u^3)$. Furthermore we have $q_2 + q_1 = y^2 + yu^2$.
\end{proposition}
\begin{proof}
We have a short exact sequence $1 \to S^1_1 \to G \buildrel p_1 \over \longrightarrow Pin(2) \to 1$, where $S^1_1$ denotes the subgroup $S^1 \times \{1\} \subset S^1 \times S^1 \subset G$. The Lyndon--Hochschild--Serre spectral sequence for $H^*_G(pt)$ has $E_2^{*,*} = H^*_{Pin(2)}( H^*_{S^1_1}(pt) ) \cong H^*_{Pin(2)}[y']$, where $deg(y') = 2$ is the generator of $H^2_{S^1_1}$. The composition $S^1_1 \to G \buildrel p_1 \over \longrightarrow O(2)$ is easily seen to be the inclusion of the circle subgroup of $O(2)$. This implies that $p^*(y) \in H^2_G(pt)$ restricts to $y'$. This implies that the spectral sequence degenerates and $H^*_G(pt) \cong H^*_{Pin(2)}[y] \cong \mathbb{Z}_2[u,y,q_1]/(u^3)$.

It remains to prove the relation $q_1 + q_2 = y^2 + yu^2$. Since $H^4_{G}(pt)$ is spanned by $q_1,y^2,yu^2$, we have that $q_2$ is a linear combination of $q_1, y^2, yu^2$. Consider the subgroup $S^1 \times S^1 \subset G$ which has cohomology $\mathbb{Z}_2[x_1,x_2]$. The restriction map $H^*_{G}(pt) \to H^*_{S^1 \times S^1}$ sends $q_i$ to $x_i^2$, $y$ to $x_1+x_2$ and $u$ to zero. This shows that $q_2$ must be either $q_1 + y^2$ or $q_1 + y^2 + yu^2$. 

Next, we note that $Pin(2)$ acts freely on $S^3 \cong SU(2)$ via the inclusion $Pin(2) \to SU(2)$ and the left action of $SU(2)$ on itself. The quotient space is $\mathbb{RP}^2$, because the quotient of $S^3$ by the $S^1$-subgroup of $Pin(2)$ is $S^3/S^1 \cong S^2$ and the remaining action of $Pin(2)/S^1 \cong \mathbb{Z}_2$ acts on $S^2$ as the antipodal map.

We also have that $G$ acts freely on $S^3 \times S^3$ via the inclusion $G \to Pin(2) \times Pin(2)$ and the obvious product action of $Pin(2) \times Pin(2)$ on $S^3 \times S^3$. Let $M = (S^3 \times S^3)/G$ be the quotient. Clearly $M$ is the quotient of $S^2 \times S^2$ by the involution which acts as the antipodal map on both factors. Projecting to either factor of $S^2$ gives two fibrations
\[
S^2 \to M \buildrel \pi_i \over \longrightarrow \mathbb{RP}^2, \quad i = 1,2.
\]
Both fibrations admit a section $s_i : \mathbb{RP}^2 \to M$ which the image under the quotient map $(S^2 \times S^2) \to M$ of the diagonal copy of $S^2$. From this it follows easily that the Leray--Serre spectral sequences for both fibrations degenerate at $E_2$. Then $H^*_G( S^3 \times S^3) \cong H^*(M)$ is a $6$-dimensional space over $\mathbb{Z}_2$ with basis $1,u,u^2,c,cu,cu^2$, where $c \in H^2(M)$ restricts non-trivially to the fibres of $\pi_1 : M \to \mathbb{RP}^2$. The diagonal $S^2 \to S^2 \times S^2$ has normal bundle $TS^2$ and taking the quotient by the antipodal map on both factors, it follows that the normal bundle of the sections $s_1,s_2$ are both equal to $T\mathbb{RP}^2$. Since $w_2( T\mathbb{RP}^2) = u^2$, it follows that the mod $2$ self-intersection of $s_i$ is odd, but this can only happen if $c^2 \neq 0$, so $c^2 = cu^2$.

Let $\mathbb{H}_1$ be the standard representation of the $i$-th copy of $SU(2)$ in the product $SU(2) \times SU(2)$. By restriction, this defines a representation of $G$ and we have that $q_i = w_4( \mathbb{H}_i)$. By taking the pullback of $\mathbb{H}_i$ to $S^3 \times S^3$ and quotienting by $G$, we have that $\mathbb{H}_i$ descends to a rank $4$ vector bundle $\widetilde{\mathbb{H}}_i \to M$. In fact, $\widetilde{\mathbb{H}}_i$ is the pullback under $\pi_i : M \to \mathbb{RP}^2$ of a rank $4$ vector bundle on $\mathbb{RP}^2$, because the $G$-action on $\mathbb{H}_i$ factors through $p_i : G \to Pin(2)$. Now since $\mathbb{RP}^2$ is $2$-dimensional, we must have that $w_4(\widetilde{\mathbb{H}}_i) = 0$. So the pullback of $q_1,q_2$ to $H^4_{G}(S^3 \times S^3) \cong H^4(M)$ are both zero.

Let $R$ denote the standard $2$-dimensional real representation of $O(2)$. We have that $y = w_2(R)$. By taking the pullback of $R$ to $S^3 \times S^3$ and quotienting by $G$, we obtain a vector bundle $\widetilde{R} \to M$ on $M$. The restriction of $\widetilde{R}$ to any fibre of $\pi_2$ is isomorphic to $\mathcal{O}(1) \to S^2$. In particular, this means that $y = w_2(\widetilde{R})$ equals either $c$ or $c+u^2$. In either case, we then have $y^2 = cu^2 = yu^2$. So $y^2 + yu^2 = 0$, but $y^2 \neq 0$ in $H^*(M)$. Then since $q_1 + q_2 = 0$ in $H^*(M)$, we see that we must have $q_1 + q_2 = y^2 + yu^2$.

\end{proof}

The above proposition implies in particular that $H^*_G(pt) \to H^*_{Pin(2)}(pt)$ is surjective, hence $SW_{Pin(2),f}^{\phi}$ is completely determined by $SW_{G,f}^{\phi}$.

Recall that we have homomorphims $p_i : G \to Pin(2)$ for $i=1,2$ and also $p : G \to O(2)$.

\begin{lemma}\label{lem:tensor}
Let $M$ be a space on which $Pin(2)$ acts. Regard $M$ as a $G$-space via $p_i : G \to Pin(2)$. Then we have an isomorphism
\[
p^* \otimes p_i^* : H^*_{O(2)}(pt) \otimes_{H^*_{\mathbb{Z}_2}(pt)} H^*_{Pin(2)}(M) \cong H^*_G(M).
\]
\end{lemma}
\begin{proof}
By symmetry, it suffices to prove this for $i=2$. Since $S_1^1$ acts trivially on $M$, we get a fibration $BS^1_1 \to M_G \to M_{Pin(2)}$ and a Leray--Serre spectral sequence for $H^*_G(M)$ which has $E_2 = H_{Pin(2)}^*(M)[y]$, where $deg(y) = 2$.

The composition $S^1_1 \to G \to O(2)$ is the inclusion $S^1_1 \to O(2)$. Next, we observe that $H^*_{O(2)}(pt) \cong \mathbb{Z}_2[u,y]$, where $deg(u) = 1$, $deg(y)=2$. Moreover the pullback of $y$ to $H^2_{S^1}(pt)$ is a generator. Since $y \in H^2_{O(2)}(pt)$ can be pulled back to a class in $H^2_G(M)$ whose restriction to the fibre is a generator of $H^2( BS^1)$, it follows that the spectral sequence degenerates at $E_2$ and we have an isomorphism
\[
H^*_{O(2)}(pt) \otimes_{H^*_{\mathbb{Z}_2}(pt)} H^*_{Pin(2)}(M) \cong H^*_{Pin(2)}(M)[y] \cong H^*_G(M)
\]
and that this isomorphism is realised by the map $p^* \otimes p_i^*$.
\end{proof}

Since $H^*_{\mathbb{Z}_2}(pt) \cong \mathbb{Z}_2[u]$ and $H^*_{O(2)}(pt) \cong \mathbb{Z}_2[u,y]$, where $deg(y) = 2$, Lemma \ref{lem:tensor} yields an isomorphism
\[
\psi_i : y^{-1}H^*_G(B) \cong \mathbb{Z}_2[y,y^{-1}] \otimes_{\mathbb{Z}_2} H^*_{Pin(2)}(B) \cong H^*_{Pin(2)}(B)[y,y^{-1}].
\]
Furthermore, $\psi_i$ is a morphism of $H^*_{\mathbb{Z}_2}(pt)$-modules. We have $\psi_i( q_i ) = q$. Furthermore, Proposition \ref{prop:q12} implies that $\psi_1(q_2) = q + y^2 + yu^2$ and $\psi_2(q_1) = q + y^2 + yu^2$. To simplify notation, we define
\[
\mu = y^2 + yu^2
\]
so that $\psi_1(q_2) = \psi_2(q_1) = q+\mu$. Since $y^4 = \mu^2$, there is essentially no difference between localising with respect to $y$ or with respect to $\mu$.

Repeating the localisation argument in $G$-equivariant cohomology gives:
\begin{theorem}\label{thm:product2}
For all $\theta \in H^*_{G}(B)$, we have
\begin{align*}
SW^{\phi}_{G , f_1 \wedge f_2}(\theta) &= (j_1)_*\left( id \otimes SW^{\phi_2}_{Pin(2) , f_2|_{Z_1}}\right)( \psi_2(e_{G}(D_1)^{-1}\theta) ) \\
& \quad \quad + (j_2)_* \left(id \otimes SW^{\phi_1}_{Pin(2) , f_1|_{Z_2}}\right)( \psi_1( e_{G}(D_2)^{-1} \theta) ).
\end{align*}
\end{theorem}

The Euler classes on the right hand side of the formula should be understood as follows. First, $V_i, V'_i$ are $Pin(2)$-equivariant bundles over $B$. Then $V_i, V'_i$ can be regarded as a $G$-equivariant vector bundles via the homomorphism $p_i : G \to Pin(2)$. Then $e_{G}(V_i), e_{G}(V'_i)$ are the images of $e_{Pin(2)}( V_i ), e_{Pin(2)}(V'_i)$ under the map $p_i : H^*_{Pin(2)}(B) \to H^*_G(B)$ induced by $p_i$. We have that $e_{G}(V_1)$ is invertible in $y^{-1}H^*_G(\widetilde{Y}^{V_1,U_1}|_{Z_2})$ and so $e_{G}(D_1)^{-1} =  e_{G}(V_1)^{-1} e_{G}(V'_1)$ is defined. Similarly $e_{G}(D_2)^{-1}$ is defined.

We will mainly be interested in applying the product formula in situations where $B$ satisfies the following assumptions:
\begin{itemize}
\item[(1)]{$B$ is a fibre bundle $B \to B_0$ such that $j : B \to B$ covers an involution $j_0 : B_0 \to B_0$.}
\item[(2)]{$j_0$ does not act freely.}
\item[(3)]{The bundles $V_i, V'_i$ are pullbacks from $B_0$ and $j_0$ lifts to an antilinear endomorphism on $V_i, V'_i$ squaring to $-1$.}
\item[(4)]{The map $u : H^*_{\mathbb{Z}_2}(B_0) \to H^{*+1}_{\mathbb{Z}_2}(B_0)$ is injective.}
\end{itemize}

For instance, in the case of the $Pin(2)$-monopole map of a $4$-manifold $X$ with spin-structure $\mathfrak{s}$, we have $B = B_{X,\mathfrak{s}} = Pic^{\mathfrak{s}}(X) \times S(H^+(X))$. Then the above assumptions are satisfied if we take $B_0 = Pic^{\mathfrak{s}}(X)$ and $B \to B_0$ the projection.

Note that condition (4) actually implies condition (2), for if the action of $j_0$ were free, then $H^*_{\mathbb{Z}_2}(B_0)$ would be zero in degrees above $dim(B_0)$.

\begin{lemma}\label{lem:degen}
Let $j_0 : B_0 \to B_0$ be an involution satisfying condition (4) above. Then $j_0$ acts trivially on $H^*(B_0)$ and the Leray--Serre spectral sequence for the Borel fibration $(B_0)_{\mathbb{Z}_2} \to B\mathbb{Z}_2$ degenerates at $E_2$, giving a (non-canonical) isomorphism $H^*_{\mathbb{Z}_2}(B_0) \cong H^*(B_0)[u]$ of $\mathbb{Z}_2[u]$-modules.
\end{lemma}
\begin{proof}
Suppose that $A$ is a finite dimensional representation of $\mathbb{Z}_2$ over $\mathbb{Z}_2$. Any such $A$ is a direct sum of copies of the trivial representation $\mathbb{Z}_2$ and the regular representation $R = \mathbb{Z}_2^2$. Since $H^*( \mathbb{Z}_2 ; \mathbb{Z}_2) \cong \mathbb{Z}_2[u]$, $H^0( \mathbb{Z}_2 ; R) \cong R^{\mathbb{Z}_2} \cong \mathbb{Z}_2$ and $H^p( \mathbb{Z}_2 ; R) = 0$ for $p>0$, it follows that $u : H^p( \mathbb{Z}_2 ; A) \to H^{p+1}( \mathbb{Z}_2 ; A)$ is surjective for all $p$ and an isomorphism for $p > 0$. It also follows that $u : H^0( \mathbb{Z}_2 ; A) \to H^1(\mathbb{Z}_2 ; A)$ is injective if and only if $A^{\mathbb{Z}_2} = A$.

Now consider the Leray--Serre spectral sequence for the Borel fibration $(B_0)_{\mathbb{Z}_2} \to B\mathbb{Z}_2$. This has $E^{p,q}_2 = H^p( \mathbb{Z}_2 ; H^q(B_0) )$. Injectivity of $u : H^*_{\mathbb{Z}_2}(B_0) \to H^{*+1}_{\mathbb{Z}_2}(B_0)$ implies that the map $H^*_{\mathbb{Z}_2}(pt) \to H^*_{\mathbb{Z}_2}(B_0)$ is an injection. Hence for all $r$ there are no differentials mapping into the $q=0$ row of $E_r$.

Consider the map $u : E_\infty^{0,1} \to E_\infty^{1,1}$. If this map is not injective, then $u : H^1_{\mathbb{Z}_2}(B_0) \to H^2_{\mathbb{Z}_2}(B_0)$ will also not be injective. So $u : E_\infty^{0,1} \to E_\infty^{1,1}$ is injective. But $E_\infty^{0,1} \cong H^0( \mathbb{Z}_2 ; H^1(B_0))$, $E_\infty^{1,1} \cong H^1( \mathbb{Z}_2 ; H^1(B_0))$. For this map to be injective, the action of $\mathbb{Z}_2$ on $H^1(B_0)$ must be trivial. This means $E_2^{p,1} \cong H^1(B_0)$ for all $p$. Injectivity of $u : H^*_{\mathbb{Z}_2}(B_0) \to H^{*+1}_{\mathbb{Z}_2}(B_0)$ then implies that there can be no differentials mapping into the $q=1$ row of $E_r$ for any $r$.

Continuing row by row in the same manner, we see that the action of $\mathbb{Z}_2$ on $H^q(B_0)$ is trivial for all $q$ and that there are no differentials in the spectral sequence. This gives the result.
\end{proof}

\begin{lemma}\label{lem:q}
Let $j_0 : B_0 \to B_0$ be an involution and suppose that conditions (2) and (4) in the above list are satisfied. The homomorphism $Pin(2) \to Pin(2)/S^1 = \mathbb{Z}_2$ induces a natural map $H^*_{\mathbb{Z}_2}(B_0) \to H^*_{Pin(2)}(B_0)$ which makes $H^*_{Pin(2)}(B_0)$ into a $H^*_{\mathbb{Z}_2}(B_0)$ module, and with respect to this module structure we have a canonical isomorphism $H^*_{Pin(2)}(B_0) \cong H^*_{\mathbb{Z}_2}(B_0)[q]/(u^3)$.
\end{lemma}
\begin{proof}
First of all, Lemma \ref{lem:degen} gives an isomorphism $H^*_{\mathbb{Z}_2}(B_0) \cong H^*(B_0)[u]$. Next since $S^1 \subset Pin(2)$ acts trivially on $B_0$, we see that the Borel model $(B_0)_{Pin(2)}$ is given by
\[
(B_0)_{Pin(2)} = (B_0 \times EPin(2))/Pin(2) \cong (B_0 \times EPin(2)/S^1)/\mathbb{Z}_2.
\]
Then since $EPin(2)/S^1 \cong BS^1$, we get a fibration $BS^1 \to (B_0)_{Pin(2)} \to (B_0)_{\mathbb{Z}_2}$. In particular, this makes $H^*_{Pin(2)}(B_0)$ into a $H^*_{\mathbb{Z}_2}(B_0)$-module. Furthermore the associated Leray--Serre spectral sequence has $E_2 \cong H^*_{\mathbb{Z}_2}(B_0)[x] \cong H^*(B_0)[u,x]$, where $deg(x) = 2$. By condition (2), there exists a fixed point $b \in B_0$. Since $x$ has even degree, the spectral sequence has no differentials for even $r$ and so $E_2 \cong E_3$. Restricting the fibration $(B_0)_{Pin(2)} \to (B_0)_{\mathbb{Z}_2}$ to $\{ b \} \subseteq B_0$, we see that $d_3(x) = u^3 + \cdots $, where $\cdots $ denotes terms involving lower powers of $u$. But we also know that $u^3 = 0$ and since there are no other differentials mapping to $E_r^{3,0}$, the only way this can happen is if there are no lower powers of $u$ in $d_3(x)$. So $d_3(x) = u^3$. It follows that $E_5 \cong H^*(B_0)[u,q]/(u^3) \cong H^*_{\mathbb{Z}_2}(B_0)[q]/(u^3)$, where $q = x^2$. There can be no further differentials since $q$ can be identified with the pullback of $q \in H^4_{Pin(2)}(pt)$.
\end{proof}

Suppose $B_0$ satisfies conditions (2) and (4). Let $E \to B_0$ be any complex vector bundle on $B_0$ and suppose there is an antilinear lift of $j_0$ to $E$ squaring to $-1$. This makes $E$ into a $Pin(2)$-equivariant vector bundle. Suppose $E$ has complex rank $a$. The fibre of $E$ over a fixed point of $j_0$ has a quaternionic structure, so $a$ is even. Using the canonical isomorphism of Lemma \ref{lem:q}, it follows that $e_{Pin(2)}(E)$ can be uniquely expressed in the form
\[
e_{Pin(2)}(E) = \alpha_0(E) q^{a/2} + \alpha_1(E) q^{a/2-1} + \cdots + \alpha_{a/2}(E),
\]
where $\alpha_j(E) \in H^{4j}_{\mathbb{Z}_2}(B_0)/(u^3)$. Consider the restriction to $S^1 \subset Pin(2)$. Since $S^1$ acts trivially on $B_0$, we have $H^*_{S^1}(B_0) \cong H^*(B_0)[x]$, where $deg(x) = 2$. Under the forgetful map $H^*_{Pin(2)}(B_0) \to H^*_{S^1}(B_0)$, $u$ is sent to zero and $q$ is sent to $x^2$. On the other hand, we have
\[
e_{S^1}(E) = x^a + c_1(E)x^{a-1} + \cdots + c_{a}(E).
\]
This implies that the image of $\alpha_j(E)$ under the map $H^*_{\mathbb{Z}_2}(B_0)/(u^3) \to H^*(B_0)$ is $c_{2j}(E)$. It also implies that the odd Chern classes of $E$ are zero mod $2$. Because of this, we will denote $\alpha_j(E)$ by $c_{2j , \mathbb{Z}_2}(E)$ and refer to them as the $\mathbb{Z}_2$-equivariant Chern classes of $E$. Note that this terminology is somewhat innacurate. Since the lift of $j_0$ to $E$ squares to $-1$, we do not have a $\mathbb{Z}_2$-action on $E$ and so we do not have $\mathbb{Z}_2$-equivariant Chern classes in the usual sense. Note also that we have only defined the classes $c_{2j,\mathbb{Z}_2}(E)$ in cohomology with $\mathbb{Z}_2$-coefficients.
 
Having defined the classes $c_{2j,\mathbb{Z}_2}(E)$, one can also define equivariant Segre classes $s_{2j,\mathbb{Z}_2}(E) \in H^*_{\mathbb{Z}_2}(B_0)/(u^3)$, characterised by the relation $c_{\mathbb{Z}_2}(E) s_{\mathbb{Z}_2}(E) = 1$, where $c_{\mathbb{Z}_2}(E) = 1 + c_{2,\mathbb{Z}_2}(E) + \cdots $ and $s_{\mathbb{Z}_2}(E) = 1 + s_{2,\mathbb{Z}_2}(E) + \cdots $ are the $\mathbb{Z}_2$-equivariant total Chern and total Segre classes. These classes are stable and so we can also define the $\mathbb{Z}_2$-equivariant Chern and Segre classes of a virtual bundle $E_1 - E_2$, where $E_1,E_2$ are complex vector bundles on $B_0$ both admitting an antilinear lift of $j_0$ squaring to $-1$.

Suppose now that assumptions (1)-(4) hold. Then by the above discussion, $e_{Pin(2)}(V_i)$ has the form
\[
e_{Pin(2)}( V_i ) = \sum_{j=0}^{a_i/2} q^{a_i/2 - j} c_{2j , \mathbb{Z}_2}( V_i )
\]
and similarly for $e_{Pin(2)}(V'_i)$. Since $p_1^*(q) = q_1$ and $\psi_2( q_1 ) = q + \mu$, we have
\begin{align*}
\psi_2( e_{G}(V_1) ) &= \sum_{j=0}^{a_1/2} (q+\mu)^{a/2-j} c_{2j , \mathbb{Z}_2}(V_1) \\
&= \sum_{j=0}^{a_1/2} \sum_{l=0}^{a_1/2-j} \binom{a_1/2 - j}{l} \mu^{a_1/2 - j - l} q^l c_{2j, \mathbb{Z}_2}(V) \\
&= \sum_{k=0}^{a_1/2} \mu^{a_1/2 - k } \sum_{j=0}^{k} q^{k-j} c_{2j , \mathbb{Z}_2 }(V) \binom{ a/2 - j}{k-j}.
\end{align*}

Similar expressions hold for $e_{Pin(2)}(V'_i)$. From this it follows that 
\[
\psi_2( e_{G}(D_1) )^{-1} = \sum_{k \ge 0} \mu^{-d_1/2 - k} \sum_{j=0}^k q^{k-j}  s_{2j , \mathbb{Z}_2}(D_1) \binom{ -d_1/2 - j}{k-j}
\]
and similarly
\[
\psi_1( e_{G}(D_2) )^{-1} = \sum_{k \ge 0} \mu^{-d_2/2 - k}  \sum_{j=0}^k q^{k-j}  s_{2j , \mathbb{Z}_2}(D_2) \binom{ -d_2/2 - j}{k-j}.
\]
Although these expressions appear to be infinite sums, they should really be regarded as elements of $H^*_{Pin(2)}(F_i)[y,y^{-1}]$. But $Pin(2)$ acts freely on $F_i$, so only finitely many terms are non-zero when pulled back to $F_i$.

\section{Seiberg--Witten invariants for spin structures}\label{sec:swspin}

In this section we will use the product formula to compute $SW^{Pin(2)}_{X,\mathfrak{s}}$ for any compact, oriented, smooth $4$-manifold and any spin-structure $\mathfrak{s}$. The connected sum formula for the Bauer--Furuta invariant \cite{b2} says that the Seiberg--Witten monopole map $f_{X \# Y}$ for a connected sum $X \# Y$ is the external smash product of the monopole maps $f_X, f_Y$ for $X$ and $Y$. In other words $f_{X\#Y}$ is obtained by pulling back $f_X$ and $f_Y$ to the product $Pic^{\mathfrak{s}_X}(X) \times Pic^{\mathfrak{s}_Y}(Y)$ and then taking the fibrewise smash product.

To simplify notation we will often omit mention of the spin structures $\mathfrak{s}_X, \mathfrak{s}_Y$. Write $\widehat{f}_X$ for the pullback of $f_X$ to $B_{X} = Pic^{\mathfrak{s}_X}(X) \times S(H^+(X))$ and similarly define $\widehat{f}_Y$ and $\widehat{f}_{X \# Y}$. Let $\phi = (\phi_1 , \phi_2) : B_{X \# Y } \to S( H^+(X) \oplus H^+(Y) )$ be the tautological chamber for $X \# Y$. The zero loci $Z_1, Z_2 \subset B_{X \# Y}$ are $Z_1 = Pic^{\mathfrak{s}_X}(X) \times B_Y$ and $Z_2 = B_X \times Pic^{\mathfrak{s}_Y}(Y)$, moreover $\phi_1|_{Z_2}$ is the tautological chamber for $X$ (pulled back under the projection $Z_2 \to B_X)$ and $\phi_2|_{Z_1}$ is the tautological chamber for $Y$ (pulled back under $Z_1 \to B_Y$).

Let $j_i : Z_i \to B_{X \# Y}$ be the inclusions. Recall by Proposition \ref{prop:z2cohom} that we have isomorphisms 
\begin{align*}
H^*_{\mathbb{Z}_2}(B_{X}) &\cong H^*_{\mathbb{Z}_2}( Pic^{\mathfrak{s}_X}(X))/( u^{b_+(X)}) \\
H^*_{\mathbb{Z}_2}(B_{Y}) &\cong H^*_{\mathbb{Z}_2}( Pic^{\mathfrak{s}_Y}(X))/( u^{b_+(Y)}) \\
H^*_{\mathbb{Z}_2}(B_{X \# Y}) &\cong H^*_{\mathbb{Z}_2}( Pic^{\mathfrak{s}_X}(X) \times Pic^{\mathfrak{s}_Y}(Y))/( u^{b_+(X)+b_+(Y)}).
\end{align*}
By a similar argument, we have isomorphisms
\begin{align*}
H^*_{\mathbb{Z}_2}(Z_1) &\cong H^*_{\mathbb{Z}_2}( Pic^{\mathfrak{s}_X}(X) \times Pic^{\mathfrak{s}_Y}(Y) )/( u^{b_+(Y)}) \\
H^*_{\mathbb{Z}_2}(Z_2) &\cong H^*_{\mathbb{Z}_2}( Pic^{\mathfrak{s}_X}(X) \times Pic^{\mathfrak{s}_Y}(Y) )/( u^{b_+(X)}).
\end{align*}

It is straightforward to see that the push-forward map
\[
(j_1)_* : H^*_{\mathbb{Z}_2}( Pic^{\mathfrak{s}_X}(X) \times Pic^{\mathfrak{s}_Y}(Y) )/( u^{b_+(Y)})  \to H^*_{\mathbb{Z}_2}( Pic^{\mathfrak{s}_X}(X) \times Pic^{\mathfrak{s}_Y}(Y) )/( u^{b_+(X)+b_+(Y)}) 
\]
is multiplication by $u^{b_+(X)}$ and similarly $(j_2)_*$ is multiplication by $u^{b_+(Y)}$. Therefore, Theorem \ref{thm:product2} takes the form:

\begin{align*}
SW^{G}_{X \# Y}(\theta) &=  u^{b_+(X)} \left( id \otimes SW^{Pin(2)}_Y \right)( \psi_2(e_{G}(D_X)^{-1}\theta) ) \\
& \quad \quad + u^{b_+(Y)} \left(id \otimes SW^{Pin(2)}_X \right)( \psi_1( e_{G}(D_Y)^{-1} \theta) ).
\end{align*}

\begin{lemma}\label{lem:j}
Let $X$ be a compact, oriented, smooth $4$-manifold with $b_1(X) = 0$ and let $\mathfrak{s}$ be a spin structure on $X$. Then $SW^{Pin(2)}_{X,\mathfrak{s}}(q^j) = 0$ unless $j = -\sigma(X)/16 - 1$. In particular, $SW^{Pin(2)}_{X,\mathfrak{s}} = 0$ unless $\sigma(X) < 0$.
\end{lemma}
\begin{proof}
Observe that $SW^{Pin(2)}_{X,\mathfrak{s}}(q^j ) \in H^*( \mathbb{RP}^{b_+(X)-1} ) \cong \mathbb{Z}_2[u]/(u^{b_+(X)})$ has degree $4j + \sigma(X)/4 + b_+(X)+1$. Furthermore, since $u^3 = 0$ in $H^*_{Pin(2)}(pt)$, we have $u^3 SW^{Pin(2)}_{X,\mathfrak{s}}(q^j) = SW^{Pin(2)}_{X,\mathfrak{s}}(u^3 q^j) = 0$. Hence $SW^{Pin(2)}_{X,\mathfrak{s}}(q^j)$ is zero unless $b_+(X)-3 \le 4j+\sigma(X)/4 + b_+(X)+1 \le b_+(X)-1$. The only value of $j$ satisfying this is $j=-\sigma(X)/16-1$.
\end{proof}

For $m \ge 1$, let $mK3$ denote the connected sum of $m$ copies of the $K3$ surface. This has a unique spin structure $\mathfrak{s}$ and so we will write $SW^{Pin(2)}_{mK3}$ for $SW^{Pin(2)}_{mK3,\mathfrak{s}}$.

\begin{lemma}\label{lem:mk3}
For any $m \ge 1$, we have
\[
SW^{Pin(2)}_{mK3}( q^j) = \begin{cases} u^{3m-3} & j = m-1, \\ 0 & \text{otherwise}. \end{cases}
\]
\end{lemma}
\begin{proof}
By Lemma \ref{lem:j}, $SW^{Pin(2)}_{mK3}(q^j)$ is zero unless $j=m-1$. So it remains to compute $SW^{Pin(2)}_{mK3}(q^{m-1})$ for each $m \ge 1$. We will prove the result by induction. In the case $m=1$, $SW^{Pin(2)}_{K3}(1)$ has degree zero and by Lemma \ref{lem:u=0}, $SW^{Pin(2)}_{K3}(1) = SW_{K3}(1)$ is the ordinary mod $2$ Seiberg--Witten invariant of $K3$, which is $1$.

Now assume $SW^{Pin(2)}_{mK3}(q^{m-1}) = u^{3m-3}$ for some $m \ge 1$. Writing $(m+1)K3 = K3 \# mK3$, we can apply Theorem \ref{thm:product2}. We have
\[
\psi_1 (e_G(D_1)^{-1}) = (\mu + q)^{-1}, \quad \psi_2 (e_G(D_2)^{-1} ) = (\mu + q)^{-m}.
\]
Then, taking $\theta = q_1^{m}$, we find
\[
SW^{G}_{K3 \# mK3}( q_1^{m}) = u^3 SW_{mK3}^{Pin(2)}( (\mu+q)^{m-1} ) + u^{3m} SW_{K3}^{Pin(2)}(  (\mu+q)^{-m} q^m ).
\]
Since $(\mu+q)^{-m}q^m$ is a multiple of $q$, the second term drops out. Expanding $(\mu+q)^{m-1}$ and using the inductive hypothesis then gives
\[
SW^{G}_{K3 \# mK3}( q_1^{m}) = u^{3m}.
\]
Lastly, the forgetful map $H^*_{G}(pt) \to H^*_{Pin(2)}(pt)$ sends $q_1$ to $q$ and hence \linebreak $SW^{Pin(2)}_{(m+1)K3}(q^m) = u^{3m}$, which completes the inductive step.
\end{proof}

\begin{theorem}\label{thm:swpin1}
Let $X$ be a compact, oriented, smooth $4$-manifold with $b_+(X) > 0$ and let $\mathfrak{s}$ be a spin-structure on $X$. If $b_+(X) \ge 3$, then
\[
SW^{Pin(2)}_{X,\mathfrak{s}}( q^j ) = u^{b_+(X)-3} s_{2(j+1+\sigma(X)/16) , \mathbb{Z}_2}(D)
\]
where we set $s_{l,\mathbb{Z}_2}(D) = 0$ if $l < 0$.

If $b_+(X) < 3$, then $s_{2k , \mathbb{Z}_2}(D) \in H^*_{\mathbb{Z}_2}( Pic^{\mathfrak{s}}(X) )/(u^3)$ is divisible by $u^{3-b_+(X)}$ for all $k>0$. Then we can write $s_{2k , \mathbb{Z}_2}(D) = u^{3-b_+(X)} \alpha_k$ for a uniquely determined class $\alpha_k \in H^*_{\mathbb{Z}_2}(Pic^{\mathfrak{s}}(X))/(u^{b_+})$ and we have
\[
SW^{Pin(2)}_{X,\mathfrak{s}}( q^j ) = \alpha_{(j+1)}.
\]
\end{theorem}
\begin{proof}
Choose an $m > 0$ such that $16m \ge 4b_1(X) - \sigma(X)$. Let $M = X \# 22m(S^2 \times S^2)$ and let $\mathfrak{s}_M = (\mathfrak{s} , \mathfrak{s}_0)$, where $\mathfrak{s}_0$ is the unique spin-structure on $22m(S^2 \times S^2)$. We will compute $SW^{Pin(2)}_{M,\mathfrak{s}_M}$ in two different ways. First note that $SW^{Pin(2)}_{M , \mathfrak{s}_M}$ takes values in $H^*_{\mathbb{Z}_2}( Pic^{\mathfrak{s}}(X) \times S^{b_+(X)+22m-1} ) \cong H^*( Pic^{\mathfrak{s}}(X) )[u]/( u^{b_+(X)+22m})$, by Proposition \ref{prop:z2cohom}.

We write $M$ as the connected sum $M = X \# 22m(S^2 \times S^2)$ and apply Theorem \ref{thm:product2}. By Lemma \ref{lem:j}, $SW^{Pin(2)}_{22m(S^2 \times S^2) , \mathfrak{s}_0}(q^j) = 0$ for all $j$ and hence
\[
SW^{G}_{X \# 22m(S^2 \times S^2) , \mathfrak{s} \# \mathfrak{s}_0}(q_1^j) = u^{22m} SW^{Pin(2)}_{X , \mathfrak{s}}( q^j ).
\]
Restricting to $Pin(2) \subset G$, this reduces to
\begin{equation}\label{equ:swm1}
SW^{Pin(2)}_{M , \mathfrak{s}_M}(q^j) = u^{22m} SW^{Pin(2)}_{X , \mathfrak{s}}(q^j).
\end{equation}
Next, we recall that $22(S^2 \times S^2)$ is diffeomorphic to $K3 \# \overline{K3}$ \cite{gom}. Hence we can write $M = (X \# m\overline{K3}) \# mK3 = M_1 \# mK3$, where $M_1 = X \# m\overline{K3}$ and apply the product formula to this decomposition of $M$. Write $\mathfrak{s}_1$ for the spin-structure on $M_1$ which is the connected sum of $\mathfrak{s}$ with the unique spin-structure on $\overline{K3}$ and write $\mathfrak{s}_2$ for the unique spin-structure on $mK3$. 

We claim that $SW_{M_1,\mathfrak{s}_1}^{Pin(2)}(q^j) = 0$ for all $j \ge 0$. In fact, $SW_{M_1,\mathfrak{s}_1}^{Pin(2)}(q^j) \in H^*_{\mathbb{Z}_2}( Pic^{\mathfrak{s}}(X) \times S^{19m+b_+(X)-1} )$ has degree 
\begin{align*}
4j +(\sigma(X)+16m)/4 &  + 19m+b_+(X)+1 \\
& \quad \quad \ge \sigma(X)/4 + 4m+19m+b_+(X)+1 \\ & \quad \quad \ge b_1(X)+19m+b_+(X)+1,
\end{align*}
by the assumption that $16m \ge 4b_1(X) - \sigma(X)$. But $H^*( Pic^{\mathfrak{s}}(X) \times S^{19m+b_+(X)-1})$ is non-zero only in degree at most $b_1(X)+19m+b_+(X)-1$, which proves the claim. Hence Theorem \ref{thm:product2} gives:
\[
SW^{G}_{M_1 \# mK3 , \mathfrak{s}_1 \# \mathfrak{s}_2}(q_2^j) = u^{19m+b_+(X)} SW_{mK3}^{Pin(2)}( \psi_2( e_G(D_{M_1})^{-1} q_2^j ) ).
\]
Next, using $D_{M_1} = D_X - \mathbb{C}^{2m}$ (by the excision property of the families index of the Dirac operator), we have that
\begin{align*}
\psi_2( e_G(D_{M_1})^{-1} q_2^j ) &= \sum_{k \ge 0} \mu^{m-d_X/2-k} \sum_{l=0}^{k} q^l s_{2(k-l),\mathbb{Z}_2}(D_{M_1}) \binom{m-d_X/2 - k + l}{l} q^j \\
&= \sum_{k \ge 0} \mu^{m-d_X/2-k} \sum_{l=0}^{k} q^l s_{2(k-l),\mathbb{Z}_2}(D_{X}) \binom{m-d_X/2 - k + l}{l} q^j \\
&= \sum_{k \ge 0} \mu^{m-d_X/2-k} \sum_{l=0}^{k} q^{j+l} s_{2(k-l),\mathbb{Z}_2}(D_X) \binom{m-d_X/2 - k + l}{l}
\end{align*}
where $d_X = -\sigma(X)/8$. Hence by Lemma \ref{lem:mk3}, and assuming $m$ is chosen large enough that $j \le m-1$, we have
\begin{align*}
& SW^{G}_{M_1 \# mK3 , \mathfrak{s}_1 \# \mathfrak{s}_2}(q_2^j) \\
& \quad = u^{22m+b_+(X)-3} \! \! \! \! \sum_{k \ge m-j-1 } \! \! \! \! \mu^{-d_X/2-k+m}  s_{2(k+j-m+1),\mathbb{Z}_2}(D_X) \binom{2m-d_X/2 - k-j-1}{m-j-1}.
\end{align*}

This is a Laurent polynomial in $\mu$. Note that negative powers of $\mu$ can be written in terms of negative powers of $y$. In fact, from $\mu^2 = y^4$ one sees that $\mu^{-2r} = y^{-4r}$ and $\mu^{-2r-1} = y^{-4r-2} + y^{-4r-3}u^2$. Similarly positive powers of $\mu$ can be written in terms of positive powers of $y$. Since the left hand side of this equation is a polynomial in $y$, the terms with negative powers of $\mu$ must actually be zero. In any case, upon restricting to $Pin(2) \subset G$, only the $\mu^0$-coefficient survives, giving
\begin{equation}\label{equ:swm2}
SW^{Pin(2)}_{M,\mathfrak{s}_M}(q^j) = u^{22m+b_+(X)-3}s_{2(-d_X/2+j+1),\mathbb{Z}_2}(D_X), \\
\end{equation}
where we set $s_l(D_X) = 0$ if $l < 0$. Equating (\ref{equ:swm1}) and (\ref{equ:swm2}) gives
\[
u^{22m} SW^{Pin(2)}_{X,\mathfrak{s}}(q^j) = u^{22m+b_+(X)-3}s_{2(-d_X/2+j+1),\mathbb{Z}_2}(D_X).
\]
This is an equality in $H^*_{\mathbb{Z}_2}( Pic^{\mathfrak{s}}(X) \times S^{22m+b_+(X)-1}) \cong H^*( Pic^{\mathfrak{s}}(X) )[u]/( u^{22m+b_+(X)})$. The left hand side is divisible by $u^{22m}$, so the right hand side must also be. Hence if $b_+(X) < 3$, then $s_{2(-d_X/2+j+1),\mathbb{Z}_2}(D_X)$ is divisble by $u^{3-b_+(X)}$. Cancelling a common factor of $u^{22m}$ gives
\[
SW^{Pin(2)}_{X , \mathfrak{s}}(q^j) = u^{b_+(X)-3} s_{2(-d_X/2+j+1) , \mathbb{Z}_2}(D_X)
\]
which holds in  $H^*_{\mathbb{Z}_2}( Pic^{\mathfrak{s}}(X) \times S^{b_+(X)-1}) \cong H^*( Pic^{\mathfrak{s}}(X) )[u]/( u^{b_+(X)})$. Finally, note that if $b_+(X) < 3$ then $\sigma(X) = 0$ by the $10/8$-inequality, so in this case $SW^{Pin(2)}_{X , \mathfrak{s}}(q^j) = u^{b_+(X)-3} s_{2(j+1) , \mathbb{Z}_2}(D_X)$.
\end{proof}

\begin{remark}
Note that in the above proof, we do not really need to know that $K3 \# \overline{K3} \cong 22(S^2 \times S^2)$, since the only property of $K3 \# \overline{K3}$ we use is that it has the same Betti numbers and signature as $22(S^2 \times S^2)$.
\end{remark}

\begin{lemma}\label{lem:swx2k}
Let $X$ be a compact, oriented, smooth $4$-manifold with $b_+(X) > 0$ and let $\mathfrak{s}$ be a spin structure. Then $SW_{X,\mathfrak{s}}(x^{2k}) = 0$ for all $k > 0$, where $SW_{X,\mathfrak{s}}(x^k)$ denotes the mod $2$ families Seiberg--Witten invariants $SW_{X,\mathfrak{s}}(x^m) \in H^*( Pic^{\mathfrak{s}}(X) \; \mathbb{Z}_2)$ (in the case $b_+(X) = 1$, the statement holds independent of the choice of chamber).
\end{lemma}
\begin{proof}

According to \cite[Theorem 4.10]{bk}, the second Steenrod square $Sq^2( SW_{X,\mathfrak{s}}(x^m))$ is given by
\[
Sq^2( SW_{X,\mathfrak{s}}(x^m)) = (-\sigma(X)/8 + m)SW_{X,\mathfrak{s}}(x^{m+1}) + (s_1(D) + w_2(H^+(X))) SW_{X,\mathfrak{s}}(x^m).
\]
This formula can be greatly simplified. First, since $Pic^{\mathfrak{s}}(X)$ is a torus, the Steenrod squares are trivial and the left hand side is zero. Second since $X$ is spin, $\sigma(X)/8$ is even and $s_1(D) = 0$. Third, $H^+(X) \to Pic^{\mathfrak{s}}(X)$ is a trivial bundle, so $w_2(H^+(X)) = 0$. So we are left with $m SW_{X , \mathfrak{s} }( x^{m+1} ) = 0$ for all $m \ge 0$. Taking $m=2k-1$, we see that $SW_{X,\mathfrak{s}}(x^{2k}) = 0$ for all $k > 0$.

\end{proof}

\begin{theorem}\label{thm:b+3}
Let $X$ be a compact, oriented, smooth $4$-manifold with $b_+(X) = 3$ and let $\mathfrak{s}$ be a spin structure with corresponding index bundle $D \to Pic^{\mathfrak{s}}(X)$. Then $s_{2j}(D) = 0 \; ({\rm mod} \; 2)$ for all $j > 1+\sigma(X)/16$.
\end{theorem}
\begin{proof}
Lemma \ref{lem:swx2k} gives $SW_{X,\mathfrak{s}}(x^{2k}) = 0$ for all $k > 0$. Now since $b_+(X) = 3$, Lemma \ref{lem:u=0} and Theorem \ref{thm:swpin1} give
\[
0 = SW_{X,\mathfrak{s}}(x^{2k}) = SW^{Pin(2)}_{X,\mathfrak{s}}( q^k )|_{u=0} = s_{2(k+1+\sigma(X)/16)}(D)
\]
for all $k > 0$. Hence $s_{2j}(D) = 0 \; ({\rm mod} \; 2)$ for all $j > 1+\sigma(X)/16$.

\end{proof}

\subsection{Case $b_+(X) \ge 3$}\label{sec:b+3}

Lemmas \ref{lem:u=0}, \ref{lem:swx2k} and Theorems \ref{thm:swpin1}, \ref{thm:b+3} yield a complete calculation of the mod $2$ Seiberg--Witten invariant for spin-structures with $b_+(X) \ge 3$:

\begin{theorem}\label{thm:swspin1}
Let $X$ be a compact, oriented, smooth $4$-manifold with $b_+(X) \ge 3$ and let $\mathfrak{s}$ be a spin-structure on $X$. If $b_+(X) \neq 3$, then $SW_{X,\mathfrak{s}}(x^j) = 0$ for all $j \ge 0$. If $b_+(X) = 3$, then $SW_{X,\mathfrak{s}}(x^j) = 0$ for all $j > 0$ and
\[
SW_{X,\mathfrak{s}}( 1 ) = s_{2(1+\sigma(X)/16)}(D).
\]
\end{theorem}

\begin{remark}
If $b_+(X) = 3$ and $X$ is spin, then $\sigma(X) = 0$ or $-16$ by the $10/8$-inequality. In the case $\sigma(X) = -16$, we get $SW_{X,\mathfrak{s}}(1) = s_0(D) = 1$ and in the case $\sigma(X)=0$, we get $SW_{X,\mathfrak{s}}(1) = s_2(D) = c_2(D) \in H^4( Pin^{\mathfrak{s}}(X) ; \mathbb{Z}_2)$. When $\sigma(X)=0$, our result is a generalisation of a result of Morgan--Szab\'o \cite{ms}, who proved the $b_1(X) = 0$ case. When $\sigma(X) = -16$, our result is a generalisation of a result of Ruberman--Strle \cite{rs}, who proved the $b_1(X) = 4$ case.
\end{remark}

Theorems \ref{thm:swspin1} and \ref{thm:swpin1} give $SW_{X,\mathfrak{s}}$ and $SW^{Pin(2)}_{X,\mathfrak{s}}$ in terms of Segre classes of the index bundle $D \to Pic^{\mathfrak{s}}(X)$. These can be computed using the families index theorem, as we will now show.

Let $T_X = H^1(X ; \mathbb{R})/H^1(X ; \mathbb{Z})$ be the moduli space of flat unitary line bundles on $X$. Over $X \times T_X$ we have a universal line bundle with connection $L_X \to X \times T_X$ with the property that its restriction to $X \times {p}$ is the flat line bundle corresponding to $p \in T_X$. Let $\Omega \in H^2(X \times T_X ; \mathbb{Z})$ be the first Chern class of $L_X$. We have that $\Omega = \sum_i x_i \! \smallsmile \! y_i $, where $\{ y_i \}$ is a basis of $H^1(X ; \mathbb{Z})$ and $\{ x_i \}$ is the corresponding dual basis of $H^1( T_X ; \mathbb{Z}) \cong Hom( H^1(X ; \mathbb{Z}) , \mathbb{Z})$. 

The spin connection gives an identification $Pic^{\mathfrak{s}}(X) \cong T_X$. Then the families index theorem gives:
\[
Ch(D) = \int_{X} e^{\Omega } \wedge \left( 1 - \frac{\sigma(X)}{8} vol_X \right),
\]
where $\int_{X}$ means integration over the fibres of $X \times T_X \to T_X$ and $vol_X$ is a $4$-form on $X$ such that $\int_X vol_X = 1$. Since each term in $\Omega$ has degree $1$ in $X$, we have that $\Omega^5 = 0$ and that $\int_X \Omega^n \wedge vol_X = 0$ for any $n > 0$. It follows that
\[
Ch(D) = -\frac{\sigma(X)}{8} + \frac{1}{24} \int_X \Omega^4.
\]

For any subset $I \subset \{1,\dots , b_1(X)\}$ of size $4$, let $c_I = \langle y_{i_1}y_{i_2} y_{i_3} y_{i_4} , [X] \rangle \in \mathbb{Z}$ where $I = \{ i_1 , i_2 , i_3 , i_4 \}$ ordered so that $i_1 < i_2 < i_3 < i_4$. Also set $x_I = x_{i_1}x_{i_2}x_{i_3}x_{i_4}$. Then we have
\[
\frac{1}{24} \int_X \Omega^4 = \sum_{|I|=4} c_I x_I \in H^4(T_X ; \mathbb{Z}).
\]
Let $s = (1/24) \int_X \Omega^4$ and $d = -\sigma(X)/8$. Then $Ch(D) = d + s$. If we write $Ch(D) = \sum_{i \ge 0} Ch_i(D)$, where $Ch_i(D)$ has degree $2i$, then $Ch_0(D) = d$, $Ch_2(D) = s$ and all other terms are zero. Since $Ch_1(D) = c_1(D)$ and $Ch_2(D) = (c_1(D)^2 - 2c_2(D))/2 = -c_2(D)$, we see that $c_1(D) = 0$ and $s = Ch_2(D) = s_2(D)$ is the second Segre class of $D$.

Using the splitting principle, one can express the total Segre class of a virtual bundle $V$ in terms of the Chern character as:
\[
s(V) = exp\left( \sum_{n \ge 1} (-1)^n (n-1)! \; Ch_n(V) \right).
\]
Therefore, in the case $V = D$, we have $s(D) = e^{s_2(D)}$. Thus $s_j(D) = 0$ for odd $j$ and
\[
s_{2j}(D) = \frac{1}{j!} s_2(D)^j,
\]
where, as shown above, $s_2(D)$ is given by
\begin{equation}\label{equ:s2}
s_2(D) = \sum_{|I|=4} c_I x_I \in H^4(T_X ; \mathbb{Z}).
\end{equation}
Choose an arbitrary ordering of subsets of $\{1 , \dots , b_1(X)\}$ of size $4$. Then it follows that $s_{2j}(D)$ can be written as
\[
s_{2j}(D) = \sum_{I_1 < \cdots < I_j } c_{I_1} \cdots c_{I_j} x_{I_1} \cdots x_{I_j}.
\]

Using the above formula, Theorem \ref{thm:swspin1} gives a complete description of the mod $2$ Seiberg--Witten invariant of any spin structure, depending only on $b_+(X), \sigma(X)$ and the $4$-fold cup products $\langle y_1 \,y_2\, y_3\,  y_4 , [X] \rangle$, $y_1,y_2,y_3,y_4 \in H^1(X ; \mathbb{Z})$.

We note here a consequence of Theorem \ref{thm:b+3} that is of independent interest:

\begin{theorem}\label{thm:cup}
Let $X$ be a compact, oriented, smooth spin $4$-manifold with $b_+(X) = 3$ and $\sigma(X) = -16$. Then $\langle y_1 \,y_2\, y_3\,  y_4 , [X] \rangle$ is even for any $y_1,y_2,y_3,y_4 \in H^1(X ; \mathbb{Z})$.
\end{theorem}
\begin{proof}
Theorem \ref{thm:b+3} implies that $s_2(D) = 0 \; ({\rm mod} \; 2)$, which by Equation (\ref{equ:s2}) implies that all $4$-fold cup products $\langle y_1 \,y_2\, y_3\,  y_4 , [X] \rangle$ are even.
\end{proof}

This result actually follows from a theorem of Furuta--Kametani \cite[Theorem 5]{fk}, proved by different means. See also \cite[Theorem 4]{na2} for a related result.

\begin{corollary}\label{cor:notsmooth}
Let $M_{E_8}$ denote the compact, simply-connected topological $4$-manifold with intersection form the negative definite $E_8$ lattice. Then $T^4 \# 2M_{E_8} \# n(S^1 \times S^3)$ does not admit a smooth structure for any $n \ge 0$.
\end{corollary}
\begin{proof}

Suppose that $X = T^4 \# 2M_{E_8} \# n(S^1 \times S^3)$ admits a smooth structure. Since $H^2(X ; \mathbb{Z})$ has no $2$-torsion, the map $\mathfrak{s} \to c(\mathfrak{s})$ sending a spin$^c$-structure to its characteristic is a bijection. But the intersection form on $H^2(X ; \mathbb{Z})$ is even, so $X$ is spin. We also have that $\langle y_1 \,y_2\, y_3\,  y_4 , [X] \rangle = \pm 1$ for a basis $y_1, y_2, y_3 , y_4$ of $H^1( T^4 ; \mathbb{Z}) \subseteq H^1(X ; \mathbb{Z})$. But this contradicts Theorem \ref{thm:cup}.
\end{proof}

\subsection{Case $b_+(X) =1$ or $2$}\label{sec:b+12}

We now address the case that $b_+(X) =1$ or $2$. By Lemma \ref{lem:swx2k}, we have that $SW_{X , \mathfrak{s}}(x^m) = 0 \; ({\rm mod} \; 2)$ for any $m>0$, so we are reduced to the case $m=0$. Since $X$ is spin and $b_+(X) < 3$, we have that $\sigma(X) = 0$ and it follows that $SW_{X , \mathfrak{s}}(1) \in H^{b_+(X)+1}( Pic^{\mathfrak{s}}(X))$. Furthermore, Theorem \ref{thm:swpin1} implies that $s_{2 , \mathbb{Z}_2}(D)$ is divisible by $u^{3-b_+(X)}$, in other words, there is a class $\alpha_1 \in H^{1+b_+(X)}_{\mathbb{Z}_2}( Pic^{\mathfrak{s}}(X) )/(u^{b_+(X)})$ such that $s_{2,\mathbb{Z}_2}(D) = u^{3-b_+(X)} \alpha_1$ and
\begin{equation}\label{equ:swb+12}
SW_{X , \mathfrak{s}}(1) = \alpha_1|_{u=0},
\end{equation}
where $a \mapsto a|_{u=0}$ denotes the map $H^*_{\mathbb{Z}_2}( Pic^{\mathfrak{s}}(X) )/(u^{b_+(X)}) \to H^*( Pic^{\mathfrak{s}}(X) )$ described in Remark \ref{rem:u0}. In order to evaluate $SW_{X , \mathfrak{s}}(1)$ we therefore need to understand the class $s_{2, \mathbb{Z}_2}(D)$ in more detail. For this purpose it will be convenient to make use of Real and Quaternionic $K$-theory as in \cite{at,dup}. 

A {\em Real structure} on a topological space $X$ is a continuous involution $j : X \to X$. A {\em Real vector bundle} on $X$ is a complex vector bundle $E \to X$ together with a lift of $j$ to an antilinear involution on $E$. Similarly a {\em Quaternionic vector bundle} on $X$ is a complex vector bundle $E \to X$ together with a lift of $j$ to an antilinear map $j : E \to E$ such that $j^2 = -1$. To a Real space one can define the Real and Quaternionic $K$-theories $KR^*(X), KH^*(X)$ with the property that $KR^0(X), KH^0(X)$ are the Grothendieck groups of Real and Quaternionic bundles. One also has canonical isomorphisms $KH^j(X) \cong KR^{j+4}(X)$.

Let $V$ be an $n$-dimensional real vector space and $\Lambda \subset V$ a latttice. Let $T^n = V/\Lambda$ be the corresponding $n$-dimensional torus. Let $T^n_-$ denote $T^n$ equipped with the Real structure $j : T^n \to T^n$ given by $j(x) = -x$. Recall that we have defined the equivariant Segre and Chern classes $s_{2j,\mathbb{Z}_2}(D), c_{2j ,\mathbb{Z}_2}(D) \in H^*_{\mathbb{Z}_2}( T^n_-)/(u^3)$ associated to any Quaternionic virtual vector bundle on $T^n_-$. Since these classes depend only on the class of $D$ in $KH(T^n_-)$, the total equivariant Segre class defines a map
\[
s_{\mathbb{Z}_2} : KH(T^n_-) \to H^*_{\mathbb{Z}_2}(T^n_-)/(u^3)
\]
satisfying $s_{\mathbb{Z}_2}(E \oplus F) = s_{\mathbb{Z}_2}(E)s_{\mathbb{Z}_2}(F)$. Recall from Proposition \ref{prop:z2cohom} and Remark \ref{rem:u0} that a choice of splitting of the forgetful map $H_{\mathbb{Z}_2}^*(T^n_- ) \to H^*( T^n_- )$ determines a $\mathbb{Z}_2[u]$-module isomorphism $H^*_{\mathbb{Z}_2}( T^n_- ) \cong H^*( T^n_- )[u]$ and therefore also an isomorphism of $\mathbb{Z}_2[u]$-modules
\[
H^*_{\mathbb{Z}_2}(T^n_-)/(u^3) \cong H^*( T^n_-)[u]/(u^3).
\]

Under such a choice of splitting, the second equivariant Segre class is an additive homomorphism
\[
s_{2 , \mathbb{Z}_2} : KH( T^n_-) \to H^4( T^n) \oplus H^3(T^n)u \oplus H^2(T^n)u^2.
\]
Recall that $s_{2, \mathbb{Z}_2}|_{u=0} = s_2$ is the ordinary Segre class. Then we can write
\[
s_{2,\mathbb{Z}_2}(E) = s_2(E) + u t_2(E) + u^2 r_2(E)
\]
for some $t_2(E) \in H^3(T^n)$, $r_2(E) \in H^2(T^n)$. So we have homomorphisms
\[
t_2 : KH(T^n_-) \to H^3(T^n), \quad r_2 : KH(T^n_-) \to H^2(T^n).
\]
We stress that the homomorphisms depend on the choice of splitting map $H^*( T^n_- ) \to H^*_{\mathbb{Z}_2}( T^n_- )$. However the leading order term clearly does not depend on the choice. So if $s_2(E) = 0$, then $t_2(E)$ does not depend on the choice of splitting and if $s_2(E) = t_2(E) = 0$, then $r_2(E)$ does not depend on the choice of splitting.

It then follows from Equation (\ref{equ:swb+12}) that the Seiberg--Witten invariants in the case $b_+(X) = 1,2$ are given by:
\begin{equation}\label{equ:sw12}
SW_{X,\mathfrak{s}}(1) = \begin{cases} r_2(D) & \text{if } b_+(X) = 1, \\ t_2(D) & \text{if } b_+(X) =2, \end{cases}
\end{equation}
where $D \in KH( Pic^{\mathfrak{s}}(X) )$ is the families index of the Dirac operator equipped with its Quaternionic structure. We stress that this equation holds no matter what choice of splitting is used to define $r_2,t_2$.

In what follows we will fix once and for all a choice of splitting of $H^*_{\mathbb{Z}_2}( T^n_-) \to H^*( T^n_-)$ and the homomorphisms $t_2, r_2$ will be defined with respect to this choice. Fix a generator $e$ of $H^1(T^1_-)$ and let $\hat{e} \in H^1_{\mathbb{Z}_2}( T^1_-)$ be the unique lift of $e$ such that $\hat{e}|_{0} = 0$, where $0 \in T^1_-$ denotes the origin. Now for any $n$, let $\pi_j : T^n_- \to T^1_-$ denote projection to the $j$-th factor. Fix a generator $e$ of $H^1(T^1_-)$ and set $e_j = \pi_j^*(e)$ for $j=1,\dots , n$. Then a basis for $H^k( T^n_-)$ is given by $\{ e_I \}_{|I| = k}$ where $e_I = e_{i_1} e_{i_2} \cdots e_{i_k}$, $I = \{ i_1, \dots , i_k\}$ and $I$ runs over all subsets of $\{1, \dots , n\}$ of size $k$ (notice that since we work over $\mathbb{Z}_2$, the order of the indices $i_1, \dots , i_k$ is irrelevant). Then for each $k$ we get a splitting $\phi_n : H^k( T^n_-) \to H^k_{\mathbb{Z}_2}( T^n_-)$ by taking
\[
\phi_n( e_I ) = \pi^*_{i_1}( \hat{e} ) \pi^*_{i_2}(\hat{e}) \cdots \pi^*_{i_k}( \hat{e} ), \quad I = \{i_1, \dots , i_k\}.
\]

This choice of splitting has the advantage that it commutes with restriction maps. More precisely, if $I = \{ i_1 < i_2 < \cdots < i_k\}$ is a subset of $\{1, \dots , n\}$ we let $\iota_I : T^k_- \to T^n_-$ be the inclusion map which sends $(x_1, \dots , x_k)$ to $(y_1, \dots , y_n)$, where $y_{i_a} = x_a$ for $a=1, \dots , k$ and $y_j = 0$ if $j \notin I$ (where $0 \in T^1_-$ denotes the identity element). Then it is easily seen that $\iota_I^* \phi_n = \phi_k \iota_I^*$. It follows that the corresponding homomorphisms $t_2 : KH(T^n_-) \to H^3(T^n)$, $r_2 : KH(T^n_-) \to H^2(T^n)$ determined by these splittings likewise commute with the pullback maps $\iota_I^*$. In a similar manner the splittings $\phi_n$ can be shown to commute with the maps $\pi_I^* : H^*(T^k_-) \to H^*(T^n_-)$ induced by the projection maps $\pi_I : T^n_- \to T^k_-$ given by $\pi_I(x_1, \dots , x_n) = (x_{i_1} , \dots , x_{i_k})$. Hence the homomorphisms $r_2,t_2$ commute with the maps $\pi_I^*$.

Since a degree $k$ cohomology class $a \in H^k(T^n)$ of an $n$-torus is determined by its restriction to $k$-dimensional subtori, to compute $r_2(D),t_2(D)$, it suffices to compute their restriction to $2$- or $3$-dimensional subtori of $Pic^{\mathfrak{s}}(X)$. Now one can easily show that (see e.g. \cite{fk})
\[
KH( T^2_-) \cong \mathbb{Z} \oplus \mathbb{Z}_2
\]
where the $\mathbb{Z}$-summand is generated by the trivial Quaternionic bundle $\mathbb{H}$. Since $s_{2,\mathbb{Z}_2}(\mathbb{H}) = 0$, it follows that the homomorphism $r_2 : KH(T^2_-) \to H^2(T^2) \cong \mathbb{Z}_2$ is either identically zero, or is the projection to the $\mathbb{Z}_2$ summand. We claim that $r_2$ is not identically zero. To see this, it suffices to find a compact, oriented, spin $4$-manifold $(X , \mathfrak{s})$ with $b_+(X) = 1$, $b_1(X) = 2$ and $SW_{X,\mathfrak{s}}(1) \neq 0 \; ({\rm mod} \; 2)$. For this we can take $X$ to be a hyperelliptic surface of the form $T/\mathbb{Z}_6$ as in Example \ref{ex:kod0}. We will show that $X$ has a non-trivial mod $2$ Seiberg--Witten invariant for some spin structure without appealing to Theorem \ref{thm:thm1}. In fact, since $X$ is a K\"ahler manifold with $c_1(X) = 0$, the canonical spin$^c$-structure $\mathfrak{s}$ comes from a spin structure and $SW_{X,\mathfrak{s}}(1) \neq 0 \; ({\rm mod} \; 2)$ (for both chambers).

In the case $n=3$, we have
\begin{equation}\label{equ:decomp}
KH( T^3_-) \cong \mathbb{Z} \oplus \mathbb{Z}_2^3 \oplus \mathbb{Z}_2.
\end{equation}
The subgroup of $KH(T^3_-)$ given by the first and second summand in (\ref{equ:decomp}) is also the subgroup generated by pullbacks of projections $T^3_- \to T^2_-$ to $2$-dimensional quotients of $T^3$. Since $t_2 : KH(T^3_-) \to H^3(T^3) \cong \mathbb{Z}_2$ takes values in $H^3(T^3)$ and since $t_2$ commutes with pullbacks (for our choice of splittings), it is clear that $t_2$ must vanish on the first two summands of (\ref{equ:decomp}). Therefore, $t_2$ is either identically zero or is given by projection to the third summand. To show that $t_2$ is not identically zero, it suffices to find a compact, oriented, spin $4$-manifold $(X , \mathfrak{s})$ with $b_+(X) = 2$, $b_1(X) = 3$ and $SW_{X , \mathfrak{s}}(1) \neq 0 \; ({\rm mod} \; 2)$. As in Example \ref{ex:kod0}, we can take $X$ to be a primary Kodaira surface diffeomorphic to the product $S^1 \times Y$ where $Y$ is the total space of a degree $1$ circle bundle $Y \to C$ over an elliptic curve. Recall that $X$ has a nowhere vanishing holomorphic $2$-form $\Omega$. Then $\omega = Re(\Omega)$ is a symplectic form on $X$ and the projection $X \to C$ is a Lagrangian fibration with respect to $\omega$. This implies that first Chern class of a compatible almost complex structure $J$ is trivial and hence the canonical spin$^c$-structure $\mathfrak{s}$ associated to $J$ comes from a spin structure. Then since $X$ is symplectic, we have $SW_{X , \mathfrak{s}}(1) \neq 0 \; ({\rm mod} \; 2)$. This proves that $t_2$ is not identically zero.

The above calculations completely characterise the homomorphisms $r_2,t_2$ for any torus $T^n_-$ and thus allow us to compute $SW_{X,\mathfrak{s}}(1)$ in terms of the class $D \in KH( Pic^{\mathfrak{s}}(X) )$. Our next task is to describe this class.

Let $A^n = V^*/\Lambda^*$ denote the dual torus of $T^n$. Equip $A^n$ with the trivial involution. Let $L \to A^n \times T^n$ denote the Poincar\'e line bundle. This can be explicitly constructed, as follows \cite{fk}. First consider the trivial line bundle $\widetilde{L} = V^* \times V \times \mathbb{C}$ on $V^* \times V$. Let $\Lambda^* \times \Lambda$ act on $\widetilde{L}$ according to
\[
( \mu , \lambda ) \cdot ( w , v , z ) = (w + \mu , v + \lambda , e^{2\pi i \langle \mu , v \rangle } z).
\]
The Poincar\'e line bundle $L$ may be defined as the quotient line bundle $L = \widetilde{L}/(\Lambda^* \times \Lambda) \to A^n \times T^n$. Then $L$ is a Real line bundle with Real structure given by
\[
j( w , v , z ) = (w , -v , \overline{z}).
\]

Fix an orientation on $V$. This induces orientations on $T^n$ and $A^n$. Since $A^n$ is a Lie group, it has a translation invariant trivialisation $\mathcal{F}_{A^n} \cong A^n \times SO(n)$ of the oriented frame bundle. This defines a distinguished spin structure whose prinicpal $Spin(n)$-bundle is the product $A^n \times Spin(n)$. We equip $A^n$ with this spin structure. Next, let $p_2 : A^n \times T^n_- \to T^n_-$ be the projection to $T^n_-$ and $p_1 : A^n \times T^n_- \to A^n$ the projection to $A^n$. Define the {\em Real Fourier--Mukai transform}
\[
FM : KO^j( A^n ) \to KR^{j-n}( T^n_- )
\]
by
\[
FM( E ) = (p_2)_*( L \otimes (p_1)^*(E) ).
\]

Let $X$ be a compact, oriented, smooth $4$-manifold. Let $V$ be the vector space $V = H^1(X ; \mathbb{R})$ and $\Lambda$ the lattice $\Lambda = H^1(X ; \mathbb{Z})$. Recall that the Jacobian torus of $X$ is $T_X = V/\Lambda$. The {\em Albanese torus} of $X$ is by definition the dual torus:
\[
A_X = V^*/\Lambda^* = Hom( H^1(X ; \mathbb{R}) , \mathbb{R})/ Hom( H^1(X ; \mathbb{Z}) , \mathbb{Z} ).
\]
Define the Albanese map $a : X \to A_X$ as follows. Choose a basepoint $x_0 \in X$ and choose a $b_1(X)$-dimensional subspace $\widetilde{V} \subset \Omega^1_{cl}(X)$ of closed $1$-forms such that the projection to cohomology $\widetilde{V} \to V$ is an isomorphism. Thus every isomorphism class in $H^1(X , \mathbb{R})$ has a unique representative in $\widetilde{V}$. For example, we could fix a choice of Riemannian metric $g$ and take $\widetilde{V}$ to be the space of $g$-harmonic $1$-forms. For $x \in X$ we define $a(x)$ as follows. Choose a path $\gamma$ from $x_0$ to $x$. Then we obtain an element of $Hom( H^1(X ; \mathbb{R}) , \mathbb{R})$ given by $[\lambda] \mapsto \int_{x_0}^{x} \lambda$, where $\lambda$ is the unique representative of $[\lambda]$ in $\widetilde{V}$. If we choose a different path $\gamma'$ from $x_0$ to $x$, then the difference $\lambda \mapsto \int_{\gamma} \lambda - \int_{\gamma'} \lambda$ is the $\mathbb{R}$-linear extension of an element of $Hom( H^1(X ; \mathbb{Z}) , \mathbb{Z})$. Hence, we get a well-defined element $a(x) \in A_X$. More concretely, if $\lambda_1, \dots , \lambda_b$ is a basis for $H^1(X ; \mathbb{Z})$, then we obtain an isomorphism $A_X \cong \mathbb{R}^b/\mathbb{Z}^b$ under which an element $(l_1 , \dots , l_b) \in \mathbb{R}^b/\mathbb{Z}^b$ corresponds to the image in $A_X$ of the homomorphism $H^1(X ; \mathbb{R}) \to \mathbb{R}$ which sends $\lambda_i$ to $l_i$. Under this identification the Albanese map $a : X \to A_X \cong \mathbb{R}^b/\mathbb{Z}^b$ is given by $a(x) = ( \int_{\gamma} \lambda_1 , \dots , \int_{\gamma} \lambda_b)$, where $\gamma$ is any path from $x_0$ to $x$. The Albanese map depends on the choice of basepoint $x_0$ and also on the choice of subspace $\widetilde{V}$. Different choices can be smoothly interpolated, hence $a$ is well defined up to homotopy. We abuse notation and speak of ``the'' Albanese map.

A spin structure $\mathfrak{s}$ on $X$ determines a $KO$-orientation. As above, we give $A_X$ the distinguished spin structure as a Lie group. Hence both $X$ and $A_X$ are $KO$-oriented. We define the {\em Albanese invariant} of $(X , \mathfrak{s})$ to be
\[
A(X,\mathfrak{s}) = a_*(1) \in KO^{b_1(X) - 4}(A_X) \cong KSp^{b_1(X)}(A_X)
\]
where $a_* : KO(X) \to KO^{b_1(X)-4}(A_X)$ denotes the push-forward map in $KO$-theory with respect to the given spin structures. Given the spin structure $\mathfrak{s}$ on $X$ we obtain a pushforward map 
\[
\pi_* : KR(X \times T_X) \to KR^{-4}(T_X) = KH(T_X)
\]
where $\pi : X \times T_X \to T_X$ is the projection to $T_X$. The pullback $L_X = (a \times id_{T_X})^*(L)$ of the Poincar\'e line bundle $L$ under the Albanese map is the universal line bundle on $X \times T_X$ which parametrises flat line bundles on $X$, as described in Section \ref{sec:b+3}. By definition of the index bundle $D \in KH(T_X)$, we have
\[
D = \pi_*( (a \times id)^*(L) ).
\]

\begin{proposition}\label{prop:dfm}
We have
\[
D = FM( A(X,\mathfrak{s}) ).
\]
\end{proposition}
\begin{proof}
Observe that $\pi = p_2 \circ (a \times id)$. Therefore
\begin{align*}
D &= \pi_*( (a \times id)^*(L) ) \\
&= (p_2)_* ( (a \times id)_*(  (a \times id)^*(L) ) ) \\
&= (p_2)_*( L \otimes (a \times id)_*(1) ) \\
&= (p_2)_*( L \otimes (p_1)^*( a_*(1) ) ) \\
&= (p_2)_*( L \otimes (p_1)^*(A(X , \mathfrak{s})) ).
\end{align*}
\end{proof}

Fix a basis $\lambda_1 , \dots , \lambda_{b_1(X)}$ of $\Lambda = H^1(X ; \mathbb{Z})$. For any subset $I \subset \{1,\dots , b_1(X) \}$ let $T(I) \subset T_X = V/\Lambda$ be the subtorus corresponding to the subspace $V(I)$ of $V = H^1(X ; \mathbb{R})$ spanned by $\{ \lambda_i \}_{i \in I}$. Let $\Lambda(I) = \Lambda \cap V(I)$ be the sublattice of $\Lambda$ spanned by $\{ \lambda_i \}_{i \in I}$, so $T(I) = V(I)/\Lambda(I)$. Corresponding to the inclusion $\iota_I : T(I) \to T_X$ there is a dual projection $\rho_I : A_X \to A(I)$, where $A(I) = V(I)^*/\Lambda(I)^*$. This is the map of tori induced by the linear map $V^* \to V(I)^*$ which is adjoint to the inclusion $V(I) \to V$. Let $FM_I : KO^j( A(I) ) \to KR^{j-|I|}( T(I) )$ denote the Real Fourier--Mukai transform for the pair $( A(I) , V(I) )$. Define the projected Albanese map $a_I : X \to A(I)$ to be the composition $a_I = \rho_I \circ a$.

\begin{proposition}\label{prop:ftres}
We have a commutative diagram
\[
\xymatrix{
KO^j( A^n) \ar[r]^-{FM} \ar[d]^-{(\rho_I)_*} & KR^{j-b_1(X)}( T^n_- ) \ar[d]^-{\iota_I^*} \\
KO^{j-(b_1(X)-|I|)}(A(I)) \ar[r]^-{FM_I} & KR^{j-b_1(X)}( T(I) ).
}
\]
Moreover, we have $\iota_I^*( D) = FM_I( (a_I)_*(1) )$.
\end{proposition}
\begin{proof}
Let $J = \{1, \dots , b_1(X) \} \setminus I$ be the complementary subset of $\{1,\dots , b_1(X)\}$. Then $T_X \cong T(I) \times T(J)$ and dually $A_X \cong A(I) \times A(J)$. In the calculations that follow we make use of various projection maps whose domain and codomain is the product of some subset of $\{ T(I) , T(J) , A(I) , A(J) \}$. To keep notation simple we set $A_1 = A(I)$, $A_2 = A(J)$, $A_3 = T(I)$, $A_4 = T(J)$ and if $R = \{ i_1 < i_2 < \cdots < i_k\}$ we set $A_R = A_{i_1} \times \cdots \times A_{i_K}$. If $S \subseteq R$ then we write $\pi^R_S : A_R \to A_S$ for the projection and $\iota^S_R : A_S \to A_R$ for the inclusion. Let $L \to A_{1234}$, $L(I) \to A_{13}$, $L(J) \to A_{24}$ be the Poincar\'e line bundles. Then $L \cong (\pi^{1234}_{13})^*(L(I)) \otimes (\pi^{1234}_{24})^*(L(J))$. Let $a \in KO^*(A_X)$. Then
\begin{align*}
\iota_I^* FM(x) &= \iota_I^* (\pi_{34}^{1234})_* ( L \otimes (\pi^{1234}_{12})^*(a) ) \\
&= (\pi^{123}_{3})_* ( \iota^{123}_{1234} )^* ( L \otimes (\pi^{1234}_{12})^*(a) ) \\
&= (\pi^{123}_{3})_* \left( (\pi^{123}_{13})^*(L(I)) \otimes (\pi^{123}_{2})^* ( \iota^2_{24})^*(L(J)) \otimes (\pi^{123}_{12})^*(a) \right) \\
&= (\pi^{123}_{3})_* \left( (\pi^{123}_{13})^*(L(I)) \otimes (\pi^{123}_{12})^*(a) \right) \\
&= (\pi^{13}_{3})_* (\pi^{123}_{13})_* \left( (\pi^{123}_{13})^*(L(I)) \otimes (\pi^{123}_{12})^*(a) \right) \\
&= (\pi^{13}_{3})_* \left( L(I) \otimes (\pi^{13}_1)^* ( \pi^{12}_{1} )_*(a) \right) \\
&= FM_I( (\pi^{12}_{1})_*(a) ) \\
&= FM_I( (\rho_I)_*(a) ).
\end{align*}
Where we used that $(\iota^2_{24})^*(L(J)) \cong 1$. This proves commutativity of the diagram. Then since $D = FM( a_*(1))$, it follows that
\[
\iota_I^*(D) = \iota_I^* FM( a_*(1) ) = FM_I( (\rho_I)_* a_* (1) ) = FM_I( (a_I)_*(1)).
\]
\end{proof}

When $|I|=2$, we can regard $r_2( \iota_I^*(D) ) \in H^2( T(I) ) \cong \mathbb{Z}_2$ as an element of $\mathbb{Z}_2$. Similarly, when $|I|=3$, we can regard $t_2( \iota_I^*(D) ) \in H^3( T(I) ) \cong \mathbb{Z}_2$ as an element of $\mathbb{Z}_2$. We then have:

\begin{proposition}\label{prop:pullback}
If $b_+(X) = 1$ and $|I|=2$, then $r_2( \iota_I^*(D) ) = \tau^*( (a_I)_*(1) )$, where $\tau^* : KO^{-2}(A(I)) \to KO^{-2}(pt) \cong \mathbb{Z}_2$ is the map induced by the inclusion of a point $\tau : pt \to A(I)$. Similarly, if $b_+(X)=2$ and $|I|=3$, then $t_2( \iota_I^*(D) ) = \tau^*( (a_I)^*(1) )$, where $\tau^* : KO^{-1}(A(I)) \to KO^{-1}(pt) \cong \mathbb{Z}_2$ is the map induced by the inclusion of a point $\tau : pt \to A(I)$.
\end{proposition}
\begin{proof}
We give the proof for the case $b_+(X)=1$ and $|I|=2$. The case $b_+(X)=2$ and $|I|=3$ is similar. Recall that for $|I|=2$, $KH(T(I) ) \cong \mathbb{Z} \oplus \mathbb{Z}_2$ and that $r_2$ is the projection to the second factor. Since the rank of $D$ as a virtual bundle is zero (by the families index theorem and the fact that $\sigma(X)=0$), it follows that the rank of $D_I = \iota_I^*(D)$ is also zero, and hence $D_I$ lies in the $\mathbb{Z}_2$ summand of $KH(T(I))$. In particular, the image of $D_I$ under the restriction to a point $KH(T(I)) \to KH(pt) \cong \mathbb{Z}$ is zero. By Proposition \ref{prop:dfm}, it follows that $\rho_*( (a_I)_*(1) ) = 0$, where $\rho_* : KO^{-2}(A(I)) \to KO^{-4}(pt) \cong \mathbb{Z}$ is the pushforward map to a point. But $KO^{-2}(A(I)) \cong \mathbb{Z} \oplus \mathbb{Z}_2$, where the second summand is the kernel of the pushforward map. This means that $(a_I)_*(1)$ lies in the $\mathbb{Z}_2$ summand of $KO^{-2}(A(I))$.

Next, we observe that the map $\tau^* : KO^{-2}(A(I)) \to KO^{-2}(pt) \cong \mathbb{Z}_2$ is given by projection to the second factor of $KO^{-2}(A(I))$. The composition 
\[
\mathbb{Z}_2 \to KO^{-2}(A(I)) \buildrel FM_I \over \longrightarrow KH(T(I)) \buildrel r_2 \over \longrightarrow \mathbb{Z}_2
\]
is either identically zero, or an isomorphism, where the first map is inclusion of the $\mathbb{Z}_2$ summand. To show that it is not identically zero, it suffices to give one example of a compact, oriented, smooth spin $4$-manifold $(X,\mathfrak{s})$ for which $SW_{X,\mathfrak{s}}(1) \neq 0 \; ({\rm mod} \; 2)$. For this we can take $X$ to be the hyperelliptic surface given in Example \ref{ex:kod0}. This shows that $r_2 \circ FM_I = \tau^*$ on the subgroup $\mathbb{Z}_2 \subset KO^{-2}(A(I))$ and hence $r_2( D_I ) = \tau^*( (a_I)_*(1))$.
\end{proof}

Let $\mathfrak{s}$ be a spin structure on $X$. As in the introduction, we define maps $q_{\mathfrak{s}}^k : H^1(X ; \mathbb{Z})^k \to \mathbb{Z}_2$, for $k=2,3$ as follows. Let $a_1, \dots , a_k \in H^1(X ; \mathbb{Z})$. Each class $a_i$ is equivalent to specifying a homotopy class of smooth map $f_i : X \to S^1$ for which $a_i = f_i^*(dt)$, where $dt$ is the generator of $H^1(S^1 ; \mathbb{Z})$. Hence we get a map $f = (f_1 ,\dots , f_k) : X \to T^k$ from $X$ to the $k$-dimensional torus. The level set $C = f^{-1}(t)$ of a regular value of $f$ is a normally framed submanifold of $X$ of dimension $4-k$. The normal framing defines a map $\mathcal{F}_C \to \mathcal{F}_X|_C$ from the frame bundle of $C$ to the restriction of the frame bundle of $X$ to $C$. The spin structure $\mathfrak{s}$ on $X$ is equivalent to specifying a double cover $\widetilde{\mathcal{F}}_X \to \mathcal{F}_X$ which on each fibre of $\mathcal{F}_X$ is isomorphic to the unique non-trivial double covering $Spin(4) \to SO(4)$. The pullback of the covering $\widetilde{\mathcal{F}}_X \to \mathcal{F}_X$ to $\mathcal{F}_C$ defines a spin structure on $C$ which we denote by $\mathfrak{s}|_C$. Then we define $q^k_{\mathfrak{s}}(a_1 , \dots , a_k)$ to the be mod $2$ index of $\mathfrak{s}|_{C}$. More precisely, the spin structure $\mathfrak{s}|_C$ defines a $KO$-orientation on $C$ and hence a push-forward map $p_* : KO(C) \to KO^{k-4}(pt)$. Now since $k-4 \in \{-1,-2\}$, we have that $KO^{k-4}(pt) \cong \mathbb{Z}_2$ and $p_*(1)$ is the mod $2$ index of the Dirac operator on $C$ \cite{as}. The following proposition shows that $q^k_{\mathfrak{s}}$ does not depend on the choice of $C$ and hence is well-defined.
\begin{proposition}\label{prop:embedded}
The mod $2$ index of $\mathfrak{s}|_C$ does not depend on the choice of functions $f_1, \dots , f_k$ representing $a_1, \dots, a_k$ or on the choice of regular value $t$. Hence $q_{\mathfrak{s}}^k$ is well-defined for $k=2,3$.
\end{proposition}
\begin{proof}
The map $f : X \to T^k$ may be composed with the map $(S^1 \times \cdots \times S^1) \to (S^1 \wedge \cdots \wedge S^1) \cong S^k$, giving a map $g : X \to S^k$. The homotopy class of this map only depends on the cohomology classes $a_1, \dots , a_k$. The Pontryagin--Thom construction implies that the homotopy class of map $f : X \to S^k$ is equivalent to a framed cobordism class of submanifold of $X$ codimension $k$. The level set $C$ of a regular value of $f$ represents this cobordism class. Any two representatives $C,C'$ are framed cobordant. If $\phi : Y \to [0,1] \times X$ is a framed cobordism from $C$ to $C'$, then using the normal framing we get an induced spin structure $\mathfrak{s}|_Y$. Hence $C$ and $C'$ are spin cobordant. This implies that the mod $2$ indices of $\mathfrak{s}|_C$ and $\mathfrak{s}|_{C'}$ are the same, since the mod $2$ index is a spin cobordism invariant \cite[Chapter III, \textsection 16]{lami}. 
\end{proof}

\begin{proof}[Proof of cases (1) and (2) of Theorem \ref{thm:thm1}] Suppose that $b_+(X) = 1$. Then by Equation (\ref{equ:sw12}), $SW_{X,\mathfrak{s}}(1) = r_2(D)$. To compute $r_2(D) \in H^2(T_X)$, it suffices to compute $\iota_I^* r_2(D) = r_2( \iota_I^*(D) )$ for every subset $I \subseteq \{1,\dots , b_1(X)\}$ of size $2$. By Proposition \ref{prop:ftres}, $\iota_I^*(D) = FM_I( (a_I)_*(1) )$, where $a_I = \rho_I \circ a : X \to A(I)$ is the projection of the Albanese map, $(a_I)_*(1) \in KO^{-2}( A(I) )$ and $FM_I : KO^{-2}( A(I) ) \to KR^{-4}( T(I) ) = KH( T(I) )$ is the Fourier--Mukai transform. By Proposition \ref{prop:pullback}, we have
\[
r_2( \iota_I^* D ) = r_2( FM_I( (a_I)_*(1) ) ) = \tau^*( (a_I)_*(1) ),
\]
where $\tau : pt \to A(I)$ is the inclusion of a point. Let $I = \{ i_1 < i_2 \}$ and set $a = \lambda_{i_1}, b = \lambda_{i_2}$. Let $A, B \subset X$ be closed, oriented embedded submanifolds of $X$ Poincar\'e dual to $a,b$. Moreover, we choose $A$, $B$ so that they intersect transversally in an embedded surface $C = A \cap B$ which is Poincar\'e dual to $a \smallsmile b$. Now choose closed differential forms $\alpha, \beta$ which represent $a,b$ in de-Rham cohomology. We can choose $\alpha$ and $\beta$ to be supported in neighbourhoods of $A$ and $B$. Choose a basepoint $x_0 \in X$ away from the supports of $\alpha, \beta$. The projected Albanese map $a_I : X \to \mathbb{R}^2/\mathbb{Z}^2$ is given by $a_I(x) = ( \int_{\gamma} \alpha , \int_{\gamma} \beta )$, where $\gamma$ is a path from $x_0$ to $x$. Each time $\gamma$ passes through $A$, $\int_{\gamma} \alpha$ winds once around the circle and similarly each time $\gamma$ passes through $B$, $\int_{\gamma} \beta$ winds once around the circle. Furthermore, we can choose $\alpha, \beta$ such that $\int_{\gamma} \alpha \in (1/2) + \mathbb{Z}$ if and only if $x \in A$ and $\int_{\gamma} \beta \in (1/2) + \mathbb{Z}$ if and only if $x \in B$. This is easily achieved by using the explicit construction of Poincar\'e dual cohomology classes \cite[Chapter I, \textsection 6]{botu}. We also choose $\alpha, \beta$ so that $(1/2 , 1/2)$ is a regular value of $a_I$. Note further that $(a_I)^{-1}( 1/2 , 1/2 ) = A \cap B = C$. Now consider the commutative diagram
\[
\xymatrix{
X \ar[r]^-{a_I} & A(I) \\
C \ar[u]^{i} \ar[r]^{p} & pt \ar[u]^{\tau}
}
\]
where $\tau$ is the inclusion map sending $pt$ to $(1/2 , 1/2 ) \in A(I)$. By commutativity, we have that
\[
r_2( \iota_I^*(D) ) = \tau^*( (a_I)_*(1) ) = p_*( i^*(1) ) = p_*(1).
\]
The spin structure $\mathfrak{s}$ restricts to a spin structure $\mathfrak{s}|_{C}$ and this is used to define the push-forward map $p_* : KO(C) \to KO^{-2}(pt) \cong \mathbb{Z}_2$. Now we observe that $p_*(1) \in KO^{-2}(pt) \cong \mathbb{Z}_2$ is precisely the mod $2$ index of $\mathfrak{s}|_C$. So we have proven that for any subset $I = \{i_1 < i_2\}$ of size $2$,
\begin{equation}\label{equ:L12}
\langle SW_{X,\mathfrak{s}}(1) , \lambda_{i_1} \smallsmile \lambda_{i_2} \rangle = r_2( \iota_I^*(D) ) = q^2_{\mathfrak{s}}( \lambda_{i_1} ,  \lambda_{i_2} ).
\end{equation}
Since the choice of basis $\lambda_1, \dots , \lambda_n$ of $H^1(X ; \mathbb{Z})$ was arbitrary, we have in fact shown that (\ref{equ:L12}) holds for any pair $\lambda_{i_1}, \lambda_{i_2} \in H^1(X ; \mathbb{Z})$ which generate a primitive sublattice of $\Lambda = H^1(X ; \mathbb{Z})$ (a sublattice $\Lambda' \subseteq \Lambda$ is primitive if $\Lambda/\Lambda'$ is torsion-free).

We will show that $(\ref{equ:L12})$ holds for any pair of elements $a,b \in H^1(X ; \mathbb{Z})$ in place of $\lambda_{i_1}, \lambda_{i_2}$. Let $f_a,f_b : X \to S^1$ be maps to $S^1$ which represent $a,b$. If $a,b$ are linearly dependent, then $a \smallsmile b = 0$, so $\langle SW_{X,\mathfrak{s}}(1) , a \smallsmile b \rangle = 0$. On the other hand, if $a$ and $b$ are linearly dependent, then the map $f = (f_a , f_b) : X \to T^2$ factors through a $1$-dimensional torus and hence the pre-image $f^{-1}(t)$ of a generic $t \in T^2$ is empty, so $q_{\mathfrak{s}}^2( a , b ) = 0$. This proves the result in the case that $a$ and $b$ are linearly dependent. Now suppose $a,b$ are linearly dependent. Let $\Lambda' = \{ \lambda \in \Lambda \; | \; k\lambda = n a + m b \text{ for some } k,n,m \in \mathbb{Z}, k > 0 \}$. Then $\Lambda'$ is a primitive sublattice of $\Lambda$. Let $\lambda_1, \lambda_2$ be a basis for $\Lambda'$. Then there exists an element $A = \left[ \begin{matrix} \alpha & \beta \\ \gamma & \delta \end{matrix} \right] \in SL(2,\mathbb{Z})$ and non-zero integers $d_1,d_2$ such that
\begin{equation}\label{equ:abA}
\left[ \begin{matrix} a \\ b \end{matrix} \right] = \left[ \begin{matrix} \alpha & \beta \\ \gamma & \delta \end{matrix} \right] \left[ \begin{matrix} d_1 \lambda_1 \\ d_2 \lambda_2 \end{matrix} \right].
\end{equation}
We claim this implies that $q^2_{\mathfrak{s}}( a , b ) = q^2_{\mathfrak{s}}( d_1 \lambda_1 , d_2 \lambda_2)$. To see this, note that the matrix $A$ defines an orientation preserving diffeomorphism $A : T^2 \to T^2$. Then $(f_a , f_b) = A( f_{\lambda_1}^{d_1} , f_{\lambda_2}^{d_2} )$, where $f_{\lambda_1} , f_{\lambda_2} : X \to S^1$ are maps representing $\lambda_1, \lambda_2$. Let $t \in T^2$ be a regular value of $g = (f_{\lambda_1} , f_{\lambda_2})$. Then $At$ is a regular value of $f$ and we have an equality $f^{-1}(At) = g^{-1}(t)$. The normal framings on $C = f^{-1}(At)$ induced by $f$ and $g$ are different, but they are related by a constant change of basis given by the matrix $A$. Since $A \in SL(2,\mathbb{Z}) \subset SL(2 , \mathbb{R})$ and $SL(2,\mathbb{R})$ is connected, the two normal framings are isotopic. This proves the claim that $q^2_{\mathfrak{s}}(a,b) = q^2_{\mathfrak{s}}(d_1\lambda_1 , d_2\lambda_2)$. Next, we claim that $q^2_{\mathfrak{s}}( d_1 \lambda_1 , d_2 \lambda_2) = d_1 d_2 q^2_{\mathfrak{s}}( \lambda_1 , \lambda_2)$. To see this, let $A_1, A_2 \subset X$ be submanifolds Poincar\'e dual to $\lambda_1, \lambda_2$ intersecting transversally in $C = A_1 \cap A_2$. Since the normal bundles of $A_1, A_2$ are trivial, we can take $|d_1|$ parallel copies of $A_1$ representing $|d_1| \lambda_1$ and $|d_2|$ parallel copies of $A_2$ representing $|d_2| \lambda_2$ and intersecting in $|d_1 d_2|$ disjoint copies of $C$. Hence $q^2_{\mathfrak{s}}( |d_1| \lambda_1 , |d_2| \lambda_2 ) = |d_1 d_2| q^2_{\mathfrak{s}}( \lambda_1 , \lambda_2)$. One can also see that $q^2_{\mathfrak{s}}( d_1 \lambda_1 , d_2 \lambda_2)$ is insensitive to the signs of $d_1$ and $d_2$ and so we have proven the claim that $q^2_{\mathfrak{s}}( d_1 \lambda_1 , d_2 \lambda_2) = d_1 d_2 q^2_{\mathfrak{s}}( \lambda_1 , \lambda_2)$. Now from (\ref{equ:abA}), we have that $a \smallsmile b = d_1 d_2 \lambda_1 \smallsmile \lambda_2$ and hence
\begin{align*}
\langle SW_{X , \mathfrak{s}}(1) , a \smallsmile b \rangle &= \langle d_1 d_2 SW_{X , \mathfrak{s}}(1) , \lambda_1 \smallsmile \lambda_2 \rangle \\
&= d_1 d_2 q^2_{\mathfrak{s}}( \lambda_1 , \lambda_2) \\
&= q^2_{\mathfrak{s}}( d_1\lambda_1 , d_2 \lambda_2 ) \\
&= q^2_{\mathfrak{s}}( a , b)
\end{align*}
where the second equality follows since $\lambda_1, \lambda_2$ generate a primitive sublattice of $\Lambda$. So we have proven the result for arbitrary $a,b \in H^1(X ; \mathbb{Z})$. 

The case $b_+(X) = 2$ is proven in an entirely similar manner.

\end{proof}

\begin{proof}[Proof of Theorem \ref{thm:almosttop}]
For a compact, oriented, smooth $4$-manfold $X$ with spin structure $\mathfrak{s}$, let us write $q_{X,\mathfrak{s}}^k$ instead of $q_{\mathfrak{s}}^k$ in order to emphasize the dependence on $X$. Set $X' = X \# S^2 \times S^2$. Since $S^2 \times S^2$ is simply-connected it has a unique spin structure $\mathfrak{s}_{S^2 \times S^2}$ and there is a bijection $S_{X} \to S_{X'}$ between spin structures on $X$ and $X'$ which sends a spin structure $\mathfrak{s}$ on $X$ to the spin structure $\mathfrak{s}' = \mathfrak{s} \# \mathfrak{s}_{S^2 \times S^2}$. We also have an isomorphism $H^1(X ; \mathbb{Z}) \cong H^1(X' ; \mathbb{Z})$ which may be described concretely as follows. Regard $X'$ as the result of removing open balls $B_1 \subset X$, $B_2 \subset S^2 \times S^2$ and identifying their boundaries. Let $a \in H^1(X ; \mathbb{Z})$. Choose a Poinca\'re dual submanifold $A \subset X$ supported away from the ball $B_1$. Then $A$ can be also be regarded as a submanifold of $X'$ and it has a Poincar\'e dual class in $H^1(X' ; \mathbb{Z})$. This gives the isomorphism $H^1(X ; \mathbb{Z}) \cong H^1(X' ; \mathbb{Z})$. We claim that under the isomorphisms $S_X \cong S_{X'}$ and $H^1(X ; \mathbb{Z}) \cong H^1(X')$, we have an equality $q_{X , \mathfrak{s}}^k = q_{X' , \mathfrak{s}'}^k$ for $k=2,3$. To see this, let $a_1, \dots , a_k \in H^1(X ; \mathbb{Z})$. Let $A_1, \dots , A_k$ be Poincar\'e dual submanifolds which intersect transversally in $C$. We can choose $A_1, \dots , A_k$ supported away from the neck. Then $C$ is supported away from the neck and the index of $\mathfrak{s}|_C$ agrees with the index of $\mathfrak{s}'|_C$.

Now let $X_1, X_2$ be as in the statement of Theorem \ref{thm:almosttop}. By a theorem of Wall in the simply-connected case \cite{wa} and Gompf in general \cite{gom2}, there exists a $k \ge 0$ and a diffeomorphism $f : X_1 \# k(S^2 \times S^2) \to X_2 \# k(S^2 \times S^2)$. Set $W_i = X_i \# k(S^2 \times S^2)$ for $i=1,2$. From what we have just shown there exists bijections $\varphi_i : S_{X_i} \to S_{W_i}$ and bijections $\psi_i : H^1( X_i ; \mathbb{Z} ) \to H^1(W_i ; \mathbb{Z} )$ such that $\psi_i( q^k_{X_i , \mathfrak{s} } ) = q^k_{W_i , \varphi_i(\mathfrak{s})}$ for any $\mathfrak{s} \in S_{X_i}$. The diffeomorphism $f : W_1 \to W_2$ induces bijections $\varphi_f : S_{W_1} \to S_{W_2}$ and $\psi_f : H^1(W_1 ; \mathbb{Z}) \to H^1(W_2 ; \mathbb{Z})$ such that $\psi_f( q^k_{W_1 , \mathfrak{s}} ) = q^k_{W_2 , \varphi_f(\mathfrak{s})}$ for any $\mathfrak{s} \in S_{W_1}$. Combining all of these isomorphisms yields isomorphisms $\varphi : S_{X_1} \to S_{X_2}$ and $\psi : H^1(X_1 ; \mathbb{Z}) \to H^1(X_2 ; \mathbb{Z})$ such that $\psi( q^k_{X_1 , \mathfrak{s}}) = q^k_{X_2 , \varphi(\mathfrak{s})}$ for any $\mathfrak{s} \in S_{X_1}$. Now the result follows from Theorem \ref{thm:thm1}.
\end{proof}

\begin{proof}[Proof of Theorem \ref{thm:thm2}] 
The case that $b_+(X) = 3$ is immediate since $SW_{X,\mathfrak{s}}(1)$ can be expressed in terms of the second Segre class of the index bundle $D$, which by Equation (\ref{equ:seg2}) does not depend on the choice of spin structure.

Assume now that $b_+(X) = 2$. Then for any $a,b,c \in H^1(X ; \mathbb{Z})$ and any $A \in H^1(X ; \mathbb{Z}_2)$, Theorem \ref{thm:thm1} gives
\[
\langle SW_{X,A \otimes \mathfrak{s}}(1) , a \smallsmile b \smallsmile c \rangle + \langle SW_{X,\mathfrak{s}}(1) , a \smallsmile b \smallsmile c \rangle = q^3_{A \otimes \mathfrak{s}}(a , b , c) + q^3_{\mathfrak{s}}(a , b , c).
\]
Let $D \subset X$ be Poincar\'e dual to $a \smallsmile b \smallsmile c$. Then $D$ is a union of finitely many embedded circles in $X$. Recall that the circle $S^1$ has two spin structures $\mathfrak{s}_0, \mathfrak{s}_1$ with the property that $\mathfrak{s}_0$ extends to a spin structure on the disc $D^2$ while $\mathfrak{s}_1$ does not. Furthermore the mod $2$ index of $\mathfrak{s}_0$ is $0$ and the mod $2$ index of $\mathfrak{s}_1$ is $1$. Now write $D = D_1 \cup \cdots \cup D_m$, where $D_1, \dots , D_m$ are the connected components of $D$. Then $q_{A \otimes \mathfrak{s}}(a , b , c) + q_{\mathfrak{s}}(a , b , c)$ is equal to the mod $2$ count of the number of components $D_i$ for which $\mathfrak{s}|_{D_i}$ and $(A \otimes \mathfrak{s})|_{D_i}$ are different spin structures. Equivalently, this is the number of $i$ mod $2$ for which $A|_{D_i}$ is non-trivial, which in turn equals $\langle [X] , A \smallsmile a \smallsmile b \smallsmile c \rangle$.

Lastly, suppose that $b_+(X) = 1$. Then for any $a,b \in H^1(X ; \mathbb{Z})$ and any $A,B \in H^1(X ; \mathbb{Z}_2)$, Theorem \ref{thm:thm1} gives
\begin{align*}
& \langle SW_{X , A \otimes B \otimes \mathfrak{s}}(1) , a \smallsmile b \rangle + \langle SW_{X , A \otimes \mathfrak{s}}(1) , a \smallsmile b \rangle + \langle SW_{X , B \otimes \mathfrak{s}}(1) , a \smallsmile b \rangle + \langle SW_{X , \mathfrak{s}}(1) , a \smallsmile b \rangle \\
& \quad \quad \quad  =  q^2_{X , A \otimes B \otimes \mathfrak{s}}(a , b) + q^2_{A \otimes \mathfrak{s}}(a , b) + q^2_{B \otimes \mathfrak{s}}(a , b) + q^2_{\mathfrak{s}}(a , b).
\end{align*}

Now let $C$ be Poincar\'e dual to $a \smallsmile b$. Then according to \cite[Theorem 2]{at2}, the mod $2$ index on a compact oriented surface is a quadratic function on the set of spin structures, whose assoiciated bilinear form is the intersection form on $H^1(C ; \mathbb{Z}_2)$. From this it follows that
\begin{align*}
q^2_{X , A \otimes B \otimes \mathfrak{s}}(a , b) + q^2_{A \otimes \mathfrak{s}}(a , b) + q^2_{B \otimes \mathfrak{s}}(a , b) + q^2_{\mathfrak{s}}(a , b) &= \langle [C] , A|_C \smallsmile B|_C \rangle \\
&= \langle [X] , A \smallsmile B \smallsmile a \smallsmile b \rangle.
\end{align*}

\end{proof}

\section{Seiberg--Witten invariants of spin families}\label{sec:swfam}

By adapting the arguments of Section \ref{sec:swspin}, we will obtain a general formula for the $Pin(2)$-equivariant Seiberg--Witten invariants of spin families.

Let $B$ be a compact manifold with $\mathbb{Z}_2$-action defined by an involution $j : B \to B$ and $f : S^{V,U} \to S^{V',U'}$ a $Pin(2)$-monopole map. Since $f$ might not admit a $Pin(2)$-equivariant chamber, we replace $B$ by $\widehat{B} = S(H^+)$, the unit sphere bundle of $H^+$ over $B$ and let $\pi : \widehat{B} \to B$. Consider the pullback $\widehat{f}$ of $f$ to $\widehat{B}$. Then $\widehat{f}$ admits a tautological chamber $\phi^{taut} : \widehat{B} \to \pi^*(H^+)$ which is simply given by the inclusion $\widehat{B} = S(H^+) \to \pi^*(H^+)$. Using $\phi^{taut}$ we get a $Pin(2)$-equivariant Seiberg--Witten invariant
\[
SW_{\widehat{f}}^{Pin(2)} : H^*_{Pin(2)}(pt) \to H^{*-2d+b_+ + 1}_{\mathbb{Z}_2}(\widehat{B})
\]
where we have written $SW_{\widehat{f}}^{Pin(2)}$ in place of $SW_{\widehat{f}}^{Pin(2),\phi^{taut}}$ to simplify notation. If $\phi : B \to S(H^+)$ is a $Pin(2)$-equivariant chamber for $f$, then we clearly have the relation
\[
SW_f^{Pin(2),\phi}(\theta) = \phi^*( SW^{Pin(2)}_{\widehat{f}}(\theta))
\]
and hence it suffices to compute $SW^{Pin(2)}_{\widehat{f}}$.

We will make the following assumptions about the $\mathbb{Z}_2$-action on $B$ which are satisfied for the Seiberg--Witten monopole map of a spin family, provided that the monodromy of the family acts trivially on $H^1(X ; \mathbb{Z})$. Namely,
\begin{itemize}
\item[(1)]{$j$ does not act freely on $B$.}
\item[(2)]{The map $u : H^*_{\mathbb{Z}_2}(B) \to H^{*+1}_{\mathbb{Z}_2}(B)$ is injective.}
\item[(3)]{Over each fixed point $b \in B$ of $j : B \to B$, the involution $j : H^+_b \to H^+_b$ acts as $-1$. Hence $j$ acts freely on $S(H^+)$.}
\end{itemize}

One motivation for assumption (2) is given by the following lemma:

\begin{lemma}\label{lem:inj}
Let $E_1,E_2 \to B$ be $\mathbb{Z}_2$-vector bundles over $B$, $\pi_i : S(E_i) \to B$ the unit sphere bundles. Suppose that $\iota : E_1 \to E_2$ is a $\mathbb{Z}_2$-equivariant inclusion and that $E_2/\iota E_1 \cong \mathbb{R}_{-1}^k$, where $\mathbb{R}_{-1}$ denotes the trivial real line bundle where $\mathbb{Z}_2$ acts as multiplication by $-1$. If $u : H^*_{\mathbb{Z}_2}(B) \to H^{*+1}_{\mathbb{Z}_2}(B)$ is injective, then $\iota_* : H^*_{\mathbb{Z}_2}( S(E_1) ) \to H^{*+k}_{\mathbb{Z}_2}( (S(E_2) )$ is also injective.
\end{lemma}
\begin{proof}
The Gysin sequences for $S(E_1)$ and $S(E_2)$ are related through the following commutative diagram whose rows are exact:
\[
\xymatrix{
\cdots \ar[r] & H^{*-r_1}_{\mathbb{Z}_2}(B) \ar[r]^-{ e(E_1) } \ar@{=}[d] & H^*_{\mathbb{Z}_2}(B) \ar[r]^-{\pi_1^*} \ar[d]^{u^k} & H^*_{\mathbb{Z}_2}( S(E_1) ) \ar[r]^-{(\pi_1)_*} \ar[d]^{\iota_*} & H^{*+1-r_1}_{\mathbb{Z}_2}(B) \ar[r] \ar@{=}[d] & \cdots \\
\cdots \ar[r] & H^{*-r_1}_{\mathbb{Z}_2}(B) \ar[r]^-{e(E_2)} & H^{*+k}_{\mathbb{Z}_2}(B) \ar[r]^-{\pi_2^*} & H^{*+k}_{\mathbb{Z}_2}( S(E_2)) \ar[r]^-{(\pi_2)_*} &
H^{*+1-r_1}_{\mathbb{Z}_2}(B) \ar[r] & \cdots
}
\]
where $r_1,r_2=r_1+k$ are the ranks of $E_1, E_2$. The result follows by injectivity of $u^k : H^*_{\mathbb{Z}_2}(B) \to H^{*+k}_{\mathbb{Z}_2}(B)$ and a diagram chase.
\end{proof}

Now we are ready to repeat the argument of Theorem \ref{thm:swpin1}. Given a $Pin(2)$-equivariant monopole map $f : S^{V,U} \to S^{V',U'}$ over a base $B$, we compute $SW^{Pin(2)}_{\widehat{g}}$ in two different ways, where $g =f \wedge f_{22m(S^2 \times S^2)}$ and $f_{22m(S^2 \times S^2)}$ is the Seiberg--Witten $Pin(2)$-monopole map for $22m(S^2 \times S^2)$, for some sufficiently large $m$. First of all, let $\mathbb{R}_{-1} \to B$ denote the trivial real line bundle where $\mathbb{Z}_2$ acts as multiplication by $-1$. Consider the inclusions
\[
\iota_1 : S(H^+) \to S(H^+ \oplus \mathbb{R}^{22m}_{-1}), \quad \iota_2 : S(\mathbb{R}^{3m}_{-1}) \to S(H^+ \oplus \mathbb{R}^{22m}_{-1})
\]
induced by the inclusions $H^+ \to H^+ \oplus \mathbb{R}^{22m}_{-1}$ and $\mathbb{R}^{3m}_{-1} \to (H^+ \oplus \mathbb{R}^{19m}_{-1}) \oplus \mathbb{R}^{3m}_{-1} \cong H^+ \oplus \mathbb{R}^{22m}_{-1}$ of $\mathbb{Z}_2$-vector bundles. Theorem \ref{thm:product2} gives
\begin{equation}\label{equ:ghat1}
SW^{Pin(2)}_{\widehat{g}}(\theta) = (\iota_1)_*( SW^{Pin(2)}_{\widehat{f}}(\theta)).
\end{equation}
Now let $f_{mK3}$, $f_{m\overline{K3}}$ be the $Pin(2)$ Seiberg--Witten monopole maps for $mK3$ and $m\overline{K3}$. We can assume that $f_{22m(S^2 \times S^2)} = f_{m\overline{K3}} \wedge f_{mK3}$. We write $g = f \wedge f_{22m(S^2 \times S^2)}$ in the form $g = (f \wedge f_{m\overline{K3}}) \wedge f_{mK3}$ and apply Theorem \ref{thm:product2}. First of all, note that for any $\theta$, $SW^{Pin(2)}_{\widehat{f \wedge f}_{m\overline{K3}}}(\theta) \in H^*_{\mathbb{Z}_2}( S(H^+ \oplus \mathbb{R}^{19m}_{-1}))$ has degree
\begin{align*}
deg(\theta) + \frac{\sigma(X)}{4} + 4m + 19m + b_+ + 1 & \ge  \frac{\sigma(X)}{4} + 4m + 19m + b_+ + 1 \\
& > dim(B) + 19m + b_+ - 1
\end{align*}
prodvided $m$ satisfies $m \ge dim(B)/4 - \sigma(X)/16$. Our assumption on $j : H^+ \to H^+$ ensures that $\mathbb{Z}_2$ acts freely on $S(H^+ \oplus \mathbb{R}^{19m}_{-1})$ and hence the quotient $S(H^+ \oplus \mathbb{R}^{19m}_{-1})/\mathbb{Z}_2$ is a manifold of dimension $dim(B) + 19m + b_+ - 1$. Hence the degree of $SW^{Pin(2)}_{\widehat{f \wedge f}_{m\overline{K3}}}(\theta)$ is greater than the highest degree in which $H^*_{\mathbb{Z}_2}( S(H^+ \oplus \mathbb{R}^{19m}_{-1}))$ is non-zero. So $SW^{Pin(2)}_{\widehat{f \wedge f}_{m\overline{K3}}}(\theta) = 0$ for all $\theta$. Together with Theorem \ref{thm:product2}, this implies that
\[
SW^{G}_{\widehat{g}}(q_2^j) = (\iota_2)_*( SW^{Pin(2)}_{mK3}(  \psi_2( e_G(D_1)^{-1} q_2^j )))
\]
where $D_1 = D - \mathbb{C}^{2m}$ and $D = V - V'$. Expanding the Euler class $e_G(D_1)$, collecting $\mu^0$-terms and simplifying, exactly as in the proof of Theorem \ref{thm:swpin1}, we obtain
\begin{equation}\label{equ:ghat2}
SW^{Pin(2)}_{\widehat{g}}(q_2^j) = (\iota_2)_* ( u^{3m-3} s_{2(j+1-d/2) , \mathbb{Z}_2}(D) ).
\end{equation}
Equating (\ref{equ:ghat1}) and (\ref{equ:ghat2}) yields
\begin{equation}\label{equ:swf1}
(\iota_1)_*( SW^{Pin(2)}_{\widehat{f}}(q^j)) =  (\iota_2)_* ( u^{3m-3} s_{2(j+1-d/2), \mathbb{Z}_2}(D) ).
\end{equation}

By Lemma \ref{lem:inj}, and the assumption that $u : H^*_{\mathbb{Z}_2}(B) \to H^{*+1}_{\mathbb{Z}_2}(B)$ is injective, it follows that $(\iota_1)_*$ is injective and thus Equation (\ref{equ:swf1}) completely determines $SW^{Pin(2)}_{\widehat{f}}$.

\begin{lemma}\label{lem:pullback}
If $u : H^*_{\mathbb{Z}_2}(B) \to H^{*+1}_{\mathbb{Z}_2}(B)$ is injective then for any $\theta \in H^*_{Pin(2)}(pt)$, $SW^{Pin(2)}_{\widehat{f}}(\theta)$ is in the image of the pullback $\pi^* : H^*_{\mathbb{Z}_2}(B) \to H^*_{\mathbb{Z}_2}( S(H^+) )$, where $\pi : S(H^+) \to B$ is the projection. In particular, the invariants $SW_{f}^{Pin(2),\phi}(\theta)$ do not depend on the choice of a chamber.
\end{lemma}
\begin{proof}
To simplify notation, set $\alpha = SW^{Pin(2)}_{\widehat{f}}(\theta) \in H^*_{\mathbb{Z}_2}( S(H^+) )$. Since $u^3 = 0$ in $H^*_{Pin(2)}(pt)$, we have $u^3 \alpha = SW^{Pin(2)}_{\widehat{f}}(u^3 \theta) = 0$. Consider the Gysin sequence
\[
\cdots \to H^*_{\mathbb{Z}_2}( B ) \buildrel \pi^* \over \longrightarrow H^*_{\mathbb{Z}_2}(S(H^+)) \buildrel \pi_* \over \longrightarrow H^{*-(b_+-1)}_{\mathbb{Z}_2}(B) \to \cdots
\]
Now $u^3 \alpha = 0$ implies $u^4 \pi_*(\alpha) = 0$. But $u : H^*_{\mathbb{Z}_2}(B) \to H^{*+1}_{\mathbb{Z}_2}(B)$ is injective, so $\pi_*(\alpha) = 0$. Hence by exactness of the Gysin sequence, $\alpha = \pi^*(\beta)$ for some $\beta \in H^*_{\mathbb{Z}_2}(B)$. Now suppose that $\phi : B \to S(H^+)$ is a chamber. Then $SW^{Pin(2),\phi}_{f}( \theta ) = \phi^*( \alpha ) = \phi^* \pi^*(\beta) = \beta$, which does not depend on the choice of $\phi$ (note that in general, $\beta$ is only unique up to multiples of $e_{\mathbb{Z}_2}(H^+)$, but if a chamber exists then $e_{\mathbb{Z}_2}(H^+) = 0$).
\end{proof}

Let $E_1 \to E_2$ be an inclusion of $\mathbb{Z}_2$-vector bundles on $B$ and $\iota : S(E_1) \to S(E_2)$ the induced map of sphere bundles. Then for any $\beta \in H^*_{\mathbb{Z}_2}(B)$, we have
\begin{equation}\label{equ:pf}
\iota_*( \pi_1^*(\beta) ) = \pi_2^*( e_{\mathbb{Z}_2}(E_2/E_1) \beta),
\end{equation}
where $\pi_1,\pi_2$ are the projections $\pi_i : S(E_i) \to B$. By Lemma \ref{lem:pullback}, we have that $SW^{Pin(2)}_{\widehat{f}}(q^j) = \pi^*( \alpha_j)$ for some $\alpha_j \in H^*_{\mathbb{Z}_2}(B)/\langle e_{\mathbb{Z}_2}(H^+) \rangle$, where $\pi$ is the projection $\pi : S(H^+) \to B$. Applying (\ref{equ:pf}) to Equation (\ref{equ:swf1}) gives:
\[
u^{22m} \alpha_j = e_{\mathbb{Z}_2}(H^+) u^{22m-3} s_{2(j+1 - d/2),\mathbb{Z}_2}(D),
\]
which holds as an equality in $H^*_{\mathbb{Z}_2}(B)/ (u^{22m} e_{\mathbb{Z}_2}(H^+) )$. Cancelling a common factor of $u^{22m-3}$ gives
\[
u^3 \alpha_j = e_{\mathbb{Z}_2}(H^+) s_{2(j+1-d/2),\mathbb{Z}_2}(D)
\]
in $H^*_{\mathbb{Z}_2}(B)/( u^3 e_{\mathbb{Z}_2}(H^+) )$. Since $u : H^*_{\mathbb{Z}_2}(B) \to H^*_{\mathbb{Z}_2}(B)$ is injective, the right hand side must be a multiple of $u^3$. Let $u^{-3} : u^3 H^*_{\mathbb{Z}_2}(B) \to H^{*-3}_{\mathbb{Z}_2}(B)$ denote the inverse of $u^3$ on its image. Then
\[
\alpha_j = u^{-3} e_{\mathbb{Z}_2}(H^+) s_{2(j+1-d/2),\mathbb{Z}_2}(D).
\]
Pulling back to $H^*_{\mathbb{Z}_2}( S(H^+) )$, we obtain
\[
SW^{Pin(2)}_{\widehat{f}}( q^j ) = \pi^*( u^{-3} e_{\mathbb{Z}_2}(H^+) s_{2(j+1-d/2),\mathbb{Z}_2}(D) ).
\]

We have proven the following:

\begin{theorem}\label{thm:swpingen}
Suppose that $u : H^*_{\mathbb{Z}_2}(B) \to H^{*+1}_{\mathbb{Z}_2}(B)$ is injective and suppose that $j : H^+ \to H^+$ acts as $-1$ over the fixed points of $j : B \to B$. Then for each $j \ge 0$, $e_{\mathbb{Z}_2}(H^+) s_{2(j+1 -d/2),\mathbb{Z}_2}(D)$ is a multiple of $u^3$ and
\[
SW^{Pin(2)}_{\widehat{f}}( q^j ) = \pi^*( u^{-3} e_{\mathbb{Z}_2}(H^+) s_{2(j+1-d/2),\mathbb{Z}_2}(D) ).
\]
\end{theorem}

Now suppose that $\pi_E : E \to B_0$ is a smooth spin family of $4$-manifolds. This means that $E$ is a fibre bundle with fibres given by a compact, oriented smooth $4$-manifold $X$, with transition functions given by orientation preserving diffeomorphisms of $X$ and in addition we are given a spin-structure $\mathfrak{s}_E$ on the vertical tangent bundle $T(E/B_0) = Ker( (\pi_E)_* : TE \to TB_0 )$. If $b_1(X) > 0$ then we need to assume also that there exists a section $s : B_0 \to E$. In this case we get a families Seiberg--Witten monopole map $f : S^{V,U} \to S^{V',U'}$ over $B$, where $B = Pic^{\mathfrak{s}_E}(E/B_0)$ is the moduli space of gauge equivalence classes of spin$^c$-connections on the fibres of $E$. See \cite[Example 2.4]{bk} for details of the construction, including an explanation of why the section $s$ is needed. Thus $B \to B_0$ is a fibre bundle whose fibre over $b \in B_0$ is $Pic^{\mathfrak{s}_E|_{X_b}}(X_b)$, where $X_b = \pi_E^{-1}(b)$ is the fibre of $E$ over $b$. This is a torus bundle over $B_0$. Moreover, since the family has a spin structure, there is a section $s : B_0 \to B$ given by the spin connection. Thus $B$ is completely determined by the degree $1$ monodromy representation $\pi_1(B_0) \to Aut( H^1(X ; \mathbb{Z}) )$.

The involution $j : B \to B$ acts as the identity on $B_0$ and as $-1$ on the fibres of $B \to B_0$. Assuming that the monodromy of the family $E \to B_0$ acts trivially on $H^1(X ; \mathbb{Z})$, then $B \cong B_0 \times Pic^{\mathfrak{s}_E}(X)$ and it follows easily that $H^*_{\mathbb{Z}_2}(B) \cong H^*(B)[u]$ (non-canonically, as $\mathbb{Z}_2[u]$-modules). In particular, $u : H^*_{\mathbb{Z}_2}(B) \to H^{*+1}_{\mathbb{Z}_2}(B)$ is injective. To each spin$^c$-connection, we have a corresponding Dirac operator. Thus $B$ is the parameter space for a family of Dirac operators. The virtual bundle $D = V - V'$ is the families index of this family. The bundle $H^+ \to B$ is the pullback to $B$ of the bundle $H^+(X) \to B_0$ whose fibre over a point $b \in B_0$ is the space $H^+(X_b)$ of self-dual harmonic $2$-forms on $X_b$ (with respect to a given family of metrics). The involution $j : H^+(X_b) \to H^+(X_b)$ acts as a combination of $j : B \to B$ on the base, together with multiplication by $-1$ on the fibres. Thus $j$ acts as multiplication by $-1$ over the fixed points of $j : B \to B$. Our assumptions (1)-(3) are satisfied and hence Theorem \ref{thm:swpingen} applies. To compute the equivariant Euler class $e_{\mathbb{Z}_2}(H^+)$, it is best to think of $H^+$ as being the tensor product $H^+(X) \otimes \mathbb{R}_{-1}$. Then the splitting principle gives
\begin{equation}\label{equ:eh+}
e_{\mathbb{Z}_2}(H^+) = \sum_{j=0}^{b_+(X)}u^j  w_{b_+(X)-j}( H^+(X) ) \in H^*(B)[u].
\end{equation}
Since $e_{\mathbb{Z}_2}(H^+)$ is a monic polynomial in $u$, multiplication $e_{\mathbb{Z}_2}(H^+) : H^*(B)[u] \to H^{*+b_+(X)}(B)[u]$ is injective. Hence, by the Gysin sequence for $S(H^+) \to B$, we have an isomorphism
\[
H^*_{\mathbb{Z}_2}( S(H^+) ) \cong H^*_{\mathbb{Z}_2}(B)/(e_{\mathbb{Z}_2}(H^+)).
\]
Write $SW^{Pin(2)}_{E,\mathfrak{s}_E}$ for the $Pin(2)$-equivariant Seiberg--Witten invariants of $\widehat{f}$. This is a map
\[
SW^{Pin(2)}_{E,\mathfrak{s}_E} : H^*_{Pin(2)}(pt) \to H^{*-2d+b_+(X)+1}_{\mathbb{Z}_2}( S(H^+) ).
\]

Applying Theorem \ref{thm:swpingen} gives

\begin{theorem}\label{thm:spinfam1}
Let $E \to B_0$ be a spin family. If $b_1(X) > 0$, assume there exists a section $s : B_0 \to E$ and that the monodromy of the family acts trivially on $H^1(X ; \mathbb{Z})$. Then for any $k \ge 1+\sigma(X)/16$, we have
\begin{equation}\label{equ:vanish}
(w_{b_+}(H^+(X)) + u w_{b_+-1}(H^+(X)) + u^2 w_{b_+-2}(X) ) s_{2k,\mathbb{Z}_2}(D) = 0 \; ({\rm mod} \; u^3)
\end{equation}
In particular, if $\sigma(X) < 0$, then $w_l(H^+(X)) = 0$ for $b_+(X) \le l \le b_+(X)-2$. Furthermore, we have
\[
SW^{Pin(2)}_{E,\mathfrak{s}_E}(q^j) = \sum_{l=0}^{b_+(X)} u^{l-3} w_{b_+(X)-l}(H^+(X)) s_{2(j+1+\sigma(X)/16) , \mathbb{Z}_2}(D).
\]
When $l-3 < 0$, the right hand side of this equality should be interpreted in the same manner as was done in Theorem \ref{thm:swpin1}.
\end{theorem}
\begin{proof}
According to Theorem \ref{thm:swpingen}, $e_{\mathbb{Z}_2}(H^+(X))s_{2(j+1+\sigma(X)/16), \mathbb{Z}_2}(D)$ is a multiple of $u^3$ for any $j \ge 0$. Using Equation (\ref{equ:eh+}), this gives
\[
(w_{b_+}(H^+(X)) + u w_{b_+-1}(H^+(X)) + u^2 w_{b_+-2}(X) ) s_{2(j+1+\sigma(X)/16),\mathbb{Z}_2}(D) = 0 \; ({\rm mod} \; u^3)
\]
for all $j \ge 0$. In particular, if $\sigma(X) < 0$, then taking $j = -1 -\sigma(X)/16$, gives $w_l(H^+(X)) = 0$ for $b_+(X) - 2 \le l \le b_+(X)$. Futhermore, Theorem \ref{thm:swpingen} and Equation (\ref{equ:eh+}) give
\[
SW^{Pin(2)}_{E,\mathfrak{s}_E}(q^j) = \sum_{l=0}^{b_+(X)} u^{l-3} w_{b_+(X)-l}(H^+(X)) s_{2(j+1+\sigma(X)/16),\mathbb{Z}_2}(D).
\]
\end{proof}

\begin{remark}
The condition that $w_l( H^+(X) ) = 0$ for $b_+(X)-2 \le l \le b_+(X)$ when $\sigma(X) < 0$ was also shown in \cite[Theorem 1.2]{bar}.
\end{remark}

The vanishing condition (\ref{equ:vanish}) is somewhat difficult to use because of the presence of the $\mathbb{Z}_2$-equivariant Segre classes which could possibly have $u$ and $u^2$-terms. However, by looking at the lowest order term in $u$ in (\ref{equ:vanish}), we obtain:

\begin{corollary}
Let $E \to B_0$ be a spin family. If $b_1(X) > 0$, assume there exists a section $s : B_0 \to E$ and that the monodromy of the family acts trivially on $H^1(X ; \mathbb{Z})$. Let $l \ge 0$ be the largest non-negative integer for which $w_l(H^+(X)) \neq 0$. If $l \ge b_+(X)-2$, then 
\[
w_l(H^+(X))s_{2(j+1+\sigma(X)/16)}(D) = 0.
\]
for all $j \ge 0$.
\end{corollary}
\begin{proof}
If $w_k(H^+(X)) = 0$ for $k > l$ and $l \ge b_+(X)-2$, then the left hand side of (\ref{equ:vanish}) has the form $u^{b_+(X)-l} w_l(H^+(X))s_{2(j+1+\sigma(X)/16)}(D) + \cdots $ where $\cdots$ denotes terms of higher order in $u$.
\end{proof}

Consider a spin family of $4$-manifolds $E \to B_0$ with $Pin(2)$-equivariant monopole map $f$. Restricting to the circle subgroup $S^1 \subset Pin(2)$, and choosing a chamber $\phi : B \to S(H^+)$, we may consider the $S^1$-equivariant Seiberg--Witten invariants of $f$ (taken as usual with $\mathbb{Z}_2$-coefficients)
\[
SW^\phi_{E , \mathfrak{s}_E} : H^*_{S^1}(pt) \to H^{*-2d+b_+(X)+1}(B).
\]
The cohomology classes $SW_{E , \mathfrak{s}_E}(x^m) \in H^{2m-2d+b_+(X)+1}(B)$ are the (mod $2$) {\em families Seiberg--Witten invariants} of the spin$^c$-family $(E , \mathfrak{s}_E)$, as defined for instance in \cite{bk}. Using Theorem \ref{thm:spinfam1}, we obtain 

\begin{theorem}\label{thm:spinfam2}
Let $E \to B_0$ be a spin family. If $b_1(X) > 0$, assume there exists a section $s : B_0 \to E$ and that the monodromy of the family acts trivially on $H^1(X ; \mathbb{Z})$. Then for any chamber $\phi$, we have
\[
SW^\phi_{E,\mathfrak{s}_E}( x^{2j} ) = \sum_{l=1}^{3} \left. \left( u^{l-3} w_{b_+(X)-l}(H^+(X)) s_{2(j+1+\sigma(X)/16),\mathbb{Z}_2}(D) \right) \right|_{u=0}.
\]
In particular if $\sigma(X) < 0$ or if $b_1(X)=0$, then 
\[
SW^\phi_{E,\mathfrak{s}_E}( x^{2j} ) = w_{b_+(X)-3}(H^+(X)) s_{2(j+1+\sigma(X)/16)}(D).
\]
\end{theorem}
\begin{proof}
Let $\phi : B \to S(H^+(X))$ be a chamber. This defines a $\mathbb{Z}_2$-equivariant map $S^0 \times B = S(\mathbb{R}\phi) \to S(H^+(X))$, inducing a map
\[
\phi^* : H^*_{\mathbb{Z}_2}( S(H^+(X)) ) \to H^*_{\mathbb{Z}_2}( S^0 \times B ) \cong H^*(B).
\]
By the same reasoning as in Lemma \ref{lem:u=0}, we have that $SW_{E,\mathfrak{s}_E}( x^{2j} ) = \phi^* (SW_{E,\mathfrak{s}_E}^{Pin(2)}(q^j))$, so it remains to describe the map $\phi^*$. The existence of $\phi$ implies that $w_{b_+(X)}( H^+(X) ) = 0$ and hence $e_{\mathbb{Z}_2}(H^+(X))$ is divisible by $u$, by Equation (\ref{equ:eh+}). Hence, the forgetful map $H^*_{\mathbb{Z}_2}(B) \to H^*(B)$ factors through $H^*_{\mathbb{Z}_2}( S(H^+(X))) \to H^*(B)$, and it is clear that this gives the map $\phi^*$, because $\phi^* \circ \pi^* : H^*_{\mathbb{Z}_2}(B) \to H^*(B)$ is the forgetful map, where $\pi : S(H^+(X)) \to B$ is the projection. It then follows that $\phi^* (SW_{E,\mathfrak{s}_E}^{Pin(2)}(q^j))$ is given by extracting the $u^0$-term, giving
\[
SW^\phi_{E,\mathfrak{s}_E}( x^{2j} ) = \sum_{l=1}^{3} \left. \left( u^{l-3} w_{b_+(X)-l}(H^+(X)) s_{2(j+1+\sigma(X)/16),\mathbb{Z}_2}(D) \right) \right|_{u=0}.
\]
If $\sigma(X) < 0$, then Theorem \ref{thm:spinfam1} also gives $w_l(H^+(X)) = 0$ for $l > b_+(X)-3$, so the formula simplifies to $SW^\phi_{E,\mathfrak{s}_E}( x^{2j} ) = w_{b_+(X)-3}(H^+(X)) s_{2(j+1+\sigma(X)/16)}(D)$. If $b_1(X) = 0$, then $j$ acts trivially on $B$ and $D$ is a quaternionic virtual bundle. Using the inclusion $Pin(2) \subset Sp(1)$, it follows easily that
\[
e_{Pin(2)}( D )^{-1} = q^{d/2} + q^{d/2 - 1} s_2(D) + \cdots + s_{d/2}(D).
\]
Hence in this case the equivariant Segre classes of $D$ are just equal to the usual Segre classes. So it follows that  $SW^\phi_{E,\mathfrak{s}_E}( x^{2j} ) = w_{b_+(X)-3}(H^+(X)) s_{2(j+1+\sigma(X)/16)}(D)$.
\end{proof}



\bibliographystyle{amsplain}

\end{document}